\documentclass{article}

\usepackage[a4paper]{geometry}

\usepackage{tikz}
\usetikzlibrary{matrix,arrows}

\usepackage{latexsym}
\usepackage{oldgerm}
\usepackage{curves}
\usepackage[all]{xy}
\usepackage{pb-diagram}
\usepackage{pb-xy}
\usepackage{bm}

\usepackage{amsmath}
\usepackage{amsthm}
\usepackage{amssymb}

\usepackage[hidelinks]{hyperref}

\theoremstyle{plain}
  \newtheorem{theorem}{Theorem}[section]
  \newtheorem{corollary}[theorem]{Corollary}
  
  \newtheorem{lemma}[theorem]{Lemma}
  
  \newtheorem{proposition}[theorem]{Proposition}

\theoremstyle{definition}
  \newtheorem{definition}[theorem]{Definition}

  \newtheorem{ex}[theorem]{Example}

  \newtheorem{remark}[theorem]{Remark}

  \newenvironment{example}{\begin{ex}}{\qed\end{ex}}

\newcommand{\category}[1]{\mathbf{#1}}

  \newcommand{\Map}{\operatorname{Map}}

  \newcommand{\Ima}{\operatorname{Im}}

  \DeclareMathOperator*{\colim}{\mathrm{colim}}

  \newcommand{\quotient}[2]{%
  \left(#1\right)
  \hspace{-4pt}\raisebox{-5pt}{$\bigg/$}\hspace{-2pt}\raisebox{-12pt}{$#2$}%
  }

\newcommand{\lset}[2]{%
\left.\left\{#1 \ \right| \ #2\right\}
}
\newcommand{\set}[2]{%
\left\{#1 \mathrel{}\middle|\mathrel{}  #2\right\}
}

  \newcommand{\rarrow}[1]{\buildrel #1 \over \longrightarrow}

  \newcommand{\ev}{{\operatorname{\mathrm{ev}}}}

  \newcommand{\transpose}[1]{\raisebox{1ex}{$\scriptstyle t$}\kern-0.2ex #1}
  \newcommand{\rank}{\operatorname{\mathrm{rank}}}

  \newcommand{\Int}{{\operatorname{\mathrm{Int}}}}

  \newcommand{\R}{\mathbb{R}}

  \newcommand{\Z}{\mathbb{Z}}

  \newcommand{\hodim}{\operatorname{\mathrm{hodim}}} 

  \newcommand{\sign}{\operatorname{\mathrm{sign}}}

  \newcommand{\Path}{\operatorname{\mathrm{Path}}}

  \newcommand{\Br}{\operatorname{\mathrm{Br}}} 
  \newcommand{\PBr}{\operatorname{\mathrm{PBr}}} 
  
  \newcommand{\String}{\mathrm{\String}}

  \newcommand{\St}{\mathrm{St}}

  \newcommand{\Conf}{\mathrm{Conf}}

\newcommand{\sk}{\operatorname{sk}}

\newcommand{\Sd}{\operatorname{\mathrm{Sd}}}

\newcommand{\fixhyperref}{%
\ifnum 42146=\euc"A4A2 \AtBeginDvi{}\else
\AtBeginDvi{}\fi}

\newcommand{\comp}{\mathrm{comp}}
\newcommand{\braid}{\mathrm{braid}}
\newcommand{\Abrams}{\mathrm{Abrams}}

\title{\bfseries Totally Normal Cellular Stratified Spaces and
Applications to the Configuration Space of Graphs}
\author{Mizuki Furuse, Takashi Mukouyama, and Dai Tamaki}

\begin{document}

\maketitle

\begin{center}
 Dedicated to Professor Yuli Rudyak on the occasion of his 65th
 birthday. 
\end{center}

\begin{abstract}
 The notion of regular cell complexes plays a central role in
 topological combinatorics because of its close relationship with
 posets. A generalization, called totally normal
 cellular stratified spaces, was introduced in 
 \cite{1009.1851v5,1106.3772} by relaxing two conditions;
 face posets are replaced by acyclic categories and cells with
 incomplete boundaries are allowed. 
 The aim of this article is to demonstrate the usefulness of totally
 normal cellular stratified spaces by constructing a combinatorial 
 model for the configuration space of graphs.
 As an application, we obtain a simpler proof of Ghrist's theorem on the 
 homotopy dimension of the configuration space of graphs. We also make
 sample calculations of the fundamental group of ordered and unordered
 configuration spaces of two points for small graphs.
\end{abstract}

\setcounter{tocdepth}{2}
\tableofcontents

\section{Introduction}
\label{graph_intro}

Given a topological space $X$, the configuration space $\Conf_k(X)$ of
$k$ distinct ordered points in $X$ is defined by
\[
 \Conf_k(X) = X^k \setminus \Delta_k(X),
\]
where the discriminant set $\Delta_k(X)$ is given by
\[
 \Delta_k(X) = \set{(x_1,\ldots,x_k)\in X^k}{x_i=x_j \text{ for some }
 i\neq j}. 
\]
Configuration spaces have been studied by topologists when $X$ is a
manifold because of their appearance in geometry and homotopy
theory. See \cite{Fadell62,Fadell-Neuwirth62,Arnold69,May72,F.Cohen78},
for example.  
When $X$ is a manifold, we have a nice fibration of the following form
\begin{equation}
 \Conf_{k-1}(X\setminus\{x_0\}) \rarrow{} \Conf_k(X) \rarrow{p_{k,1}} X,
\label{configuration_space_fibration} 
\end{equation}
which has been an indispensable tool for studying the homotopy type of
configuration spaces of manifolds.

When $X$ is not a manifold, however, we cannot expect the map $p_{k,1}$
to be a fibration and it is much harder to study its configuration
spaces. 
It was Ghrist \cite{math.GT/9905023} who found an interpretation of
$\Conf_k(X)$ in terms of the problem of controlling automated guided
vehicles (AGVs) in a factory and initiated the study of $\Conf_k(X)$
when $X$ is a $1$-dimensional cell complex, i.e.\ a graph. 
It turns out that configuration spaces of graphs have
many interesting properties and attracted much attention. 
For example, Ghrist proved that they are $K(\pi,1)$ spaces. The
fundamental group of the unordered configuration space
$\Conf_k(X)/\Sigma_k$ of a graph $X$ is called the graph braid group of
$X$ and its relation to right-angled Artin groups has been studied by
several people \cite{0711.1160,0805.0082,0907.2730}.

Because of the failure of the projection $p_{k,1}$ in
(\ref{configuration_space_fibration}) to be a fibration, we need to find
a completely different method when $X$ is not a manifold. One of
successful and popular approach is to use Abrams' cellular model.   

\begin{definition}[Abrams Model]
 \label{Abrams_model}
 For a space $X$ equipped with a cell decomposition
 $\pi:X=\bigcup_{\lambda\in\Lambda} e_{\lambda}$, define a
 subcomplex $C_k^{\mathrm{Abrams}}(X,\pi)$ of $X^k$ by
 \begin{eqnarray*}
  C_k^{\mathrm{Abrams}}(X,\pi) & = &
   \bigcup_{\overline{e_{\lambda_i}}\cap
   \overline{e_{\lambda_j}}=\emptyset\ (i\neq j)} e_{\lambda_1}\times 
 \cdots \times e_{\lambda_k} \\
  & = &  \bigcup_{\overline{e_{\lambda_1}\times\cdots\times
   e_{\lambda_k}}\cap \Delta_k(X)=\emptyset} e_{\lambda_1}\times
 \cdots \times e_{\lambda_k}.  
 \end{eqnarray*}
\end{definition}

Obviously $C_k^{\mathrm{Abrams}}(X,\pi)$ is included in
$\Conf_k(X)$. Abrams proved that his model $C_k^{\mathrm{Abrams}}(X)$ is
a deformation retract of $\Conf_k(X)$ under certain conditions. 

\begin{theorem}[\cite{AbramsThesis}]
 For a $1$-dimensional finite cell complex $X$, the inclusion
 \[
  C_k^{\Abrams}(X,\pi) \hookrightarrow \Conf_k(X)
 \]
 is a homotopy equivalence as long as the cell decomposition on $X$
 satisfies the following two conditions:
 \begin{enumerate}
  \item each path connecting vertices $X$ of valency more than $2$ has
	length at least $k+1$, and 
  \item each homotopically essential path connecting a vertex to itself
	has length at least $k+1$.
 \end{enumerate}
\end{theorem}

For precise definitions of terminologies used in the above theorem, see
Abrams' thesis. 
Although Abrams' model has been used by several authors to study
configuration spaces of graphs
\cite{math.GT/9905023,math.GR/0410539,0711.1160,0805.0082,0907.2730,0908.1067}
and higher dimensional cell complexes \cite{1111.6699} successfully,
there are some difficulties in using this model because of 
the following facts:
\begin{itemize}
 \item Deformation retractions are not constructed explicitly.
 \item The action of the symmetric group $\Sigma_k$
       is not taken into account, either.
 \item The two conditions in Abrams' theorem require us to subdivide 
       $X$ finely. And taking subdivisions of $X$ makes the model
       larger.   
\end{itemize} 

Let us take a look at a couple of examples to see the effect of
subdivisions on Abrams model.

\begin{example}
 Consider the case when $X=S^1$ and $k=2$.
 It is well known that we have a $\Sigma_2$-equivariant homotopy
 equivalence $\Conf_2(S^1)\simeq_{\Sigma_2} S^1$.

 Regard $S^1$ as a graph under the minimal cell decomposition $\pi_{1}$: 
 $S^1 = e^0\cup e^1$
 
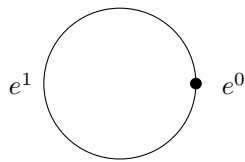
\begin{center}
  \begin{tikzpicture}
   \draw (0,0) circle (1cm);
   \draw [fill] (1,0) circle (2pt);
   \draw (-1.3,0) node {$e^1$};
   \draw (1.5,0) node {$e^0$};
  \end{tikzpicture}
 \end{center}
 Abrams model $C_2^{\Abrams}(S^1,\pi_{1})$ for this graph is the empty
 set and is not homotopy equivalent to $\Conf_2(S^1)$.

 By taking a subdivision once, we obtain a cell decomposition
 $\pi_{2}$ of $S^1$ consisting of two $1$-cells:
 $S^1 = e^0_1\cup e^0_2 \cup e^1_1\cup e^1_2$.
 \begin{center}
  \begin{tikzpicture}
  \draw (0,0) circle (1cm);
  \draw [fill] (-1,0) circle (2pt);
  \draw [fill] (1,0) circle (2pt);
  \draw (-1.3,0) node {$e^0_1$};
  \draw (1.5,0) node {$e^0_2$};
  \draw (0,-1.3) node {$e^1_1$};
  \draw (0,1.3) node {$e^1_2$};
  \end{tikzpicture}
 \end{center}
 The model
 $C_2^{\Abrams}(S^1,\pi_{2})$ is still too small. It is merely a set of
 two points $\{e^0_1\times e^0_2, e^0_2\times e^0_1\}$. 

 In general, let $\pi_{n}$ be the cell decomposition of $S^1$ as a
 cyclic graph with $n$ edges.
 
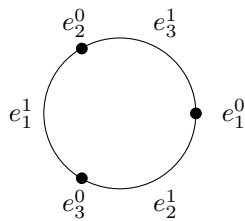
\begin{center}
 \begin{tikzpicture}
  \draw (0,0) circle (1cm);
  \draw [fill] (1,0) circle (2pt);
  \draw [fill] (-0.5,0.86) circle (2pt);
  \draw [fill] (-0.5,-0.86) circle (2pt);
  \draw (-1.3,0) node {$e^1_1$};
  \draw (0.6,-1.2) node {$e^1_2$};
  \draw (0.6,1.2) node {$e^1_3$};
  \draw (1.5,0) node {$e^0_1$};
  \draw (-0.6,1.2) node {$e^0_2$};
  \draw (-0.6,-1.2) node {$e^0_3$};
 \end{tikzpicture}
 \end{center}
 When $n\ge 3$, the assumptions in Abrams' Theorem are satisfied and we
 have a homotopy equivalence
 $C^{\Abrams}_2(S^1,\pi_{n})\simeq \Conf_2(S^1)$. When $n=3$, it is easy
 to see that $C^{\Abrams}_2(S^1,\pi_3)$ is the boundary of a hexagon and
 is homeomorphic to $S^1$. When $n>3$, however,
 $C^{\Abrams}_2(S^1,\pi_n)$ is a cell complex of dimension $2$. 
\end{example}

The following is our wish list for a combinatorial model $C_k(\Gamma)$
for the configuration space $\Conf_k(\Gamma)$ of $k$ points of a graph
$\Gamma$:   
\begin{enumerate}
 \item We would like our model to be as small as possible. It is
       desirable that its dimension coincides with the homotopy
       dimension of $\Conf_{k}(\Gamma)$.
 \item The model $C_{k}(\Gamma)$ should be a deformation retract of
       $\Conf_{k}(\Gamma)$ under a reasonable condition. Furthermore
       we would like the deformation retraction to be equivariant with
       respect to the action of the symmetric group $\Sigma_{k}$.
 \item Abrams' proof relies on Ghrist's $K(\pi,1)$ theorem and the
       Whitehead theorem. A more direct proof is desirable. 
\end{enumerate}

In this paper, we propose a new model for configuration spaces of graphs
based on the notion of totally normal cellular
stratified spaces introduced in \cite{1009.1851v5} and developed in
\cite{1106.3772,1111.4774}.
In general, given a $1$-dimensional cellular stratified space $X$,
we define an appropriate cellular subdivision on 
$X^k$ which contains $\Conf_k(X)$ as a stratified subspace. 
And we obtain an acyclic category $C(\pi_{k,X}^{\comp})$. 

\begin{theorem}[Corollary \ref{embedding_and_deformation_retraction}]
 For any $1$-dimensional finite cellular stratified space $X$ and a
 positive integer $k$, there exists a finite acyclic category
 $C(\pi_{k,X}^{\comp})$ whose classifying space $BC(\pi_{k,X}^{\comp})$
 can be embedded in $\Conf_k(X)$ as a strong $\Sigma_k$-equivariant
 deformation retract. 
\end{theorem}

The space $C^{\comp}_k(X)=BC(\pi_{k,X}^{\comp})$ is one of our
models for $\Conf_k(X)$. 
The $\Sigma_k$-equivariance of the deformation retraction in the above
theorem follows from the naturality of the construction. Our deformation
retraction is also explicitly constructed. Thus the last two
requirements in our wish list are satisfied.

The next question is how small our model is. It is easy to show that the
number of vertices controls the dimension of our model.

\begin{theorem}[Theorem \ref{dimension_of_model}]
 \label{bound_by_the_number_of_vertices}
 Let $X$ be a connected finite $1$-dimensional cellular stratified
 space. Then 
 \[
 \dim BC\left(\pi_{k,X}^{\comp}\right) \le v(X),
 \]
 where $v(X)$ is the number of $0$-cells in $X$.
\end{theorem}

Thus we obtain a smaller model if we could reduce the number of
vertices. In other words, we obtain a small model by using the minimal
cellular stratification of a given $1$-dimensional cellular stratified
space. We can reduce the dimension further by removing vertices of
valency $1$, as we will see in \S\ref{homotopy_dimension}.
As a consequence, we obtain an alternative proof of Ghrist's theorem
\cite{math.GT/9905023} on the homotopy dimension of the configuration
space of graphs. Recall that the \emph{homotopy dimension} $\hodim X$ of 
a space $X$ is defined by
\[
 \hodim X = \min_{Y\simeq X} \dim Y
\]
where $Y$ runs over all finite cell complexes that are homotopy
equivalent to $X$.

\begin{corollary}
 Let $X$ be a $1$-dimensional connected finite cellular stratified
 space. Then we have
 \[
  \hodim \Conf_k(X) \le \min\{k, v^{\mathrm{ess}}(X)\},
 \]
 where $v^{\mathrm{ess}}(X)$ is the number of \emph{essential vertices},
 i.e.\ $0$-cells that are neither of valency $1$ nor incident to exactly
 two regular $1$-cells. 
\end{corollary}

Here the valency of a $0$-cell $e^0$ in a $1$-dimensional
cellular stratified space $\Gamma$ is defined to be the cardinality of
the set
\[
 \bigcup_{e^0\subset\overline{e}} \set{b:D^0\to
 D}{\varphi\circ b = \varphi_0}, 
\]
where $\varphi:D\to \overline{e}$ is the characteristic map for a cell
$e$, $\varphi_0 : D^0\to e^0$ is the characteristic map for the
$0$-cell $e^0$, and $e$ runs over all $1$-cells containing $e^0$ in its
closure. 

For example, the numbers in the following figure indicates the valencies
of vertices in this $1$-dimensional cell complex.
 
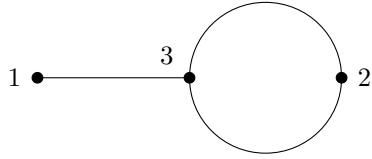
\begin{center}
  \begin{tikzpicture}
   \draw (0,0) circle (1cm);
   \draw (-3,0) -- (-1,0);
   \draw [fill] (-1,0) circle (2pt);
   \draw (-1.3,0.3) node {$3$};
   \draw [fill] (1,0) circle (2pt);
   \draw (1.3,0) node {$2$};
   \draw [fill] (-3,0) circle (2pt);
   \draw (-3.3,0) node {$1$};
  \end{tikzpicture}
 \end{center}
 Thus the number of essential vertices in this cell complex
 is $1$.

 As a more concrete application, we compute the braid group
 $\pi_1(\Conf_2(\Gamma)/\Sigma_2)$ of two strands of graphs with
 vertices $\le 2$.

\begin{theorem}
 \label{Mukouyama1}
 Let $W_{k,\ell}$ be a finite $1$-dimensional cell complex of the
 following form. 
 
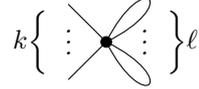
\begin{center}
  \begin{tikzpicture}
    \useasboundingbox (0,-0.5) rectangle (1,0.5);

    \draw [fill] (0,0) circle (2pt);

    \draw (0,0) -- (-0.5,0.5);
    \draw (-0.5,0.1) node {$\vdots$};
    \draw (0,0) -- (-0.5,-0.5);
    \draw (-1,0) node {$k \biggl\{$};

    \draw (0,0) .. controls (0.5,1) and (1,0.5) .. (0,0);
    \draw (0.5,0.1) node {$\vdots$};
    \draw (0,0) .. controls (0.5,-1) and (1,-0.5) .. (0,0);
    \draw (1,0) node {$\biggr\} \ell$};
  \end{tikzpicture}
 \end{center}

 Then the fundamental groups of ordered and unordered configuration
 spaces of two points in $W_{k,\ell}$ are given by
 \begin{eqnarray*}
  \pi_1\left(\Conf_2\left(W_{k,\ell}\right)\right) & \cong &
   F_{2n_{k,\ell}+1} \\ 
  \pi_1\left(\Conf_2\left(W_{k,\ell}\right)/\Sigma_2\right) & \cong &
   F_{n_{k,\ell}+1},   
 \end{eqnarray*}
 where $n_{k,\ell}=\frac{1}{2}(k+\ell)(k+3\ell-3)$ and $F_n$ denotes the
 free group of rank $n$. 
\end{theorem}

\begin{theorem}[Theorem \ref{graph_braid_group_2_vertex}]
 \label{Mukouyama2}
 Let ${}_{x}B^{k,\ell}_{p,q}$ be the finite $1$-dimensional cell complex
 obtained by gluing the essential vertices of $W_{k,\ell}$ and $W_{p,q}$
 by $x$ parallel edges. 
 
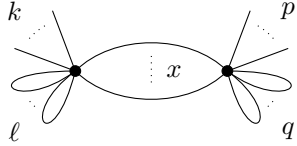
\begin{center}
  \begin{tikzpicture}
   \draw [fill] (0,0) circle (2pt);

   \draw (0,0) -- (-0.3,0.8);
   \draw [dotted] (-0.4,0.6) arc (120:140:0.8);
   \draw (0,0) -- (-0.8,0.3);
   \draw (-0.8,0.8) node {$k$};

   \draw (0,0) .. controls (-1.4,0) and (-0.8,-0.6) .. (0,0);
   \draw [dotted] (-0.6,-0.4) arc (220:230:0.8);
   \draw (0,0) .. controls (-0.3,-1.2) and (-0.8,-0.6) .. (0,0);
   \draw (-0.8,-0.8) node {$\ell$};

   \draw [fill] (2,0) circle (2pt);
   \draw (0,0) .. controls (0.5,0.5) and (1.5,0.5) .. (2,0);
   \draw [dotted] (1,0.2) -- (1,-0.2);
   \draw (1.3,0) node {$x$};
   \draw (0,0) .. controls (0.5,-0.5) and (1.5,-0.5) .. (2,0);

   \draw (2,0) -- (2.3,0.8);
   \draw [dotted] (2.4,0.6) arc (60:40:0.8);
   \draw (2,0) -- (2.8,0.3);
   \draw (2.8,0.8) node {$p$};

   \draw (2,0) .. controls (3.4,0) and (2.8,-0.6) .. (2,0);
   \draw [dotted] (2.6,-0.4) arc (320:310:0.8);
   \draw (2,0) .. controls (2.3,-1.2) and (2.8,-0.6) .. (2,0);
   \draw (2.8,-0.8) node {$q$};
  \end{tikzpicture}
 \end{center}

 Then the fundamental groups of ordered and unordered configuration
 spaces of two points in ${}_xB^{k,\ell}_{p,q}$ are given by
 \begin{eqnarray*}
  \pi_1\left(\Conf_2\left({}_xB^{k,\ell}_{p,q}\right)\right) & \cong &
   A_{\ell,q}\ast 
   A_{q,\ell}\ast F_{2{}_{x}m^{k,\ell}_{p,q}-1} \\
  \pi_1\left(\Conf_2\left({}_xB^{k,\ell}_{p,q}\right)/\Sigma_2\right)
   & \cong & A_{\ell,q}\ast F_{{}_{x}m^{k,\ell}_{p,q}}, 
 \end{eqnarray*}
 where
 \[
  A_{\ell,q}=\langle a_1,\ldots,a_{\ell},b_1,\ldots,b_{q} \mid
 [a_j,b_t]\ (1\le j\le \ell,1\le t\le q) \rangle 
 \]
 and
 \[
  {}_{x}m^{k,\ell}_{p,q} = n_{k,\ell}+n_{p,q} + x(k+\ell+p+q) +
 \frac{x(x-1)}{2}. 
 \] 
\end{theorem}

\subsection*{Organization}

Here is an outline of this paper.

\begin{itemize}
 \item \S\ref{graph_configuration_space_preliminaries} is
       preliminary. We recall definitions and basic properties of our
       main tools, i.e.\ acyclic categories in
       \S\ref{acyclic_category_graph} and cellular 
       stratified spaces in \S\ref{cellular_stratified_space_graph}.

       Although a more general class of cellular stratified spaces,
       i.e.\ cylindrically normal cellular stratified spaces, is studied 
       in \cite{1106.3772,1111.4774}, we prove basic properties of
       totally normal cellular stratified spaces from scratch in
       \S\ref{totally_normal_cellular_stratified_space} in order to be 
       self-contained. Homotopy-theoretic properties of totally normal
       cellular stratified spaces used in this paper are stated and
       proved in \S\ref{face_category} by appealing to homotopy theory
       of acyclic categories.

 \item A new combinatorial model for configuration spaces of graphs is
       constructed in \S\ref{model_for_graph} in two steps.
       After introducing a stratification on the configuration spaces of
       graphs in \S\ref{braid_stratification_for_graph}, we define an
       acyclic category model in \S\ref{category_model_for_graph}, which
       is the first step.

       In many cases, our acyclic category model can be collapsed
       further. This is done for the case of configuration spaces of two
       points in \S\ref{model_for_two_points}.

 \item We discuss two applications of our model in
       \S\ref{application_of_model_for_graph}. Theorem
       \ref{bound_by_the_number_of_vertices} is proved 
       in \S\ref{homotopy_dimension} and Theorem \ref{Mukouyama1} and
       \ref{Mukouyama2} are proved in \S\ref{graph_braid_group}. 

 \item We include a proof of elementary fact on cellular stratified
       subspaces of spheres in an appendix
       \S\ref{deformation_retraction}, which plays an essential role in
       our proof of Theorem
       \ref{barycentric_subdivision_of_cellular_stratification}.  
\end{itemize}

\subsection*{Acknowledgments}

This paper includes some parts of master's theses of the first and the
second authors at Shinshu University under the guidance of the third
author. They are grateful to Katsuhiko Kuribayashi for his comments and
encouragements. 

The third author would like to express his gratitude to Ibai Basabe,
Jes{\'u}s Gonz{\'a}lez, Yuli Rudyak, Peter Landweber, Robert Ghrist,
Sadok Kallel, and Vitaliy Kurlin for their interests.
The idea of using cellular stratified spaces comes from the joint work
of the third author with Basabe, Gonz{\'a}lez, and Rudyak, especially
during discussions with Gonz{\'a}lez over the Internet. 
Our model and Theorem \ref{bound_by_the_number_of_vertices} was
announced in NOLTA 2011 (2011 International Symposium on
Nonlinear Theory and its Applications) \cite{NOLTA2011}, during which
the third author received helpful comments and encouragements from
Ghrist.  
We would also like to thank the organizers of the conference
``Applied Topology 2013, Bedlewo'', during which the third author had
chances to discuss with Kallel and Kurlin.

The third author is supported by JSPS KAKENHI Grant Number 23540082.

\section{Cellular Stratified Spaces}
\label{graph_configuration_space_preliminaries}

Homotopy theory of acyclic categories plays the fundamental role in this
paper. The notion of cellular stratified spaces makes the connection
between configuration spaces and acyclic categories. In this section, we
first review homotopy-theoretic properties of acyclic categories in
\S\ref{acyclic_category_graph}. And
then the definition of cellular stratified spaces is recalled from
\cite{1106.3772}. Homotopy-theoretic properties of totally normal
cellular stratified spaces are stated and proved based on the discussion
in \S\ref{acyclic_category_graph}.
. 

\subsection{Acyclic Categories}
\label{acyclic_category_graph}

Acyclic categories are generalizations of posets, having possibly
multiple ``orders'' (morphisms) between two objects. 
There is a popular
way to define a quiver from a poset, i.e.\ the Hasse diagram. Any small
category has an underlying quiver and can be regarded as a ``quiver with
relations''. 

In this section, we review relations among these concepts. A good
reference is Kozlov's book
\cite{KozlovCombinatorialAlgebraicTopology}. See also Appendix B of
\cite{1106.3772}. 

\begin{definition}
 A \emph{quiver} is a diagram of sets of the form
 \[
  \xymatrix{Q_1 \ar@<1ex>[r]^{s} \ar@<-1ex>[r]_{t} & Q_0.}
 \]
 Elements of $Q_0$ and $Q_1$ are called \emph{vertices} and
 \emph{arrows}, respectively. For an arrow $u\in Q_1$, $s(u)$ and $t(u)$
 are called the \emph{source} and the \emph{target} of $u$,
 respectively. For a pair of vertices $x$ and $y$, we denote
 \[
  Q(x,y) = \lset{u\in Q_1}{s(u)=x, t(u)=y}.
 \]

 For $n\ge 1$, define 
 \[
 N_n(Q)= \lset{(u_1,\ldots,u_n)}{s(u_1)=t(u_2), \ldots,
 s(u_{n-1})=t(u_n)}. 
 \]
 We also use the notation $N_0(Q)=Q_0$. Elements of $N_n(Q)$ are called
 \emph{$n$-chains}. 
\end{definition}



\begin{definition}
 A \emph{small category} is a quiver $C$ equipped with maps
 \begin{eqnarray*}
  \circ & : & N_2(C) \longrightarrow N_1(C) \\
  i & : & N_0(C) \longrightarrow N_1(C)
 \end{eqnarray*}
 satisfying the following conditions:
 \begin{enumerate}
  \item $(u\circ v)\circ w = u\circ (v\circ w)$ for $(u,v,w)\in N_3(C)$,
	and 
  \item $u\circ i(s(u)) = u = i(t(u))\circ u$ for $u\in C_1$.
 \end{enumerate}
 Elements of $C_0$ and $C_1$ are called \emph{objects} and
 \emph{morphisms}, respectively. For objects $x,y\in C_0$, the set of
 morphisms with source $x$ and target $y$ is denoted 
 by $C(x,y)$. For an object $x\in C_0$, $i(x)$ is called the
 \emph{identity morphism} on $x$ and is denoted by $1_x$.
\end{definition}

Sometimes it is convenient to remove identity morphisms.

\begin{definition}
 For a small category $C$, define
 \[
  \overline{N}_n(C) = \lset{(u_{n},\ldots,u_1)\in N_n(C)}{\text{none of
 $u_i$'s is identity}}. 
 \]
 Elements of $\overline{N}_n(C)$ are called \emph{nondegenerate $n$-chains}.
\end{definition}

\begin{definition}
 For a small category $C$, define $Q(C)_0=C_0$ and
 $Q(C)_1 = \overline{N}_1(C)=C_1\setminus i(C_0)$. And define maps
 \[
  s,t : Q(C)_1 \longrightarrow Q(C)_0
 \]
 by the restrictions of $s$ and $t$ of $C$. This quiver is called the
 \emph{underlying quiver} of $C$.
\end{definition}

Conversely, any quiver generates a small category.

\begin{definition}
 Let $Q$ be a quiver. Define a small category $\Path(Q)$ as
 follows. Objects and morphisms are given by
 \begin{eqnarray*}
  \Path(Q)_0 & = & Q_0 \\
  \Path(Q)_1 & = & \coprod_{n\ge 0} N_n(Q).
 \end{eqnarray*}
 The source and the target maps on $N_n(Q)$
 \[
  s_n,t_n : N_n(Q) \longrightarrow Q_0
 \]
 are given by
 \begin{eqnarray*}
  s_n(u_n,\cdots,u_1) & = & s(u_1), \\
  t_n(u_n,\cdots,u_1) & = & t(u_n).
 \end{eqnarray*}
 The composition and the identity
 \begin{eqnarray*}
  \circ & : & N_2(\Path(Q)) = \coprod_{m,n\ge 0}
   \lset{(\bm{u},\bm{v})\in N_m(Q)\times N_n(Q)}{s(\bm{u})=t(\bm{v})}
   \longrightarrow \Path(Q) \\
  i & : & \Path(Q)_0 = Q_0 \longrightarrow \Path(Q)_1
 \end{eqnarray*}
 are given by the concatenation and the inclusion.

 The resulting category $\Path(Q)$ is called the \emph{path category} of $Q$.
\end{definition}

We are mainly interested in acyclic categories.

\begin{definition}
 A quiver $Q$ is said to be \emph{acyclic} if
 either $Q(x,y)$ or $Q(y,x)$ is empty.
 A small category $C$ is called \emph{acyclic} if its underlying quiver
 $Q(C)$ is acyclic.
\end{definition}

\begin{remark}
 A small category $C$ is acyclic if and only if
 \begin{itemize}
  \item $C(x,x)=\{1_x\}$ and,
  \item either $C(x,y)$ or $C(y,x)$ is empty for $x\neq y$.
 \end{itemize}
\end{remark}

Any poset $P$ can be regarded as an acyclic category by
$P(x,y)=\{\ast\}$ if $x\le y$ and $P(x,y)=\emptyset$
otherwise. Conversely any acyclic category has an associated
poset. 

\begin{definition}
 \label{associated_poset}
 For an acyclic category $C$, define a poset $P(C)$ as follows. As sets,
 $P(C)=C_0$. For $x,y\in P(C)$, $x\le y$ if and only if
 $C(x,y)\neq\emptyset$. 

 The canonical projection functor is denoted by
 \[
  p : C \longrightarrow P(C).
 \]
\end{definition}

We use the classifying space functor to translate combinatorial
(category-theoretic) structures into the subject of homotopy theory.

\begin{definition}
 \label{classifying_space_definition}
 Given a small category $C$, the collection of chains
 $N(C)=\{N_n(C)\}_{n\ge 0}$ forms a simplicial set, called the
 \emph{nerve} of $C$. The geometric realization of this simplicial set 
 is called the \emph{classifying space} of $C$ and is denoted by
 $BC=|N(C)|$.  
\end{definition}

Recall that the face operators 
$d_i : N_n(C)\to N_{n-1}(C)$ are defined by 
\[
 d_i(u_n,\cdots,u_1) = \begin{cases}
			(u_n,\cdots,u_2), & \text{ if } i=0 \\
			(u_n,\cdots, u_{i+1}\circ u_i, \cdots, u_1), &
			\text{ if } 1\le i\le n-1 \\ 
			(u_{n-1},\cdots,u_1), & \text{ if } i=n
		\end{cases}
\]
for $i=0,\cdots,n$. 

When $C$ is acyclic, we only need these face operators to describe $BC$.

\begin{lemma}
 \label{BC_by_nondegenerate_chains}
 For an acyclic category $C$, face operators can be restricted to
 $d_i : \overline{N}_n(C)\to \overline{N}_{n-1}(C)$, and, when $C$ is
 finite, we have a homeomorphism
 \[
  BC \cong \quotient{\coprod_{n} \overline{N}_n(C)\times\Delta^n}{\sim},
 \]
 where $\sim$ is the equivalence relation generated by
 \begin{equation}
  (d_i(u_n,\cdots,u_1),(t_0,\cdots,t_{n-1})) \sim ((u_n,\cdots, u_1),
   (t_0,\cdots, t_{i-1},0,t_i,\cdots, t_{n-1})).
 \label{relation_for_geometric_realization}  
 \end{equation}
 Here our $n$-simplex $\Delta^n$ is given by
 \[
  \Delta^n = \set{(t_0,\cdots,t_n)\in\R^n}{t_0+\cdots+t_n=1,t_i\ge 0}.
 \]
\end{lemma}

\begin{proof}
 It is immediate to verify that $d_i$ can be restricted to
 $\overline{N}(C)$. The inclusions
 $\overline{N}_n(C)\hookrightarrow N_n(C)$ induce a continuous bijective
 map
 \[
 \quotient{\coprod_{n} \overline{N}_n(C)\times\Delta^n}{\sim} \rarrow{}
 BC. 
 \]
 By the finiteness assumption and acyclicity of $C$,
 $\overline{N}_n(C)=\emptyset$ for sufficiently large $n$ and each
 $\overline{N}_n(C)$ is a finite set. Thus the above map is a continuous
 bijection from a compact space to a Hausdorff space, hence is a
 homeomorphism. 
\end{proof}

\begin{remark}
 It is well-known that when $P$ is a poset, the collection of
 nondegenerate chains $\overline{N}(P)$ has a structure of ordered
 simplicial complex and is often called the \emph{order complex} of
 $P$. 
\end{remark}

%
%


The following definition is slightly different from but equivalent to
the definition of graded posets in combinatorics.

\begin{definition}
 A poset $P$ is called \emph{graded}, if for any element $x\in P$,
 $\dim BP_{\le x}$ is finite, where
 $P_{\le x}= \lset{y\in P}{y\le x}$. 

 For a graded poset $P$, define the \emph{rank function}
 \[
  r : P \longrightarrow \Z_{\ge 0}
 \]
 by
 \[
  r(x)= \dim BP_{\le x}.
 \]
\end{definition}


When a poset $P$ is regarded as a small category, the rank function
$r :P \to \Z_{\ge 0}$ is a functor. We would like to define an analogous
notion for small categories.
Bessis \cite{math.GR/0610778} considered functions on the set of
morphisms.   

\begin{definition}
 \label{length_function_definition}
 A \emph{length function} on a small category $C$ is a map 
 \[
  \ell : C_1 \longrightarrow \Z_{\ge 0}
 \]
 with the properties that
 \begin{enumerate}
  \item $\ell(u\circ v) = \ell(u)+\ell(v)$;
  \item $\ell(u) = 0$ if and only if $u=1_x$ for some $x\in C_0$.
 \end{enumerate}
\end{definition}

A small category equipped with a length function is called
\emph{homogeneous} by Bessis. We use the following terminology.

\begin{definition}
 \label{category_with_length_function}
 A triple $(C,\ell,B)$ of a small category $C$, a length function $\ell$
 on $C$, and a set $B$ of objects in $C$, is called a \emph{category
 with length function} if the following conditions are satisfied: 
 \begin{enumerate}
  \item For any object $x\in C_0$, there exists an object $b\in B$ with
	$C(b,x)\neq\emptyset$. 
  \item For any morphisms $u : b\to x$ and $u' : b'\to x$ with
	$b,b'\in B$, we have $\ell(u)=\ell(u')$.
 \end{enumerate}
\end{definition}

\begin{definition}
 For a category with length function $(C,\ell,B)$, define a function 
 \[
 r : C_0 \longrightarrow \Z_{\ge 0}
 \]
 by
 \[
  r(x) = \ell(u),
 \]
 where $x\in C_0$ and $u : b\to x$ is a morphism in $C$ with $b\in B$.
 This is called a \emph{rank functor} on $C$ because of the following
 reason.  
\end{definition}

\begin{lemma}
 \label{rank_functor_from_length_function}
 The above function $r$ can be extended to a functor
 \[
  r : C \longrightarrow \Z_{\ge 0}.
 \]
\end{lemma}

\begin{proof}
 We need to show that if there exists a morphism $u : x\to y$ in $C$,
 then $r(x)\le r(y)$. Choose a morphism $v : b\to x$ with $b\in B$. Then 
 \[
  r(y)= \ell(u\circ v) = \ell(u)+\ell(v) = \ell(u)+r(x)\ge r(x).
 \]
\end{proof}

\begin{remark}
 For a category with length function $(C,\ell,B)$, we can recover $\ell$
 and $B$ from the rank functor $r$ by
 \begin{eqnarray*}
  \ell(u) & = & r(t(u)) - r(s(u)) \\
  B & = & \lset{x\in C_0}{r(x)=0}.
 \end{eqnarray*}
 In the rest of this article, we denote a category with length function
 by the pair $(C,r)$. 
\end{remark}

\begin{lemma}
 \label{ranked_category_is_acyclic}
 Any category with length function is acyclic. 
\end{lemma}

\begin{proof}
 For $u\in C(x,x)$, $\ell(u)=r(t(u))-r(s(u))=0$. By the definition of
 length function $u=1_{x}$. Suppose $C(x,y)\neq\emptyset$ and
 $C(y,x)\neq\emptyset$. By Lemma
 \ref{rank_functor_from_length_function}, $r(x)\le r(y)$ and 
 $r(y)\le r(x)$, which imply that $x=y$ and that $C$ is acyclic.
\end{proof}

\begin{definition}
 \label{ranked_category_definition}
 A category with length function $(C,\ell,B)$ with rank functor $r$ is
 called a \emph{ranked category} if, 
 for any morphism $u:x\to y$ with $\ell(u)=k$, there exists a
 factorization of $u$ into a composition
 \[
  u : x=x_0 \to x_1 \to \cdots \to x_k=y
 \]
 with $r(x_{i+1})=r(x_i)+1$ for $0\le i\le k-1$.
\end{definition}


We mainly use functors of the following form.

\begin{definition}
 A functor $f:C\to\category{Top}$ from a small category $C$ to the
 category $\category{Top}$ of topological spaces is said
 to be \emph{continuous} if for each morphism $u:x\to y$ in $C$, the
 induced map $f(u): f(x)\to f(y)$ is continuous.
\end{definition}

\begin{remark}
 This condition is equivalent to the continuity of the adjoint
 \[
  C(x,y)\times f(x) \longrightarrow f(y)
 \]
 of the map $f(x,y) : C(x,y)\to\Map(f(x),f(y))$ when $C(x,y)$ is
 equipped with the discrete topology.
\end{remark}

The following construction will be used later when we study boundaries
of cells.

\begin{definition}
 \label{comma_category_definition}
 For an object $x\in C_0$ in a small category $C$, define a small
 category $C\downarrow x$ by 
 \begin{eqnarray*}
  (C\downarrow x)_0 & = & \set{u\in C_1}{t(u)=x} \\
  (C\downarrow x)(u,v) & = & \set{w\in C_1}{u=v\circ w}.
 \end{eqnarray*}
 The composition of morphisms is given by that of $C$. This is called
 the \emph{comma category of $C$ over $x$}. 

 When $C$ is acyclic, we denote $C\downarrow x$ by $C_{\le x}$ for
 simplicity. The subcategory of $C_{\le x}$ consisting of
 $(C\downarrow x)_0\setminus\{1_{x}\}$ is denoted by $C_{< x}$.
\end{definition}
\subsection{Cellular Stratified Spaces}
\label{cellular_stratified_space_graph}

Cellular stratified spaces are generalizations of cell complexes, having
possibly non-closed cells.
Let us begin with the definition of stratifications.

\begin{definition}
 A \emph{stratification} on a topological space $X$ indexed by a poset
 $\Lambda$ is a map 
 \[
  \pi : X \rarrow{} \Lambda
 \]
 satisfying the following two conditions:
 \begin{enumerate}
  \item $\lambda\le \mu$ in $\Lambda$ if and only if
	$\pi^{-1}(\lambda)\subset \overline{\pi^{-1}(\mu)}$.
  \item Each $\pi^{-1}(\lambda)$ is connected and locally
	closed. 
 \end{enumerate}

 The image of $\pi$ as a full subposet is denoted by $P(X,\pi)$,
 $P(X)$, or $P(\pi)$, and is called the \emph{face poset} of $X$.
 The space $e_{\lambda}=\pi^{-1}(\lambda)$ is called the \emph{stratum}
 indexed by $\lambda\in\Lambda$.  
\end{definition}

\begin{definition}
 Let $\pi : X\to\Lambda$ be a stratification on a Hausdorff space
 $X$. 
 \begin{itemize}
  \item For a stratum $e_{\lambda}$, an 
	\emph{$n$-cell structure} on $e_{\lambda}$ is a pair
	$(D_{\lambda},\varphi_{\lambda})$ of a subspace $D_{\lambda}$ of
	the unit $n$-disk $D^n$ with $\Int(D^n)\subset D_{\lambda}$ and
	a quotient map  
	\[
	\varphi_{\lambda} : D_{\lambda} \longrightarrow X
	\]
	satisfying the following conditions:
	\begin{enumerate}
	 \item $\varphi(D_{\lambda})=\overline{e_{\lambda}}$.
	 \item The restriction
	       $\varphi_{\lambda}|_{\Int(D^n)}:\Int(D^n)\to e_{\lambda}$
	       is a homeomorphism.  
	\end{enumerate}

	For simplicity, we refer to an $n$-cell structure
	$(D_{\lambda},\varphi_{\lambda})$ on $e_{\lambda}$ 
	by $\varphi_{\lambda}$ when there is no risk of confusion.

  \item A stratum equipped with an $n$-cell structure is called an
	\emph{$n$-cell}. 

  \item The map $\varphi_{\lambda}$ is called the \emph{characteristic
	map} of $e_{\lambda}$ and $D_{\lambda}$ is called the
	\emph{domain} of $e_{\lambda}$. The dimension $n$ of the disk  
	$D^n$ containing the domain $D_{\lambda}$ is called the
	\emph{dimension} of $e_{\lambda}$. 

  \item A \emph{cellular stratification} on $X$ consists of 
	\begin{itemize}
	 \item a stratification $\pi : X \to \Lambda$, and 
	 \item a collection
$\Phi=\{\varphi_{\lambda}:D_{\lambda}\to\overline{e_{\lambda}}\}_{\lambda\in	P(X)}$
	of cell structures on strata 
	\end{itemize}
	satisfying the condition that,
	for each $n$-cell $e_{\lambda}$,
	$\partial e_{\lambda}=\overline{e_{\lambda}}\setminus e_{\lambda}$
	is covered by a finite number of cells of 
	dimension less than or equal to $n-1$.

	The triple $(X,\pi,\Phi)$ is called a \emph{cellular stratified
	space}. 
 \end{itemize}
\end{definition}

\begin{remark}
 Note that we require that a cell structure map to be a quotient
 map. This condition is automatic in the classical definition of cell
 complex, since a cell complex is always assumed to be Hausdorff and
 $D^n$ is compact. 
\end{remark}

\begin{definition}
 \label{morphism_of_stratified_spaces}
 Let $(X,\pi_X,\Phi_{X})$ and $(Y,\pi_Y,\Phi_{Y})$ be cellular
 stratified spaces.  
 \begin{itemize}
  \item A \emph{morphism of cellular stratified spaces} $\bm{f}$ from
	$(X,\pi_X,\Phi_{X})$ to $(Y,\pi_Y,\Phi_{Y})$ consists of 
	\begin{itemize}
	 \item a continuous map $f : X\to Y$,
	 \item a map of posets
	       $\underline{f} : P(X)\to\ P(Y)$,
	       and
	 \item a family of maps
	       $f_{\lambda} : D_{\lambda} \longrightarrow D_{f(\lambda)}$
	       indexed by cells
	       $\varphi_{\lambda}: D_{\lambda}\to \overline{e_{\lambda}}$ in
	       $X$
	\end{itemize}
	making the following diagrams commutative
	\[
	\begin{diagram}
	 \node{X} \arrow{e,t}{f} \arrow{s,l}{\pi_X} \node{Y}
	 \arrow{s,r}{\pi_Y} \\ 
	 \node{P(X)} \arrow{e,b}{\underline{f}} \node{P(Y).}
	\end{diagram}
	\begin{diagram}
	 \node{X} \arrow{e,t}{f} \node{Y} \\
	 \node{D_{\lambda}} \arrow{n,l}{\varphi_{\lambda}}
	 \arrow{e,b}{f_{\lambda}} \node{D_{f(\lambda)}}
	 \arrow{n,r}{\psi_{f(\lambda)},} 
	\end{diagram}
	\]
	where
	$\psi_{f(\lambda)} : D_{f(\lambda)} \to \overline{e_{f(\lambda)}}$
	is the characteristic map for $e_{f(\lambda)}$. 
  \item When $X=Y$ and $f$ is the identity, $\bm{f}$ is 
	called a \emph{subdivision}.
  \item When $f(e_{\lambda}) = e_{\underline{f}(\lambda)}$ and
	$f_{\lambda}(0)=0$ for each
	$\lambda$, $\bm{f}$ is called a \emph{strict morphism}. 
  \item When $\bm{f}$ is a strict morphism of cellular stratified
	spaces and $f$ is an embedding of topological spaces, $\bm{f}$
	is said to be an \emph{embedding} of cellular stratified spaces
	and $X$ is said to be a \emph{cellular stratified subspace} of
	$Y$. 
 \end{itemize}
\end{definition}


We need to impose certain ``niceness'' conditions to our cellular
stratified spaces. As is the case of cell complexes, we usually require
CW conditions. 

\begin{definition}
 A cellular stratification on a space $X$ is said to be \emph{CW}, if
 the following two conditions are satisfied: 
 \begin{enumerate}
  \item (Closure Finite) For each cell $e_{\lambda}$,
	$\partial e_{\lambda}$ is covered by a finite number of cells. 
  \item (Weak Topology) $X$ has the weak topology determined by the
	covering $\{\overline{e_{\lambda}}\mid \lambda\in P(X) \}$.
 \end{enumerate} 
\end{definition}

The CW condition allow us to express a cellular stratified space as a
quotient space as follows.

\begin{lemma}
 \label{cellular_stratified_space_as_quotient}
 For a CW cellular stratified space $X$ with cell structure
 $\Phi=\set{\varphi_{\lambda}:D_{\lambda}\to\overline{e_{\lambda}}}{\lambda\in
 P(X)}$, define
 $D(X)=\coprod_{\lambda\in P(X)}D_{\lambda}$ and
 $\widetilde{\Phi} : D(X)\to X$ by 
 $\widetilde{\Phi}(x)= \varphi_{\lambda}(x)$ if $x\in D_{\lambda}$. Then
 $\widetilde{\Phi}$ is a 
 quotient map. In particular, we have a homeomorphism
 \[
  X \cong D(X)/_{\sim_{\Phi}},
 \]
 where $x\sim_{\Phi} y$ if and only if
 $\varphi_{\mu}(x)=\varphi_{\lambda}(y)$ for $x\in D_{\mu}$ and
 $y\in D_{\lambda}$.
\end{lemma}

\begin{proof}
 The map $\widetilde{\Phi}$ factors as
 \[
 \widetilde{\Phi} : D(X) \rarrow{\coprod \varphi_{\lambda}}
 \coprod_{\lambda\in P(X)} \overline{e_{\lambda}} \rarrow{\rho} X.
 \]
 All characteristic maps $\varphi_{\lambda}$ are quotient maps. By the
 CW assumption, $\rho$ is a quotient map. Hence
 $\widetilde{\Phi}$ is a quotient map.
\end{proof}

In order to study configuration spaces, we need to understand products,
subdivisions, and taking complements of cellular stratified spaces. Let
us first consider products. 
We need to impose the following condition.

\begin{lemma}
 \label{product_of_cellular_stratified_spaces}
 Let $(X,\pi_X,\Phi_X)$ and $(Y,\pi_Y,\Phi_Y)$ be cellular stratified
 spaces and consider the product map
 \[
  \pi_X\times\pi_Y : X\times Y \longrightarrow P(X)\times P(Y).
 \]
 For a pair of cells $e_{\lambda}$, $e_{\mu}$ in $X$ and $Y$, define a
 continuous map
 \[
  \varphi_{\lambda,\mu} : D_{\lambda,\mu} \cong D_{\lambda}\times D_{\mu}
 \rarrow{\varphi_{\lambda}\times\varphi_{\mu}} 
 \overline{e_{\lambda}}\times \overline{e_{\mu}} =
 \overline{e_{\lambda}\times e_{\mu}} \subset X\times Y, 
 \]
 where $D_{\mu,\lambda}$ is the subspace of
 $D^{\dim e_{\lambda}+\dim e_{\mu}}$ 
 defined by pulling back $D_{\lambda}\times D_{\mu}$ via the standard
 homeomorphism  
 \[
 D^{\dim e_{\lambda}+\dim e_{\mu}} \cong D^{\dim e_{\lambda}}\times
 D^{\dim e_{\mu}}. 
 \]

 If $\varphi_{\lambda,\mu}$ is a quotient map for any pair
 $(\lambda,\mu)$, we have a cellular stratification on $X\times Y$. 
\end{lemma}

\begin{proof}
 Obvious from the definition.
\end{proof}

The problem is when $\varphi_{\lambda}\times\varphi_{\mu}$ is a quotient
map. When both $X$ and $Y$ are cell complexes, the compactness of
$D_{\lambda}\times D_{\mu}$ implies that
$\varphi_{\lambda}\times\varphi_{\mu}$ is a quotient map. In general,
$D_{\lambda}$ or $D_{\mu}$ is neither closed nor open. This problem is
discussed in \S3.2 of \cite{1111.4774} in detail.
For configuration spaces of graphs, the following fact is enough.

\begin{lemma}
 \label{product_of_locally_compact_quotient}
 If $f_1:X_1\to Y_1$ and $f_2:X_2\to Y_2$ are surjective closed maps
 between metrizable spaces, then the product
 $f_1\times f_2:X_1\times X_2\to Y_1\times Y_2$ is a quotient map.
\end{lemma}

\begin{proof}
 See Corollary right after Theorem 3 in \cite{Himmelberg1965}.
\end{proof}

We are interested in $1$-dimensional cellular stratified spaces.

\begin{corollary}
 \label{stratification_on_product_of_graphs}
 When $\Gamma_1$ and $\Gamma_2$ are $1$-dimensional cellular stratified
 spaces, $\Gamma_1\times\Gamma_2$ is a cellular stratified space. 
\end{corollary}

\begin{proof}
 When $\varphi : D \to \overline{e}$ is a characteristic map on a
 $1$-cell $e$, the domain $D$ is one of $(-1,1)$, $(-1,1]$, $[-1,1)$, or
 $[-1,1]$. In any of these cases, $\varphi$ is a closed
 map. Furthermore, the closure $\overline{e}$ of the cell is
 homeomorphic to one of $(-1,1)$, $(-1,1]$, $[-1,1)$, $[-1,1]$, or
 $S^1$, and is metrizable. Thus products of characteristic maps are
 again quotient maps by Lemma
 \ref{product_of_locally_compact_quotient}. 

 Other requirements hold obviously and 
 $\Gamma_1\times\Gamma_2$ is a cellular stratified space.
\end{proof}

Let us consider subdivisions next. We have already defined subdivisions
of stratified spaces in Definition
\ref{morphism_of_stratified_spaces}. Subdivisions of cell structures are
defined as follows.

\begin{definition}
 A \emph{cellular subdivision} of a cellular stratified space
 $(X,\pi,\Phi)$ consists of
 \begin{itemize}
  \item a subdivision $\bm{s}=(1_{X,s}):(X,\pi')\to (X,\pi)$ of
	$(X,\pi)$ as a stratified space, and
  \item a regular cellular stratification
	$(\pi_{\lambda},\Phi_{\lambda})$ on the domain $D_{\lambda}$ for
	each cell
	$\varphi_{\lambda} : D_{\lambda}\to \overline{e_{\lambda}}$
	containing $\Int(D_{\lambda})$ as a strict stratified subspace 
 \end{itemize}
 satisfying the following conditions:
 \begin{enumerate}
  \item For each $\lambda\in P(X,\pi)$, the characteristic map
	\[
	 \varphi_{\lambda} : (D_{\lambda},\pi_{\lambda}) \longrightarrow
	(X,\pi') 
	\]
	of $e_{\lambda}$ is a strict morphism of stratified spaces.
  \item The maps
	\[
	P(\varphi_{\lambda}) : P(\Int(D_{\lambda})) \longrightarrow
	P(X,\pi') 
	\]
	induced by the characteristic maps $\varphi_{\lambda}$ give rise
	to a bijective morphism of posets
	\[
	 \coprod_{\lambda\in P(X,\pi)}P(\varphi_{\lambda}) :
	\coprod_{\lambda\in P(X,\pi)}P(\Int(D_{\lambda}),\pi_{\lambda}) 
	\longrightarrow P(X,\pi').
	\]

  \item For each $\lambda'\in P(X,\pi')$ with
	$s(\lambda')=\lambda\in P(X,\pi)$, let us denote the
	corresponding strata in $(X,\pi')$ and
	$(D_{\lambda},\pi_{\lambda})$ by $e_{\lambda'}$ and
	$E_{\lambda'}$, respectively.
	If
	$\psi_{\lambda'} : D_{\lambda'}\to \overline{E_{\lambda'}}$ is
	the characteristic map for $E_{\lambda'}$ in the regular
	cellular stratification
	$(D_{\lambda},\pi_{\lambda},\Phi_{\lambda})$, then the
	composition
	\[
	 \varphi_{\lambda}\circ\psi_{\lambda'} : D_{\lambda'} \rarrow{}
	\overline{e_{\lambda'}} 
	\]
	is a quotient map.
 \end{enumerate}
\end{definition}

\begin{remark}
 The map $\coprod_{\lambda\in P(X,\pi)}P(\varphi_{\lambda})$ may not be
 an isomorphism of posets, although it is assumed to be a bijection.
\end{remark}

The composition $\varphi_{\lambda}\circ\psi_{\lambda}$ is essentially
the restriction of $\varphi_{\lambda}$ to $E_{\lambda'}$, since the
cellular stratification $(D_{\lambda},\phi_{\lambda},\Phi_{\lambda})$
is assumed to be regular. 
In general, however, a restriction of a quotient map may not be a
quotient map. This is the reason we need to impose the condition 3 in
the above definition. In other words, the definition is designed to make
the following proposition hold.

\begin{proposition}
 \label{cellular_subdivision_is_cellular_stratified}
 A cellular subdivision of a cellular stratified space is again a
 cellular stratified space.
\end{proposition}

\subsection{Totally Normal Cellular Stratified Spaces}
\label{totally_normal_cellular_stratified_space}

Regularity and normality conditions are important in our
analysis of cellular stratified spaces.

\begin{definition}
 A cellular stratification on a space $X$ is said to be
 \begin{itemize}
  \item \emph{normal}, if $e_{\mu}\subset\overline{e_{\lambda}}$
	whenever $e_{\mu}\cap \overline{e_{\lambda}}\neq\emptyset$, for
	any cell $e_{\lambda}$, 
  \item \emph{regular}, if the characteristic map
	$\varphi : D_{\lambda}\to \overline{e_{\lambda}}$ of each cell
	$e_{\lambda}$ is a 
	homeomorphism onto $\overline{e_{\lambda}}$, and
  \item \emph{totally normal}, if, for each $n$-cell
	$e_{\lambda}$,  
	\begin{enumerate}
	 \item there exist a structure of regular
	       cell complex on $S^{n-1}$ containing
	       $\partial D_{\lambda}$ as a cellular stratified subspace
	       of $S^{n-1}$, and 
	 \item for each cell $e$ in the cellular stratification on
	       $\partial D_{\lambda}$, there exists a cell
	       $e_{\mu}$ in $X$ and a map
	       $b:D_{\mu}\to \partial D_{\lambda}$
	       $b(\Int(D_{\mu}))=e$ and
	       $\varphi_{\lambda}\circ b=\varphi_{\mu}$.
	\end{enumerate}
 \end{itemize}
\end{definition}

\begin{remark}
 The regularity of the cellular stratification on
 $\partial D_{\lambda}$ implies that $b$ is an embedding.
\end{remark}

There is a canonical way to associate a small category to any totally
normal cellular stratified space.

\begin{definition}
 For a totally normal cellular stratified space $X$, 
 define a category $C(X)$ as follows. Objects are cells 
 \[
  C(X)_0 = \set{e}{\text{cells in } X}.
 \]

 A morphism from a cell $\varphi_{\mu} : D_{\mu} \to \overline{e_{\mu}}$
 to another cell 
 $\varphi_{\lambda} : D_{\lambda} \to \overline{e_{\lambda}}$ is a lift
 of the characteristic map $\varphi_{\mu}$ of $e_{\mu}$, i.e.\ a map
 $b : D_{\mu}\to D_{\lambda}$ making the following diagram commutative
 \[
  \begin{diagram}
   \node{D_{\lambda}} \arrow{e,t}{\varphi_{\lambda}}
   \node{\overline{e_{\lambda}}} \arrow{e,J} 
   \node{X} \\ 
   \node{D_{\mu}} \arrow{n,l}{b} \arrow{e,b}{\varphi_{\mu}}
   \node{\overline{e_{\mu}}.} 
   \arrow{ne,J} 
  \end{diagram}
 \]
 The composition is given by the composition of maps. 
 This category $C(X)$ is called the \emph{face category} of $X$. 
\end{definition}

\begin{remark}
 In \cite{1106.3772}, the set of morphisms
 $C(X)(e_{\mu},e_{\lambda})$ from $e_{\mu}$ to $e_{\lambda}$ is
 topologized by the compact open topology as a subspace of
 $\Map(D_{\mu},D_{\lambda})$ and $C(X)$ is defined as a topological
 category.  

 In this paper, we only consider face categories of totally
 normal cellular stratified spaces, in which case the topology on
 $C(X)(e,e')$ is automatically discrete.
\end{remark}

\begin{example}
 \label{graphs_are_totally_normal}
 All cellular stratified spaces of dimension $1$ are totally normal,
 since possible domains of characteristic maps are $(-1,1)$, $(-1,1]$,
 $[-1,1)$, or $[-1,1]$.  
\end{example}

Note that the existence of a morphism $b : e_{\mu} \to e_{\lambda}$ in
$C(X)$ implies 
$\overline{e_{\mu}} \subset \overline{e_{\lambda}}$. Thus we obtain a
functor  
\[
 p_X : C(X) \longrightarrow P(X).
\]

\begin{lemma}
 \label{face_category_of_totally_normal_css}
 For a totally normal cellular stratified space $X$, $C(X)$ is a
 category with length function, hence is acyclic. 
 The associated poset $P(C(X))$ coincides with $P(X)$ and the canonical
 projection $p : C(X) \to P(C(X))$ can be identified with $p_X$.
\end{lemma}

\begin{proof}
 Define $\ell :C(X)_1\to \Z_{\ge 0}$ by
 \begin{equation}
  \ell(b) = \dim e_{\lambda}-\dim e_{\mu}  
   \label{length_function_for_cellular_stratified_space}
 \end{equation}
 for a morphism $b\in C(X)(e_{\mu},e_{\lambda})$.
 This is obviously a length function in the sense of Definition
 \ref{length_function_definition}. 
 
 Define $B=\set{e\in C(X)_0}{\dim e =0}$. Then the conditions of  
 Definition \ref{category_with_length_function} are satisfied and
 $(C(X),\ell,B)$ is a category with length function. The associated rank
 functor is obviously the dimension function $\dim$.


 By definition, there exists a morphism $b\in C(X)(e_{\mu},e_{\lambda})$
 if and only if $e_{\mu}\subset\overline{e_{\lambda}}$, i.e.\
 $e_{\mu}\le e_{\lambda}$. Thus $P(C(X))=P(X)$.
\end{proof}

\begin{remark}
 In general $C(X)$ is not a ranked category if we use the function
 $\ell$ in (\ref{length_function_for_cellular_stratified_space}) as a
 length function. For example, define
 \[
  X = \Int D^2 \cup \set{(x,y)\in S^1}{x>0} \cup \{(-1,0)\}.
 \]
 $X$ has a structure of regular cellular stratified space with cells
 $e^0=\{(-1,0)\}$, $e^1=\set{(x,y)\in S^1}{x>0}$, and $e^2=\Int D^2$. 
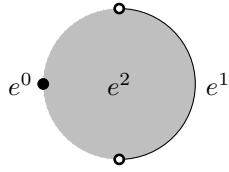
 \begin{center}
  \begin{tikzpicture}
   \draw [dotted,fill,lightgray] (0,0) circle (1);
   \draw (0,0) node {$e^2$};

   \draw (0,-1) arc (-90:90:1);
   \draw [fill] (0,1) circle (2pt);
   \draw [fill,white] (0,1) circle (1pt);
   \draw [fill] (0,-1) circle (2pt);
   \draw [fill,white] (0,-1) circle (1pt);
   \draw (1.3,0) node {$e^1$};

   \draw [fill] (-1,0) circle (2pt);
   \draw (-1.3,0) node {$e^0$};
  \end{tikzpicture}
 \end{center}
 $C(X)(e^0,e^2)$ contains a single element which cannot be factored
 into a composition of morphisms of codimension $1$.
\end{remark}

We use the following terminology for $1$-dimensional cellular stratified
spaces.

\begin{definition}
 \label{graph_terminologies}
 \hspace*{\fill}
 \begin{itemize}
  \item A $1$-dimensional cellular stratified space $\Gamma$ is called
	a \emph{graph}. 
  \item $0$-cells and $1$-cells are called \emph{vertices} and
	\emph{edges}, respectively. 
  \item Let $\varphi : D\to \overline{e}$ be a $1$-cell in $\Gamma$.
	\begin{itemize}
	 \item An edge $e$ is called a \emph{loop} if $D=[-1,1]$ and
	       $\varphi(-1)=\varphi(1)$. 
	 \item An edge $e$ is called a \emph{connection} if $D=[-1,1]$,
	       $\varphi$ is an embedding, and both $\varphi(-1)$ and
	       $\varphi(1)$ are contained in more than one edge.
	 \item An edge $e$ is called a \emph{branch} if it is not a loop
	       nor a connection. 
	\end{itemize}
  \item For a vertex $v$, let $b_v$ be the number of branches and
	bridges $e$ with $v\in\overline{e}$ and $\ell_v$ be the number
	of loops with $v\in\overline{e}$. Then the number $b_v+2\ell_v$
	is called the \emph{valency} at $v$.
  \item A vertex $v$ is called a \emph{leaf} if it has valency $1$.
  \item A graph $\Gamma$ is said to be \emph{finite} if the numbers of
	vertices and edges are finite.
 \end{itemize}
\end{definition}

For graphs, Corollary
\ref{stratification_on_product_of_graphs} can be refined as follows. 

\begin{lemma}
 \label{product_of_total_normality}
 For graphs $\Gamma_1$ and $\Gamma_2$, the product
 $\Gamma_1\times \Gamma_2$ is a totally normal cellular stratified
 space. 
\end{lemma}

\begin{proof}
 We have verified in Corollary \ref{stratification_on_product_of_graphs}
 that $\Gamma_1\times\Gamma_2$ has a cellular 
 stratification under the product stratification. Let us verify that
 this is totally normal. For simplicity, we regard domains for
 characteristic maps of $2$-cells in $\Gamma_1\times\Gamma_2$ as
 stratified subspaces of $[-1,1]^2$ instead of $D^2$.

 There are three types of domains of $1$-cells in $\Gamma_1$ or
 $\Gamma_2$ up to homeomorphisms, i.e.\ $(-1,1)$, $(-1,1]$, or
 $[-1,1]$.
 The possible types of domains for $2$-cells $e_{\mu}\times e_{\lambda}$
 in $\Gamma_1\times\Gamma_2$ 
 are depicted as follows. 
 \begin{center}
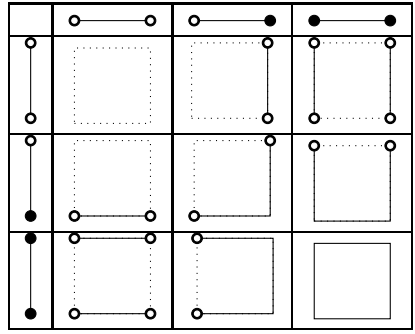
  \begin{tabular}{|c|c|c|c|} \hline
   & 
   \begin{tikzpicture}
    \draw (0,0) -- (1,0);
    \draw [fill] (0,0) circle (2pt);
    \draw [fill,white] (0,0) circle (1pt);
    \draw [fill] (1,0) circle (2pt);
    \draw [fill,white] (1,0) circle (1pt);    
   \end{tikzpicture} 
   &
   \begin{tikzpicture}
    \draw (0,0) -- (1,0);
    \draw [fill] (0,0) circle (2pt);
    \draw [fill,white] (0,0) circle (1pt);
    \draw [fill] (1,0) circle (2pt);
   \end{tikzpicture} 
   & 
   \begin{tikzpicture}
    \draw (0,0) -- (1,0);
    \draw [fill] (0,0) circle (2pt);
    \draw [fill] (1,0) circle (2pt);
   \end{tikzpicture} 
   \\ \hline
   \begin{tikzpicture}
    \draw (0,0) -- (0,1);
    \draw [fill] (0,0) circle (2pt);
    \draw [fill,white] (0,0) circle (1pt);
    \draw [fill] (0,1) circle (2pt);
    \draw [fill,white] (0,1) circle (1pt);    
   \end{tikzpicture} 
   & 
       \begin{tikzpicture}
	\draw [dotted] (0,0) rectangle (1,1);
       \end{tikzpicture}
       & 
	   \begin{tikzpicture}
	    \draw [dotted] (0,0) rectangle (1,1);
	    \draw (1,0) -- (1,1);
	    \draw [fill] (1,0) circle (2pt);
	    \draw [fill,white] (1,0) circle (1pt);
	    \draw [fill] (1,1) circle (2pt);
	    \draw [fill,white] (1,1) circle (1pt);
	   \end{tikzpicture}
	   & 
	       \begin{tikzpicture}
		\draw [dotted] (0,0) rectangle (1,1);
		\draw (0,0) -- (0,1);
		\draw [fill] (0,0) circle (2pt);
		\draw [fill,white] (0,0) circle (1pt);
		\draw [fill] (0,1) circle (2pt);
		\draw [fill,white] (0,1) circle (1pt);
		\draw (1,0) -- (1,1);
		\draw [fill] (1,0) circle (2pt);
		\draw [fill,white] (1,0) circle (1pt);
		\draw [fill] (1,1) circle (2pt);
		\draw [fill,white] (1,1) circle (1pt);
	       \end{tikzpicture}
	       \\ \hline
   \begin{tikzpicture}
    \draw (0,0) -- (0,1);
    \draw [fill] (0,0) circle (2pt);
    \draw [fill] (0,1) circle (2pt);
    \draw [fill,white] (0,1) circle (1pt);
   \end{tikzpicture} 
   &
       \begin{tikzpicture}
	\draw [dotted] (0,0) rectangle (1,1);
	\draw (0,0) -- (1,0);
	\draw [fill] (0,0) circle (2pt);
	\draw [fill,white] (0,0) circle (1pt);
	\draw [fill] (1,0) circle (2pt);
	\draw [fill,white] (1,0) circle (1pt);
       \end{tikzpicture}
       &
       \begin{tikzpicture}
	\draw [dotted] (0,0) rectangle (1,1);
	\draw (0,0) -- (1,0) -- (1,1);
	\draw [fill] (0,0) circle (2pt);
	\draw [fill,white] (0,0) circle (1pt);
	\draw [fill] (1,1) circle (2pt);
	\draw [fill,white] (1,1) circle (1pt);
       \end{tikzpicture}
	   &
       \begin{tikzpicture}
	\draw [dotted] (0,0) rectangle (1,1);
	\draw (0,1) -- (0,0) -- (1,0) -- (1,1);
	\draw [fill] (0,1) circle (2pt);
	\draw [fill,white] (0,1) circle (1pt);
	\draw [fill] (1,1) circle (2pt);
	\draw [fill,white] (1,1) circle (1pt);
       \end{tikzpicture}
	       \\ \hline
   \begin{tikzpicture}
    \draw (0,0) -- (0,1);
    \draw [fill] (0,0) circle (2pt);
    \draw [fill] (0,1) circle (2pt);
   \end{tikzpicture} 
   & 
       \begin{tikzpicture}
	\draw [dotted] (0,0) rectangle (1,1);
	\draw (0,0) -- (1,0);
	\draw [fill] (0,0) circle (2pt);
	\draw [fill,white] (0,0) circle (1pt);
	\draw [fill] (1,0) circle (2pt);
	\draw [fill,white] (1,0) circle (1pt);
	\draw (0,1) -- (1,1);
	\draw [fill] (0,1) circle (2pt);
	\draw [fill,white] (0,1) circle (1pt);
	\draw [fill] (1,1) circle (2pt);
	\draw [fill,white] (1,1) circle (1pt);
       \end{tikzpicture}
       &
       \begin{tikzpicture}
	\draw [dotted] (0,0) rectangle (1,1);
	\draw (0,0) -- (1,0) -- (1,1) -- (0,1);
	\draw [fill] (0,0) circle (2pt);
	\draw [fill,white] (0,0) circle (1pt);
	\draw [fill] (0,1) circle (2pt);
	\draw [fill,white] (0,1) circle (1pt);
       \end{tikzpicture}
	   &
       \begin{tikzpicture}
	\draw (0,0) rectangle (1,1);
       \end{tikzpicture}
	       \\ \hline
  \end{tabular}
 \end{center}

 In any of these cases, the boundary is a stratified subspace of the
 standard cell decomposition of a square and characteristic maps for $0$
 and $1$ dimensional cells lifts to maps into those boundaries. 
\end{proof}

\begin{remark}
 More generally, a $k$-fold product
 $\Gamma_1\times\cdots\times\Gamma_k$
 of graphs is totally normal. 
 For higher dimensional cellular stratified spaces, see \S3.2 of
 \cite{1111.4774}.  
\end{remark}

For totally normal cellular stratified spaces, Proposition
\ref{cellular_subdivision_is_cellular_stratified} can be 
refined as follows.

\begin{proposition}
 \label{subdivision_of_totally_normal}
 Let $(X,\pi,\Phi)$ be a totally normal cellular stratified space and
 $(X,\pi',\Phi')$ a cellular subdivision.
 Suppose that each morphism $b \in C(X)(e_{\mu},e_{\lambda})$ in the
 face category of $(X,\pi,\Phi)$ is a strict morphism
 \[
 b : (D_{\mu},\pi_{\mu},\Phi_{\mu}) \longrightarrow
 (D_{\lambda},\pi_{\lambda},\Phi_{\lambda}) 
 \]
 of stratified spaces. Then $(X,\pi',\Phi')$ has a structure of totally
 normal cellular stratified space.
\end{proposition}

\begin{proof}
 By the very definition of cellular subdivision, the boundary
 $\partial D_{\lambda}$ of the domain of each cell $e_{\lambda'}$ in
 $(X,\pi',\Phi')$ is equipped with a regular cellular stratification. It
 remains to show that, 
 for each $\lambda' \in P(X,\pi')$ and a cell $e'$ in
 $\partial D_{\lambda'}$, there exist a cell $e_{\mu'}$ in
 $(X,\pi',\Phi')$ and a map
 \[
  b' : D_{\mu'} \rarrow{} D_{\lambda'}
 \]
 making the diagram
 \[
  \begin{diagram}
   \node{D_{\lambda'}} \arrow{e,t}{\varphi_{\lambda'}}
   \node{\overline{e_{\lambda'}}} 
   \arrow{e,J} \node{X} \\
   \node{D_{\mu'}} \arrow{e,t}{\varphi_{\mu'}} \arrow{n,l,..}{b'}
   \node{\overline{e_{\mu'}}} 
   \arrow{ne,J} 
  \end{diagram}
 \]
 commutative and satisfying $b'(\Int(D_{\mu'}))=e'$.

 Suppose $s(\lambda')=\lambda$ under the subdivision
 $s: P(X,\pi')\to P(X,\pi)$. Let $e$ be a cell in $D_{\lambda}$
 containing $e'$. By the total normality of $(X,\pi,\Phi)$, there exists
 a cell $e_{\mu}$ in $(X,\pi,\Phi)$ and a map
 $b : D_{\mu}\to D_{\lambda}$ with $b(\Int(D_{\mu}))=e$ and
 $\varphi_{\lambda}\circ b=\varphi_{\mu}$.
 Define $e_{\mu'}=(\varphi_{\lambda}\circ b)(e')$. Since both $b$  and
 $\varphi_{\lambda}$ are strict morphisms of stratified spaces,
 $e_{\mu'}$ is a cell in $(X,\pi',\Phi')$.
 By the definition of cellular subdivision, there exist cells
 $\psi_{\mu'}:D_{\mu'}\to\overline{E_{\mu'}}$ 
 and $\psi_{\lambda'}:D_{\lambda'}\to\overline{E_{\lambda'}}$ in
 $(D_{\mu},\pi_{\mu},\Phi_{\mu})$ and 
 $(D_{\lambda},\pi_{\lambda},\Phi_{\lambda})$, respectively, such that 
 $\varphi_{\mu'}=\varphi_{\mu}\circ\psi_{\mu'}$ and
 $\varphi_{\lambda'}=\varphi_{\lambda}\circ\psi_{\lambda'}$. 

 When $\mu=\lambda$, both $E_{\mu'}$ and $E_{\lambda'}$
 are cells in the regular cellular stratification of $D_{\lambda}$ and $b$
 is the identity map. Hence there exists a unique map
 $b': D_{\mu'}\to D_{\lambda}'$ satisfying the 
 required conditions, since $\psi_{\mu'}$ and $\psi_{\lambda'}$ are
 embeddings. When $\mu<\lambda$, we have the following diagram 
 \[
  \begin{diagram}
   \node{} \node{D_{\lambda}} \arrow[2]{e,t}{\varphi_{\lambda}} \node{}
   \node{\overline{e_{\lambda}}} \\
   \node{D_{\lambda'}} \arrow{ne,t}{\psi_{\lambda'}}
   \arrow[2]{e,t}{\hspace*{20pt}\varphi_{\lambda'}} \node{} \arrow{n} 
   \node{\overline{e_{\lambda'}}} \arrow{ne,J} \\ 
   \node{} \node{D_{\mu}} \arrow{n,r,-}{b} \arrow{e,b,-}{\varphi_{\mu}}
   \node{} \arrow{e} 
   \node{\overline{e_{\mu}}.} \arrow[2]{n,J} \\
   \node{D_{\mu'}} \arrow[2]{n,l,..}{b'} \arrow{ne,t}{\psi_{\mu'}}
   \arrow[2]{e,b}{\varphi_{\mu'}} 
   \node{} 
   \node{\overline{e_{\mu'}}} \arrow[2]{n,J} \arrow{ne,J}
  \end{diagram}
 \]
 The regularity of cellular stratifications on $D_{\mu}$ and
 $D_{\lambda}$ and the fact that $b$ is an embedding implies that there
 exists a map $b':D_{\mu'}\to D_{\lambda'}$ making the above diagram
 commutative. And thus $(X,\pi',\Phi')$ is totally normal.
\end{proof}

\subsection{Face Categories of Totally Normal Cellular Stratified
  Spaces}
\label{face_category}

In this section, we show that the homotopy-theoretic informations of a
totally normal cellular stratified spaces $X$ are encoded in its face
category $C(X)$. A main tool is the classifying space functor
\[
 B : \category{Cats} \longrightarrow \category{Top}
\]
from the category of small categories to the category of topological
spaces defined in Definition \ref{classifying_space_definition}. 

Let us begin with the following description.

\begin{definition}
 For a totally normal cellular stratified space $X$, define a functor
 \[
  D^{X} : C(X) \longrightarrow \category{Top}
 \]
 by assigning $D_{\lambda}$ to each cell
 $\varphi_{\lambda}:D_{\lambda}\to\overline{e_{\lambda}}$. For a
 morphism $b\in C(X)(e_{\mu},e_{\lambda})$, define $D^{X}(b)=b$.
\end{definition}

\begin{proposition}
 \label{cellular_stratified_space_as_colim}
 When $X$ is a CW totally normal cellular stratified space, we have a
 natural homeomorphism
 \[
  \colim_{C(X)} D^{X} \rarrow{\cong} X.
 \]
\end{proposition}

\begin{proof}
 Let $\sim_{c}$ be the defining relation of the colimit
 $\colim_{C(X)}D^{X}$, i.e.\
 \[
 \colim_{C(X)} D^{X} = \quotient{\coprod_{\lambda\in P(X)}
 D_{\lambda}}{\sim_{c}} = D(X)/_{\sim_{c}}. 
 \]
 On the other hand, we have 
 \[
  X \cong D(X)/_{\sim_{\Phi}}
 \]
 by Lemma \ref{cellular_stratified_space_as_quotient}. 
 Let us verify that these two equivalence relations coincide.

 Suppose $x\sim_{c} y$ for $x\in D_{\mu}$ and
 $y\in D_{\lambda}$. Without loss of generality, we may assume 
 $b(x)=y$ for some $b\in C(X)(e_{\mu},e_{\lambda})$. We have
 $\varphi_{\mu}(x)=\varphi_{\lambda}(y)$ since
 $\varphi_{\mu}=\varphi_{\lambda}\circ b$.

 Suppose $\varphi_{\mu}(x)=\varphi_{\lambda}(y)$. There are three cases:
 \begin{enumerate}
  \item $x\in \Int D_{\mu}$ and $y\in \Int D_{\lambda}$.
  \item $x\in\Int D_{\mu}$ and $y\in \partial D_{\lambda}$ (or
	$x\in \partial D_{\mu}$ and $y\in \Int D_{\lambda}$).
  \item $x\in \partial D_{\mu}$ and $y\in \partial D_{\lambda}$.
 \end{enumerate}
 In the first case, $x=y$ and thus $x\sim_c y$. 

 In the second case, 
 $\varphi_{\mu}(x)\in e_{\mu}$,
 $\varphi_{\lambda}(y)\in\partial e_{\lambda}$, and
 $\varphi_{\mu}(x)=\varphi_{\lambda}(y)$. Thus we have
 $e_{\mu}\subset\partial e_{\lambda}$ and
 $\varphi_{\lambda}(y)\in e_{\mu}$. Choose a cell $e$ in
 $\partial D_{\lambda}$ with $y\in e$. By the total normality, there
 exists a cell $e_{\nu}$ in $X$ and a map $b:D_{\nu}\to \overline{e}$
 with $\varphi_{\nu}=\varphi_{\lambda}\circ b$. Since $b$ is a
 characteristic map, there exists a unique $z\in \Int(D_{\nu})$ such
 that $b(z)=y$. Then
 $\varphi_{\mu}(x)=\varphi_{\lambda}(y)=\varphi_{\nu}(z)$. Since both
 $x$ and $z$ lie in the interiors of domains of characteristic maps,
 $x=z$. We have $z\sim_{c}y$ by $b(z)=y$. Thus $x\sim_{c}y$.

 In the third case, the normality of the stratification of $X$ implies
 that there exists a cell $e_{\nu}$ such that
 $\varphi_{\mu}(x)=\varphi_{\lambda}(y)\in e_{\nu}$. By the second case,
 we obtain $x\sim_{c} y$.
\end{proof}

One of the most important features of totally normal cellular stratified
spaces is that the cellular stratification of each domain $D_{\lambda}$
can be described by using the comma category (Definition
\ref{comma_category_definition}) 
$C(X)_{\le e_{\lambda}}=C(X)\downarrow e_{\lambda}$. 
For simplicity, let us denote $C(X)_{\le\lambda}=C(X)_{\le e_{\lambda}}$
and $C(X)_{<\lambda}=C(X)_{<e_{\lambda}}$.
We obtain the following description of
$D_{\lambda}$ and $\partial D_{\lambda}$ as a corollary to Proposition
\ref{cellular_stratified_space_as_colim}. 

\begin{corollary}
 Let $X$ be a totally normal cellular stratified space. For a cell
 $\varphi_{\lambda} : D_{\lambda}\to \overline{e_{\lambda}}$ in $X$, 
 define a functor
 \[
  D^{X}_{\le\lambda} : C(X)_{\le\lambda} \longrightarrow
 \category{Top} 
 \]
 by $D^{X}_{\le\lambda}(u)=D_{\mu}$ for an object
 $u: D_{\mu}\to D_{\lambda}$ in $C(X)_{\le\lambda}$.
 Then we have a natural homeomorphism
 \[
 \colim_{C(X)_{\le\lambda}} D^{X}_{\le\lambda}
 \rarrow{\cong} D_{\lambda}.
 \]
 Define $D^{X}_{<\lambda}=D^{X}_{\le\lambda}|_{C(X)_{<\lambda}}$. Then
 the above homeomorphism induces a homeomorphism
 \[
 \colim_{C(X)_{<\lambda}} D^{X}_{<\lambda} 
 \rarrow{\cong} \partial D_{\lambda}.
 \] 
\end{corollary}

\begin{proof}
 By the definition of totally normal cellular stratified spaces, there
 is a one to one correspondence between cells in $D_{\lambda}$
 and morphisms $e_{\mu}\to e_{\lambda}$ in $C(X)$. Thus the comma
 category $C(X)_{\le\lambda}$ is isomorphic to the face
 category $C(D_{\lambda})$ and the functor
 $D^{X}_{\le\lambda}$ can be identified with
 $D^{D_{\lambda}}$. And the result follows from Proposition
 \ref{cellular_stratified_space_as_colim}. By removing
 $1_{D_{\lambda}}$, we obtain a homeomorphism
 $\displaystyle \colim_{C(X)_{<\lambda}}D^{X}_{<\lambda}\cong\partial D_{\lambda}$. 
\end{proof}

We use the following notation and terminology for the classifying space
of the face category of a cellular stratified space.

\begin{definition}
 For a cellular stratified space $X$, the classifying space $BC(X)$ of
 the face category $C(X)$ is called the \emph{barycentric subdivision}
 of $X$ and is denoted by $\Sd(X)$.
\end{definition}

When $X$ is a regular cell complex, $\Sd(X)$ coincides with the usual
barycentric subdivision of $X$, hence is homeomorphic to $X$.
In general, however, $\Sd(X)$ is smaller than $X$. A precise relation
between $X$ and $\Sd(X)$ is given by the following theorem, which is the
main tool in this paper.

\begin{theorem}
 \label{barycentric_subdivision_of_cellular_stratification}
 For a CW totally normal cellular stratified space $X$, the barycentric 
 subdivision $\Sd(X)$ of $X$ can be embedded into $X$ as a strong
 deformation retract. When $X$ is a CW complex, the embedding is a
 homeomorphism. 

 Furthermore the embeddings and homotopies can be chosen to be
 natural with respect to morphisms of cellular stratified spaces.   
\end{theorem}

The rest of this section is devoted to a proof of this Theorem.
We first need to construct an embedding $i_{X}:\Sd(X)\hookrightarrow X$.

\begin{proposition}
 \label{embedding_of_Sd}
 For a CW totally normal cellular stratified space $X$, there exists an
 embedding 
 \[
  i_{X} : \Sd(X) \hookrightarrow X
 \]
 which is natural with respect to strict morphisms of cellular stratified 
 spaces. 
\end{proposition}

\begin{proof}
 By Lemma \ref{BC_by_nondegenerate_chains}, $\Sd(X)=BC(X)$ is the
 quotient of 
 $\coprod_{n\ge 0} \overline{N}_n(C(X))\times\Delta^n$ under the
 equivalence relation generated by
 (\ref{relation_for_geometric_realization}). Thus it suffices to
 construct maps
 \[
  i_n : \overline{N}_n(C(X))\times \Delta^n \longrightarrow X
 \]
 making the following diagram commutative
 \begin{equation}
  \begin{diagram}
   \node{\overline{N}_n(C(X))\times\Delta^{n-1}}
   \arrow{e,t}{1\times d^i} \arrow{s,l}{d_i\times 1}
   \node{\overline{N}_n(C(X))\times\Delta^n} \arrow{s,r}{i_n} \\
   \node{\overline{N}_{n-1}(C(X))\times\Delta^{n-1}}
   \arrow{e,b}{i_{n-1}} \node{X,} 
  \end{diagram}
  \label{boundary_relation}
 \end{equation}
 where $d^i(t_0,\ldots,t_{n-1})=(t_0,\ldots,t_{i-1},0,t_i,\ldots,t_{n-1})$.
 Consider the space
 $D(X)=\coprod_{\lambda}\{e_{\lambda}\}\times D_{\lambda}$ 
 and a map $\widetilde{\Phi} : D(X) \to X$ in Proposition
 \ref{cellular_stratified_space_as_quotient}. 
 We construct an embedding
 \[
  z_n : \overline{N}_n(C(X))\times\Delta^n \longrightarrow D(X)
 \]
 and define $i_n=\widetilde{\Phi}\circ z_n$.
 Note that we have a decomposition
 \[
  \overline{N}_n(C(X)) = \coprod_{\bm{e}\in\overline{N}_{n}(P(X))}
 \{\bm{e}\}\times \overline{N}(\pi)^{-1}_n(\bm{e}),
 \]
 where $\overline{N}(\pi)_n : \overline{N}_n(C(X))\to\overline{N}_n(P(X))$
 is the map induced by the canonical projection $\pi : C(X)\to P(X)$.
 Thus it suffices to construct an embedding
 \[
 z_{\bm{e}} : \overline{N}(\pi)^{-1}_n(\bm{e}) \longrightarrow
 D_{\lambda_n} 
 \]
 for each
 $\bm{e}=(\lambda_n>\lambda_{n-1}>\cdots>\lambda_0)\in\overline{N}_n(P(X))$. 

 The embedding $z_{\bm{e}}$ is constructed by induction on $n$. When $n=0$,
 $\overline{N}_0(C(X))=C(X)_0 \cong P(X)$. For each $\lambda\in P(X)$,
 define $z_{e_{\lambda}}(e_{\lambda},\ast) = \varphi_{\lambda}(0)$, where
 $\Delta^0=\{\ast\}$ and $0\in D_{\lambda}$ is the origin. Suppose we
 have constructed $z_{\bm{e}}$ for $\bm{e}\in\overline{N}_{n}(P(X))$
 with $n\le k-1$. For
 $\bm{e}=(\lambda_k>\cdots>\lambda_0)\in \overline{N}_k(P(X))$, we have
 \[
  z_{d_{k}(\bm{e})} :
 \overline{N}(\pi)^{-1}_{k-1}(d_k(\bm{e}))\times\Delta^{k-1}
 \longrightarrow D_{\lambda_{k-1}}
 \]
 by the inductive assumption. Note that we have a decomposition
 \begin{equation}
  \overline{N}(\pi)^{-1}_k(\bm{e}) =
 C(X)(e_{\lambda_{k-1}},e_{\lambda_k})\times
 \overline{N}(\pi)^{-1}_{k-1}(d_k(\bm{e})). 
 \label{decomposition_of_nerve}
 \end{equation}
 We have a map
 $\tilde{z}_{\bm{e}}:\overline{N}_k^{-1}(\bm{e})\times\Delta^{k-1}\to \partial D_{\lambda_k}$
 defined by the composition
\begin{eqnarray*}
  \overline{N}(\pi)^{-1}_k(\bm{e})\times\Delta^{k-1} & = &
 C(X)(e_{\lambda_{k-1}},e_{\lambda_k})\times
 \overline{N}(\pi)^{-1}_{k-1}(d_k(\bm{e}))\times\Delta^k \\
 & \rarrow{1\times z_{d_k(\bm{e})}} &
  C(X)(e_{\lambda_{k-1}},e_{\lambda_k})\times 
  D_{\lambda_{k-1}} \\
 & \rarrow{b_{\lambda_{k-1},\lambda_k}} & \partial D_{\lambda_k},  
\end{eqnarray*}
 where 
\[
 b_{\lambda_{k-1},\lambda_k}:
 C(X)(e_{\lambda_{k-1}},e_{\lambda_k})\times D_{\lambda_{k-1}}
 \longrightarrow \partial D_{\lambda_k}\subset D_{\lambda_k}
\]
 is the adjoint of the inclusion
 $C(X)(e_{\lambda_{k-1}},e_{\lambda_k})\subset\Map(D_{\lambda_{k-1}},D_{\lambda_k})$. 
 Since $\Delta^k$ is the join of $\Delta^{k-1}$ and the $k$-th vertex
 $\bm{v}_k=(0,\ldots,0,1)$, the above map extends to
 \[
  z_{\bm{e}} : \overline{N}(\pi)_{k}^{-1}(\bm{e})\times\Delta^k =
 \overline{N}(\pi)_k^{-1}(\bm{e})\times \Delta^{k-1}\ast\bm{v}_k
 \rarrow{\tilde{z}_{\bm{e}}\ast 0} \partial D_{\lambda_k}\ast \{0\}
 \subset D_{\lambda_k}. 
 \]
 This completes the induction and we obtain maps $z_k$ for all $k$.

 Let us verify that these maps make the diagram
 (\ref{boundary_relation}) commutative. Under the decomposition
 (\ref{decomposition_of_nerve}), it 
 suffices to show the commutativity of the diagram
 \[
  \begin{diagram}
   \node{\overline{N}(\pi)_k^{-1}(\bm{e})\times\Delta^{k-1}}
   \arrow{e,t}{1\times d^i} \arrow{s,l}{d_i\times 1}
   \node{\overline{N}(\pi)_k^{-1}(\bm{e})\times\Delta^k}
   \arrow{e,t}{z_{\bm{e}}}  \node{D(X)} 
   \arrow{s,r}{\widetilde{\Phi}} \\ 
   \node{\overline{N}(\pi)_{k-1}^{-1}(d_i(\bm{e}))}
   \arrow{e,t}{z_{d_k(\bm{e})}} \node{D(X)}
   \arrow{e,t}{\widetilde{\Phi}} \node{X} 
  \end{diagram}
 \]
for each $\bm{e}=(\lambda_k>\cdots>\lambda_0)\in \overline{N}_k(P(X))$.

 When $0\le i<k$, the last element in $d_i(\bm{e})$ is also $\lambda_k$
 and the diagram reduces to
 \[
  \begin{diagram}
   \node{\overline{N}(\pi)_k^{-1}(\bm{e})\times\Delta^{k-1}}
   \arrow{e,t}{1\times d^i}
   \arrow{s,l}{d_i\times 1}
   \node{\overline{N}(\pi)_k^{-1}(\bm{e})\times\Delta^k}
   \arrow{s,r}{z_{\bm{e}}} \\  
   \node{\overline{N}(\pi)_{k-1}^{-1}(d_i(\bm{e}))\times D^{k-1}}
   \arrow{e,b}{z_{d_i(\bm{e})}} \node{D_{\lambda_k}.} 
  \end{diagram}
 \]
 The commutativity of this diagram follows from an easy diagram chasing 
 based on the inductive definition of $z_{\bm{e}}$.
 The case $i=k$ follows from the commutativity of the diagram
 \[
  \begin{diagram}
   \node{\overline{N}(\pi)_k^{-1}(\bm{e})\times\Delta^{k-1}}
   \arrow{e,t}{1\times d^k} \arrow{s,=}
   \node{\overline{N}(\pi)_k^{-1}(\bm{e})\times\Delta^{k}}
   \arrow[2]{s,r}{z_{\bm{e}}} \\   
   \node{C(X)(e_{\lambda_{k-1}},e_{\lambda_k})\times
   \overline{N}_{k-1}(d_k(\bm{e}))\times \Delta^{k-1}}
   \arrow{s,l}{1\times z_{d_k(\bm{e})}} \node{} \\
   \node{C(X)(e_{\lambda_{k-1}},e_{\lambda_{k}})\times
   D_{\lambda_{k-1}}} \arrow{e,b}{b_{\lambda_{k-1},\lambda_k}}
   \node{D_{\lambda_k}.} 
  \end{diagram}
 \]

 Now we have constructed a sequence of maps
 $i_n : \overline{N}_n(C(X))\times\Delta^n \to X$ compatible with the
 face relation. Let us denote the resulting continuous map by
 \[
  i_{X} : BC(X) \longrightarrow X.
 \]
 We only used structure maps of cellular stratified spaces and the
 origin in each $D_{\lambda}$ in the
 construction of $i_{X}$. Hence it is natural with respect to strict
 morphisms of cellular stratified spaces.

 Let us show that $i_{X}:BC(X)\to \Ima(i_{X})$ is a bijective closed
 map, hence is a homeomorphism. The surjectivity is obvious. The
 injectivity of $i_{X}$ can be proved inductively by using the
 definition of $z_{\bm{e}}$. It remains to show that $i_{X}$ is a closed
 map. By definition, we have the following commutative diagram
 \[
  \begin{diagram}
   \node{\coprod_{n,\bm{e}\in\overline{N}_n(P(X))}\{\bm{e}\}\times
   \overline{N}(\pi)_n^{-1}(\bm{e})\times\Delta^n} 
   \arrow{e,t}{\coprod z_{\bm{e}}} \arrow{s,l}{p} \node{D(X)}
   \arrow{s,r}{\widetilde{\Phi}} \\  
   \node{\Sd(X)} \arrow{e,b}{i_X} \node{X,}
  \end{diagram}
 \]
 where $p$ is the canonical projection. For each closed set
 $A\subset\Sd(X)$, it suffices to show that $\widetilde{\Phi}^{-1}(i_{X}(A))$ is
 closed in $D(X)$, since $\widetilde{\Phi}$ is a quotient map by Lemma
 \ref{cellular_stratified_space_as_quotient}. Note that $\coprod
 z_{\bm{e}}$ is a closed map, 
 since it is a disjoint union of
 \begin{equation}
  \coprod_{t(\bm{e})=\lambda}
 \{\bm{e}\}\times\overline{N}_n^{-1}(\bm{e})\times\Delta^n
 \longrightarrow \{e_{\lambda}\}\times D_{\lambda}
 \label{cellwise_embedding}
 \end{equation}
 that are continuous maps from compact sets to Hausdorff spaces. The
 compactness of the domain of the above map comes from the finiteness of
 the number of cells in $\partial D_{\lambda}$.

 We claim that
 \begin{equation}
  \widetilde{\Phi}^{-1}(i_{X}(A))= \left(\coprod z_{\bm{e}}\right)(f^{-1}(A)).  
   \label{compatible_inverse_images}
 \end{equation}
 Once this is shown, the proof is complete. The commutativity of the
 above diagram implies that
 $\widetilde{\Phi}^{-1}(i_{X}(A))\subset\left(\coprod z_{\bm{e}}\right)(f^{-1}(A))$. 
 On the other hand, for $x\in \widetilde{\Phi}^{-1}(i_{X}(A))$, there exists
 $a\in A$ such that $\widetilde{\Phi}(x)=i_{X}(a)$. 

 If $x\in \Int(D_{\lambda})$ and
 $a=[\bm{p},\bm{t}]$, 
 \[
  \varphi_{\lambda}(x)=i_{X}(a)=i_{X}(p(\bm{p},\bm{t})) =
 \widetilde{\Phi}(z_{\pi(\bm{p})}(\bm{p},\bm{t})). 
 \]
 Since $x\in\Int(D_{\lambda})$, $\bm{t}$ is of the form
 $\bm{t}=(1-t)\bm{s}+t\bm{v}_k$ for some $0<t\le 1$ and
 $\bm{s}\in\Delta^{k-1}$. This implies that
 $z_{\pi(\bm{p})}(\bm{p},\bm{t})\in \Int(D_{\lambda})$ and
 $x=z_{\pi(\bm{p})}(\bm{p},\bm{t})$. And we have
 $x\in\widetilde{\Phi}^{-1}(i_{X}(A))$. 
 
 When $x\in \partial D_{\lambda}$, let $e$ be a cell in
 $\partial D_{\lambda}$ with $x\in e$. By the total normality, there
 exists a cell $e_{\mu}$ in $X$ and a map $b$ making the following
 diagram commutative
 \[
 \begin{diagram}
  \node{\overline{e}} \arrow{e,J} \node{\partial D_{\lambda}}
 \arrow{e,J} \node{D_{\lambda}} \arrow{e,t}{\varphi_{\lambda}}
 \node{X.} \\
 \node{D} \arrow{n,l}{b} \arrow{e,=} \node{D_{\mu}}
 \arrow{ene,b}{\varphi_{\mu}} 
 \end{diagram}
 \]
 Let $y\in \Int(D)$
 be the unique element with $b(y)=x$. By the commutativity of the
 above diagram, 
 $\varphi_{\mu}(y)=\varphi_{\lambda}(x)\in i_{X}(A)$. Since
 $y\in\Int(D_{\mu})$, the previous argument implies that there exists
 $(\bm{p}',\bm{t}')\in p^{-1}(A)$ with
 $y=z_{\pi(\bm{p}')}(\bm{p}',\bm{t}')$. Thus we have
 $x=b(z_{\pi(\bm{p}')})(\bm{p}',\bm{t}')$.  
 Now define $\bm{p}$ to be the element of $\overline{N}(C(X))$ obtained
 by adjoining $b$ to $\bm{p}'$ as follows
 \[
 \bm{p}:\underbrace{D_{\lambda_0} \to \cdots \to
 D_{\lambda_k}}_{\bm{p}'} \rarrow{b} D_{\lambda} 
 \]
 and define $\bm{t}=d^k(\bm{t}')$. Then
 \begin{eqnarray*}
  p(\bm{p},\bm{t}) & = & [\bm{p},d^k(\bm{t}')] \\
  & = & [d_k(\bm{p}),\bm{t}'] \\
  & = & [\bm{p}',\bm{t}'] \\
  & = & p(\bm{p}',\bm{t}') \in A.
 \end{eqnarray*}
 Thus $(\bm{p},\bm{t})\in p^{-1}(A)$. This completes the proof.

In the above proof, the image of
$\coprod_{t(\bm{e})=\lambda}\{\bm{e}\}\times\overline{N}_n^{-1}(\bm{e})\times\Delta^n$
in $\Sd(X)$ via $p$ can be identified with $\Sd(D_{\lambda})$ and the
map (\ref{cellwise_embedding}) induces an embedding
\[
 i_{\lambda} : \Sd(D_{\lambda}) \hookrightarrow D_{\lambda}.
\]
 Thus the embedding $i_{X}$ can be constructed by gluing embeddings
 $i_{\lambda}$ together.
 When $X$ is a cell complex, all $D_{\lambda}$ are regular cell
 complexes and the the embedding $i_{\lambda}$ is a homeomorphism. Thus
 $i_{X}$ is a homeomorphism.
\end{proof}

\begin{remark}
 \label{from_natural_transformation_to_embedding}
 The above observation that $i_{X}$ is obtained by gluing
 $i_{\lambda}$ can be verified by using Proposition
 \ref{cellular_stratified_space_as_colim}. 
 In particular, define a functor
 \[
  \Sd(D^{X}) : C(X) \longrightarrow \category{Top}
 \]
 by $\Sd(D^{X})(e_{\lambda})=\Sd(D_{\lambda})$ on objects. Then, the
 embeddings $i_{\lambda}$ give rise to a natural transformation
 \[
  i_{X} : \Sd(D^{X}) \Longrightarrow D^{X}
 \]
 which induces an embedding
 \[
 i_{X} : \Sd(X)=\colim_{C(X)}\Sd(D^{X}) \rarrow{\cong}
 \colim_{C(X)} D^{X}\cong X.
 \]
\end{remark}

Now we are ready to prove Theorem
\ref{barycentric_subdivision_of_cellular_stratification}. 

\begin{proof}[Proof of Theorem
 \ref{barycentric_subdivision_of_cellular_stratification}] 
 Let us show that the image of the embedding constructed in Proposition
 \ref{embedding_of_Sd} is a strong deformation retract of $X$. The idea
 is essentially the same as the proof of Proposition
 \ref{embedding_of_Sd}. Following Remark
 \ref{from_natural_transformation_to_embedding}, 
 we construct a homotopy
 \[
  H_{\lambda} : D_{\lambda}\times[0,1] \longrightarrow D_{\lambda}
 \]
 by gluing deformation retractions on the domain of each cell. 
 This can be done by appealing to Proposition
 \ref{cellular_stratified_space_as_colim}. 
 More precisely, define a functor
 \[
  D^{X\times[0,1]}: C(X) \longrightarrow \category{Top}
 \]
 by $D^{X\times[0,1]}(e_{\lambda})=D_{\lambda}\times [0,1]$. By applying
 Proposition \ref{cellular_stratified_space_as_colim} to $X\times[0,1]$,
 we also have $X\times[0,1]\cong \colim_{C(X)}D^{X\times[0,1]}$.
 If we have constructed a natural transformation
 \[
  H : D^{X\times[0,1]} \Longrightarrow D^{X},
 \]
 we would have a homotopy
 \[
   X\times[0,1]\cong \colim_{C(X)} D^{X\times[0,1]} \rarrow{\colim H}
 \colim_{C(X)} D^{X} \cong X.
 \]

 Thus all we have to do is to construct $H_{\lambda}$ for each
 $\lambda\in P(X)$
 satisfying the following conditions:
 \begin{enumerate}
  \item $H_{\lambda}$ is a strong deformation retraction of
	$D_{\lambda}$ onto
	$i_{D_{\lambda}}(\Sd(D_{\lambda}))$. 
  \item For each $\lambda\in P(X)$ and a cell $e$ in
	$\partial D_{\lambda}$, let $e_{\mu}$ be the cell in $X$
	corresponding to $e$ and $b:D_{\mu}\to D_{\lambda}$ be
	the left of $\varphi_{\mu}$. Then the following diagram is
	commutative
	\[
	 \begin{diagram}
	  \node{D_{\lambda}\times[0,1]} \arrow{e,t}{H_{\lambda}}
	  \node{D_{\lambda}} \\
	  \node{D_{\mu}\times[0,1]} \arrow{n,l}{b} \arrow{e,b}{H_{\mu}}
	  \node{D_{\mu}.} \arrow{n,r}{b}
	 \end{diagram}
	\]
 \end{enumerate}

 This is done by induction on the dimension of cells. When $e_{\lambda}$
 is a bottom cell, $\partial D_{\lambda}=\emptyset$ and
 $\Sd(D_{\lambda})$ is a single point and can be identified with the
 origin of $D_{\lambda}$. Define
 \[
  H_{\lambda}(x,t) = (1-t)x.
 \]
 Suppose that we have defined $H_{\mu}$ for all cells $e_{\mu}$ of
 dimension less than $k$. Let $e_{\lambda}$ be a $k$-dimensional
 cell. Define 
 \[
  h_{\lambda}=\colim_{C(X)_{<\lambda}} H_{\mu} : \partial
 D_{\lambda}\times[0,1]\cong \colim_{C(X)_{<\lambda}}
 D^{X\times[0,1]}_{<\lambda} \longrightarrow 
 \colim_{C(X)_{<\lambda}} D^{X}_{<\lambda} \cong \partial D_{\lambda}.
 \]
 Then this is a strong deformation retraction of $\partial D_{\lambda}$
 onto $i_{\partial D_{\lambda}}(\Sd(\partial D_{\lambda}))$. This
 strong deformation retraction can be extended to 
 \[
  H_{\lambda} : D_{\lambda}\times[0,1] \longrightarrow D_{\lambda}
 \]
 by using Theorem \ref{spherical_case} below.
 The construction by colimit implies that the following diagram is
 commutative
 \[
  \begin{diagram}
   \node{D_{\mu}\times[0,1]} \arrow{e,t}{H_{\mu}}
   \arrow{s,l}{b\times 1_{[0,1]}} 
   \node{D_{\mu}} \arrow{s,r}{b} \\
   \node{\partial D_{\lambda}\times[0,1]} \arrow{s,J}
   \arrow{e,t}{h_{\lambda}} \node{\partial D_{\lambda}} \arrow{s,J} \\ 
   \node{D_{\lambda}\times[0,1]} \arrow{e,b}{H_{\lambda}}
   \node{D_{\lambda}.} 
  \end{diagram}
 \]
\end{proof}

\begin{theorem}
 \label{spherical_case}
 Let $\pi$ be a regular cell decomposition of $S^{n-1}$ and
 $L \subset S^{n-1}$ be a stratified subspace. Let $\tilde{\pi}$ be the
 cellular stratification on $K=\Int D^n\cup L$ obtained by adding
 $\Int D^n$ as an $n$-cell.  
 Then there exits a deformation retraction $H$ of $K$ to
 $i_{K}(\Sd(K,\tilde{\pi}))$. Furthermore if a deformation
 retraction $h$ of $L$ 
 onto $i_{L}(\Sd(L))$ is given, $H$ can be taken to be an extension of
 $h$. 
\end{theorem}

A proof of this fact is given in Appendix \ref{deformation_retraction}.

\section{Acyclic Category Models for Configuration Spaces of Graphs}
\label{model_for_graph}

In this section, we construct a combinatorial model
$C_k^{\comp}(X)$ of the configuration space $\Conf_k(X)$ of a
graph $X$, by using the braid arrangements. Our model has the following  
advantages compared to Abrams' model: 
\begin{itemize}
 \item $C_k^{\comp}(X)$ is always homotopy equivalent to the
       configuration space $\Conf_k(X)$. And the homotopy is explicitly
       given. 
 \item The construction of $C_k^{\comp}(X)$, the embedding
       $i:C_k^{\comp}(X)\hookrightarrow \Conf_k(X)$, and the deformation
       retraction, are functorial with respect to strict morphisms of
       cellular stratified spaces. Hence $C_k^{\comp}(X)$
       inherits the action of $\Sigma_k$.
 \item $C_k^{\comp}(X)$ is often much smaller than Abrams' model.
\end{itemize}

A couple of sample applications will be given in
\S\ref{application_of_model_for_graph}. 

\subsection{The Braid Stratification}
\label{braid_stratification_for_graph}

Given a graph $X$, our strategy to
construct a combinatorial model for $\Conf_k(X)$ is
to define a suitable $\Sigma_k$-equivariant cellular 
stratification on the $k$-fold product $X^k$ including the discriminant
$\Delta_k(X)$ as a stratified subspace. Then Theorem
\ref{barycentric_subdivision_of_cellular_stratification} gives us a
model when applied to the induced stratification on $\Conf_k(X)$.
In order to obtain a model with the lowest possible dimension, we would
like the subdivision to be as coarse as possible, which suggests us to
use the braid arrangements. 

Let us first recall the stratification on $\R^n$ induced by a hyperplane 
arrangement.

\begin{definition}
 Let $\mathcal{A}=\{H_1,\cdots,H_{k}\}$ be a hyperplane arrangement
 in $\R^n$ defined by a collection of affine $1$-forms
 $L=\{\ell_i :\R^n\to\R\}_{i=1,\cdots,k}$.
 The evaluation on points in $\R^n$ defines map
 \[
  \ev : \R^n \longrightarrow \Map(L,\R) \cong \R^{k}. 
 \]
 Composed with the sign function
 \[
  \sign : \R \longrightarrow \{-1,0,+1\}=S_1,
 \]
 we obtain a map
 \[
  s_{\mathcal{A}} : \R^n \rarrow{\ev} \R^k \rarrow{\sign^k} S_1^k,
 \]
 where $S_1$ is regarded as a poset with ordering $0< \pm 1$.
 This is called the \emph{stratification on $\R^n$ determined by
 $\mathcal{A}$}. 
\end{definition}

\begin{lemma}
 The stratification $s_{\mathcal{A}}$ is a regular
 totally normal cellular stratification on $\R^n$ for any hyperplane
 arrangement $\mathcal{A}$ in $\R^n$. 
\end{lemma}

In this paper, we only make use of the braid arrangements.

\begin{definition}
 For $1\le i<j\le n$, define a hyperplane in $\R^n$ by
 \[
  H_{i,j}=\lset{(x_1,\cdots,x_n)\in\R^n}{x_i=x_j}.
 \]
 The hyperplane arrangement $\lset{H_{i,j}}{1\le i<j\le n}$ is called
 the \emph{braid arrangement} of rank $n-1$ and is denoted by
 $\mathcal{A}_{n-1}$. 
\end{definition}

The structure of cellular stratification $s_{\mathcal{A}_{n-1}}$ on
$\R^n$ defined by the braid arrangement is well-known.

\begin{lemma}
 \label{braid_arrangement_and_partitions}
 Cells in the cellular stratification $s_{\mathcal{A}_{n-1}}$ are in
 one-to-one correspondence with ordered partitions $\Pi_n$ of
 $\{1,\ldots,n\}$. 
\end{lemma}

\begin{proof}
 An $n$-cell in the stratification $s_{\mathcal{A}_{n-1}}$ is given by a 
 system of inequalities
 \[
  x_{\sigma(1)}<x_{\sigma(2)}<\cdots< x_{\sigma(n)}
 \]
 for a permutation $\sigma\in\Sigma_n$. Lower dimensional cells are
 given by replacing $<$ by $=$, hence they corresponds to partitions of
 the set $\{1,\ldots,n\}$.
\end{proof}

Now we are ready to introduce our stratification on $X^k$.
The starting point is the following observation, which is an immediate
generalization of Lemma \ref{product_of_total_normality}.

\begin{lemma}
 For any $1$-dimensional finite cellular stratified space $X$, the
 product $X^k$ is totally normal.  
\end{lemma}

\begin{definition}
 Let $X$ be a $1$-dimensional finite cellular stratified space. Define a 
 subdivision $\pi^{\braid}_{k,X}$ of the product stratification
 on $X^k$ as follows: Let
 $\{e_{\lambda}^0\}_{\lambda\in\Lambda_0}$ and 
 $\{e_{\lambda}^1\}_{\lambda\in\Lambda_1}$ be $0$-cells and $1$-cells of
 $X$, respectively. We choose linear orders of $\Lambda_0$ and
 $\Lambda_1$. For a cell 
 $e_{\lambda_1}^{\varepsilon_1}\times\cdots\times
 e_{\lambda_k}^{\varepsilon_k}$ in $X^k$, choose a permutation
 $\sigma\in\Sigma_k$ with 
 \begin{equation}
  (e_{\lambda_1}^{\varepsilon_1}\times\cdots\times
 e_{\lambda_k}^{\varepsilon_k})\sigma = (\text{a product of
 $0$-cells})\times 
 (e_{\mu_1}^{1})^{m_1}\times\cdots\times 
 (e_{\mu_{\ell}}^1)^{m_{\ell}} 
 \label{standard_form}
 \end{equation}
 and $\mu_1<\cdots<\mu_\ell$.

 By using the characteristic map
 \[
  \R \rarrow{\cong} \Int(D^1) \rarrow{\varphi_{\mu_{j}}} e_{\mu_j},
 \]
 we obtain a canonical homeomorphism
 \[
  (e_{\mu_j}^1)^{m_j} \cong \R^{m_j}.
 \]
 Subdivide each $(e_{\mu_j}^1)^{m_j}$ by the braid arrangement
 $\mathcal{A}_{m_j-1}$ under this identification. The resulting
 stratification on $X^k$ is denoted by $\pi_{k,X}^{\braid}$ 
 and is called the \emph{braid stratification} on $X^k$.
\end{definition}

Proposition \ref{subdivision_of_totally_normal} gives us the following
important fact.

\begin{proposition}
 \label{braid_on_Conf_is_totally_normal}
 The braid stratification $\pi^{\braid}_{k,X}$ on $X^k$ for a
 graph $X$ is totally normal
 and contains the discriminant $\Delta_k(X)$ as a stratified
 subspace. Hence the configuration space $\Conf_k(X)$ is also a totally
 normal cellular stratified subspace of $X^k$. 
\end{proposition}

\begin{proof}
 In the product stratification on $X^k$, we use (stratified subspaces
 of) cubes as domains of characteristic maps. The braid stratification
 on each domain cube induces a stratification on each face of
 codimension $1$. The restriction of the braid arrangement
 $\mathcal{A}_{n-1}$ to the hyperplane $x_n=1$ or $x_n=-1$ is the
 arrangement $A_{n-2}$ in $\R^{n-1}$. The same is true for other
 codimension $1$ faces $x_i=\pm 1$. 

 Since any morphism in the face category of the product
 stratification is a composition of inclusions of codimension $1$ faces,
 the condition of Proposition \ref{subdivision_of_totally_normal} is
 satisfied. 
\end{proof}


\subsection{A Combinatorial Model for Configuration Spaces of Graphs}
\label{category_model_for_graph}

By Proposition \ref{braid_on_Conf_is_totally_normal}, the restriction of
the braid stratification $\pi^{\braid}_{k,X}$ to the configuration space
$\Conf_k(X)$ gives rise to an acyclic category 
$C(\pi^{\braid}_{k,X}|_{\Conf_k(X)})$. 
Its classifying space $BC(\pi^{\braid}_{k,X}|_{\Conf_k(X)})$ is our
first model for $\Conf_k(X)$.

\begin{definition}
 For a $1$-dimensional finite cellular stratified space $X$, let us
 abbreviate the induced stratification
 $\pi^{\braid}_{k,X}|_{\Conf_k(X)}$ on $\Conf_k(X)$ by
 $\pi_{k,X}^{\comp}$.  
 Define a cell complex $C^{\comp}_k(X)$ by
 \[
  C^{\comp}_k(X) = BC(\pi_{k,X}^{\comp}).
 \]
\end{definition}

By Theorem \ref{barycentric_subdivision_of_cellular_stratification},
$C^{\comp}_k(X)$ is a cell complex model and $C(\pi_{k,X}^{\comp})$ is
an acyclic category model of the homotopy type of $\Conf_k(X)$. 

\begin{corollary}
 \label{embedding_and_deformation_retraction}
 For a finite graph $X$,
 $C^{\comp}_k(X)$ is embedded in $\Conf_k(X)$ as a
 $\Sigma_k$-equivariant strong deformation retract. 
\end{corollary}

There is an alternative description. We may define $C^{\comp}_k(X)$ as a
``cellular complement'' of the discriminant set in the barycentric
subdivision $\Sd(X^k,\pi_{k,X})$ of the braid 
stratification on $X^k$. More generally, we use the following
notation. 

\begin{definition}
 Let $X$ be a cellular stratified space and $A\subset X$ be a
 subset. Define a cellular stratified subspace $M(A;X)$ by
 \[
  M(A;X) = \bigcup_{e\in P(X),\overline{e}\cap A=\emptyset} e
 \]
 This is called the \emph{cellular complement} of $A$ in $X$.
\end{definition}

In the case of totally normal cell complexes, we have the following
description. 

\begin{lemma}
 Let $X$ be a totally normal cell complex and
 $A\subset X$ be a cellular stratified subspace. Then we have
 \[
  i_{X}(\Sd(A)) = M(X\setminus A;i_{X}(\Sd(X)))
 \]
 under the inclusion $\Sd(A)\hookrightarrow \Sd(X)$ induced by
 $A\hookrightarrow X$.
\end{lemma}

\begin{proof}
 Since $X$ is a cell complex, $i_{X}(\Sd(X))=X$ and $X\setminus A$
 can be regarded as a subset of $i_{X}(\Sd(X))$. By the definition of
 classifying space, cells in $\Sd(A)$ are in one-to-one correspondence 
 with elements in $\overline{N}_*(C(A))$. For
 $\bm{e}\in\overline{N}_*(C(A))$, 
 let us denote the corresponding cell by $\sigma(\bm{e})$. By the
 construction of $i_{X}$, the image
 $i_{X}(\sigma(\bm{e}))$ is given by choosing interior points in each cell
 in $\bm{e}: e_0\to\cdots\to e_n$ and then by ``connecting'' them. Thus
 \[
  i_{X}(\sigma(\bm{e})) \subset e_0\cup\cdots\cup e_n \subset A,
 \]
 or $i_{X}(\sigma(\bm{e}))\cap (X\setminus A)=\emptyset$.

 Conversely, if $i_{X}(\sigma(\bm{e}))$ is a cell in $i_{X}(\Sd(X))$
 with $i_{X}(\sigma(\bm{e}))\cap (X\setminus A)=\emptyset$ and
 $\bm{e}: e_0\to\cdots\to e_n$, all vertices of
 $i_{X}(\sigma(\bm{e}))$ should belong to $A$. Since vertices of
 $i_{X}(\sigma(\bm{e}))$ are in the interiors of cells
 $e_0,\ldots,e_n$. Thus these cells must be cells in $A$ and we have
 $i_{X}(\sigma(\bm{e})) \subset i(\Sd(A))$.
\end{proof}

Thus the image of the embedding of $C^{\comp}_k(X)$ can be described
as a cellular complement.

\begin{corollary}
 \label{cellular_model}
 Let $X$ be a finite graph. Under
 the braid stratification $\pi_{k,X}^{\braid}$ of $X^k$, we have 
 \[
  i_{X^k}(C^{\comp}_k(X)) =
 M\left(\Delta_k(X);\Sd\left(\pi_{k,X}^{\braid}\right)\right). 
 \]
 In other words, $i_{X^k}(C^{\comp}_k(X))$ consists of those cells in
 $\Sd(\pi_{k,X}^{\braid})$ whose closures do not touch the discriminant
 $\Delta_k(X)$. 
\end{corollary}

Let us take a look at a couple of examples.

\begin{example}
 Let us consider the minimal cell decomposition $S^1=e^0\cup e^1$ of the
 circle, regarded as a graph. The product stratification on
 $(S^1)^2$ is given by
 \[
  (S^1)^2 = e^0\times e^0 \cup e^0\times e^1\cup e^1\times e^0 \cup
 e^1\times e^1.
 \]
 The $2$-cell $e^1\times e^1$ is divided into a $1$-cell $e^1_{\Delta}$
 and two $2$-cells $e^2_{-}$ and $e^2_{+}$. The braid stratification and
 its barycentric subdivision
 can be depicted as follows:
 
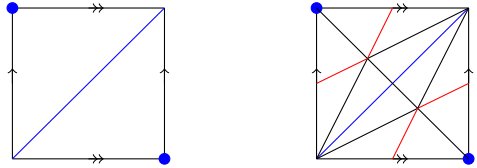
\begin{center}
  \begin{tikzpicture}
   \draw (-4,0) -- (-2,0);
   \draw [->>] (-3,0) -- (-2.8,0);
   \draw (-4,2) -- (-2,2);
   \draw [->>] (-3,2) -- (-2.8,2);
   \draw (-2,0) -- (-2,2);
   \draw [->] (-2,1) -- (-2,1.2);
   \draw (-4,0) -- (-4,2);
   \draw [->] (-4,1) -- (-4,1.2);
   \draw [blue] (-4,0) -- (-2,2);
   \draw [blue,fill] (-4,2) circle (2pt);
   \draw [blue,fill] (-2,0) circle (2pt);

   \draw (0,0) -- (0,2);
   \draw [->] (0,1) -- (0,1.2);
   \draw (2,0) -- (2,2);
   \draw [->] (2,1) -- (2,1.2);
   \draw (0,0) -- (2,0);
   \draw [->>] (1,0) -- (1.2,0);
   \draw (0,2) -- (2,2);
   \draw [->>] (1,2) -- (1.2,2);
   \draw [blue,fill] (0,2) circle (2pt);
   \draw [blue,fill] (2,0) circle (2pt);
   \draw [blue] (0,0) -- (2,2);
   \draw (0,2) -- (2,0);
   \draw (0,0) -- (0.67,1.33);
   \draw [red] (0.67,1.33) -- (1,2);
   \draw (0,0) -- (1.33,0.67);
   \draw [red] (1.33,0.67) -- (2,1);
   \draw (2,2) -- (0.67,1.33);
   \draw [red] (0.67,1.33) -- (0,1);
   \draw (2,2) -- (1.33,0.67);
   \draw [red] (1.33,0.67) -- (1,0);
  \end{tikzpicture}
 \end{center}
 The discriminant $\Delta_2(S^1)$ is colored by blue.

 By Corollary \ref{cellular_model}, we see that our model
 $C_2^{\comp}(S^1)$ consists of those cells whose closures do not touch
 the blue part, 
 namely the four red segments. It follows that $C_2^{\comp}(S^1)$ is
 isomorphic to the boundary of a square.

 It is left to the reader to verify that our model
 $C_2^{\comp}(\Sd(S^1))$ for the subdivision $\Sd(S^1)$ is a
 $2$-dimensional cell complex, which is homotopy equivalent to $S^1$.
\end{example}

As we have seen in the above example, our model (as well as Abrams' model)
gets fat as the cell decomposition becomes finer. In order to obtain a
good model, therefore, we need to reduce the number of cells. 

\begin{definition}
 We say a graph is \emph{reduced} if the only vertices with valency $2$
 are those contained in loops.
\end{definition}

By replacing a pair of $1$-cells $e$ and $e'$ sharing a
vertex $v$ by a $1$-cell $e\cup v \cup e'$, we can always make a graph
into reduced without changing the homeomorphism type. For example,
$\Sd(S^1)$ is not reduced but $S^1$ is.

The following example shows, however, the reducedness is not enough to
obtain the lowest possible dimension.

\begin{example}
 \label{two_points_in_Y}
 Consider a graph $Y$ of the following shape.
 
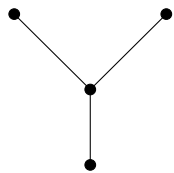
\begin{center}
  \begin{tikzpicture}
   \draw (0,0) -- (1,1);
   \draw (0,0) -- (-1,1);
   \draw (0,0) -- (0,-1);
   \draw [fill] (0,0) circle (2pt);
   \draw [fill] (1,1) circle (2pt);
   \draw [fill] (-1,1) circle (2pt);
   \draw [fill] (0,-1) circle (2pt);
  \end{tikzpicture}
 \end{center}
 By cutting out an edge, the product $Y\times Y$ can be developed
 into the following diagram:
 
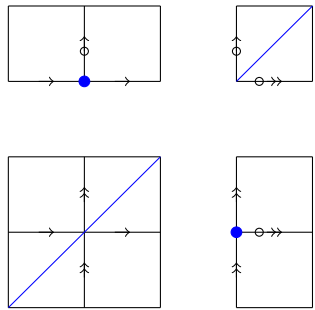
\begin{center}
  \begin{tikzpicture}
   \draw (0,0) -- (2,0);
   \draw (0,1) -- (2,1);
   \draw [->] (0.4,1) -- (0.6,1);
   \draw [->] (1.4,1) -- (1.6,1);
   \draw (0,2) -- (2,2);
   \draw (0,0) -- (0,2);
   \draw (1,0) -- (1,2);
   \draw [->>] (1,0.4) -- (1,0.6);
   \draw [->>] (1,1.4) -- (1,1.6);
   \draw (2,0) -- (2,2);

   \draw (0,3) -- (2,3);
   \draw [->] (0.4,3) -- (0.6,3);
   \draw [->] (1.4,3) -- (1.6,3);
   \draw (0,4) -- (2,4);
   \draw (0,3) -- (0,4);
   \draw (1,3) -- (1,4);
   \draw [->] (1,3.4) -- (1,3.6);
   \draw (1,3.4) circle (1.5pt);
   \draw (2,3) -- (2,4);

   \draw (3,0) -- (3,2);
   \draw [->>] (3,0.4) -- (3,0.6);
   \draw [->>] (3,1.4) -- (3,1.6);
   \draw (4,0) -- (4,2);
   \draw (3,0) -- (4,0);
   \draw (3,1) -- (4,1);
   \draw [->>] (3.4,1) -- (3.6,1);
   \draw (3.3,1) circle (1.5pt);
   \draw (3,2) -- (4,2);

   \draw (3,3) -- (3,4);
   \draw [->] (3,3.4) -- (3,3.6);
   \draw (3,3.4) circle (1.5pt);
   \draw (4,3) -- (4,4);
   \draw [->>] (3.4,3) -- (3.6,3);
   \draw (3.3,3) circle (1.5pt);
   \draw (3,3) -- (4,3);
   \draw (3,4) -- (4,4);

   \draw [blue] (0,0) -- (2,2);
   \draw [blue] (3,3) -- (4,4);
   \draw [fill,blue] (1,3) circle (2pt);
   \draw [fill,blue] (3,1) circle (2pt);
  \end{tikzpicture}
 \end{center}
 The discriminant $\Delta_2(Y)$ is drawn by blue. The barycentric
 subdivision of the braid stratification is obtained by taking the
 barycentric subdivisions of squares and triangles in the above
 figure. It is easy to see, by Corollary \ref{cellular_model}, that our
 model $C_2^{\comp}(Y)$ is a $2$-dimensional cell complex which is
 homotopy equivalent to $S^1$. The model $C_2^{\comp}(Y)$ is not optimal
 from the viewpoint of homotopy dimension.


 It is possible, however, to construct a $1$-dimensional subcomplex in
 $C_2^{\comp}(Y)$ which is a strong deformation retraction.
 This is the subject of the next section.
\end{example}



\subsection{Removing Leaves}
\label{no_leaf}

The first step to simplify the acyclic category model $C_k^{\comp}(X)$
is to remove leaves.

\begin{definition}
 For a graph $X$, the strict
 cellular stratified subspace obtained by removing all leaves from $X$
 is denoted by $X^{\circ}$.
\end{definition}

\begin{lemma}
 The inclusion $X^{\circ}\hookrightarrow X$ induces a
 $\Sigma_k$-equivariant homotopy
 equivalence $\Conf_k(X^{\circ})\simeq \Conf_{k}(X)$.
\end{lemma}

\begin{proof}
 Choose an embedding $X\hookrightarrow X^{\circ}$ by squeezing the
 lengths of edges having leaves to the halves. Then the composition
 $X\hookrightarrow X^{\circ} \hookrightarrow X$ is isotopic to the
 identity. The same is true for the composition
 $X^{\circ}\hookrightarrow X\hookrightarrow X^{\circ}$. Since $\Conf_k$
 is functorial with respect to embeddings, both compositions
 \begin{eqnarray*}
  \Conf_k(X)\hookrightarrow & \Conf_k(X^{\circ}) & \hookrightarrow
   \Conf_k(X) \\ 
  \Conf_k(X^{\circ})\hookrightarrow & \Conf_k(X) & \hookrightarrow
   \Conf_k(X^{\circ}) 
 \end{eqnarray*}
 are isotopic to identities. In particular, we have a homotopy
 equivalence $\Conf_k(X^{\circ})\simeq \Conf_k(X)$.
\end{proof}

\begin{corollary}
 \label{2nd_model_for_graph}
 For a finite graph $X$, 
 $C^{\comp}_k(X^{\circ})$ can be embedded in $\Conf_k(X)$ as a
 $\Sigma_k$-equivariant strong deformation retraction.
\end{corollary}

Thus the complex $C^{\comp}_k(X^{\circ})$ is a smaller model for
$\Conf_k(X)$. This model will be used in \S\ref{homotopy_dimension}. 

\begin{example}
 \label{two_points_in_Y_without_leaves}
 Consider the graph $Y$ in Example \ref{two_points_in_Y}.
 It has three leaves. Then $Y^{\circ}$ is a strict stratified subspace
 of $Y$ having only one vertex.

 Then $C_2^{\comp}(Y^{\circ})$ is a $1$-dimensional complex
 given by the green lines in the following figure:
 
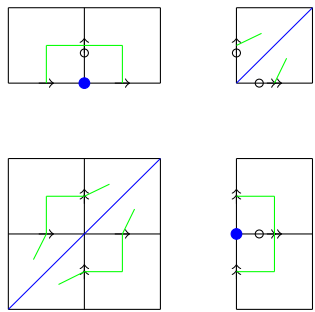
\begin{center}
  \begin{tikzpicture}
   \draw (0,0) -- (2,0);
   \draw (0,1) -- (2,1);
   \draw [->] (0.4,1) -- (0.6,1);
   \draw [->] (1.4,1) -- (1.6,1);
   \draw (0,2) -- (2,2);
   \draw (0,0) -- (0,2);
   \draw (1,0) -- (1,2);
   \draw [->>] (1,0.4) -- (1,0.6);
   \draw [->>] (1,1.4) -- (1,1.6);
   \draw (2,0) -- (2,2);

   \draw (0,3) -- (2,3);
   \draw [->] (0.4,3) -- (0.6,3);
   \draw [->] (1.4,3) -- (1.6,3);
   \draw (0,4) -- (2,4);
   \draw (0,3) -- (0,4);
   \draw (1,3) -- (1,4);
   \draw [->] (1,3.4) -- (1,3.6);
   \draw (1,3.4) circle (1.5pt);
   \draw (2,3) -- (2,4);

   \draw (3,0) -- (3,2);
   \draw [->>] (3,0.4) -- (3,0.6);
   \draw [->>] (3,1.4) -- (3,1.6);
   \draw (4,0) -- (4,2);
   \draw (3,0) -- (4,0);
   \draw (3,1) -- (4,1);
   \draw [->>] (3.4,1) -- (3.6,1);
   \draw (3.3,1) circle (1.5pt);
   \draw (3,2) -- (4,2);

   \draw (3,3) -- (3,4);
   \draw [->] (3,3.4) -- (3,3.6);
   \draw (3,3.4) circle (1.5pt);
   \draw (4,3) -- (4,4);
   \draw [->>] (3.4,3) -- (3.6,3);
   \draw (3.3,3) circle (1.5pt);
   \draw (3,3) -- (4,3);
   \draw (3,4) -- (4,4);

   \draw [blue] (0,0) -- (2,2);
   \draw [blue] (3,3) -- (4,4);
   \draw [fill,blue] (1,3) circle (2pt);
   \draw [fill,blue] (3,1) circle (2pt);

   \draw [green] (0.5,1.5) -- (1,1.5);
   \draw [green] (0.5,1.5) -- (0.5,1);
   \draw [green] (1,1.5) -- (1.33,1.66);
   \draw [green] (0.5,1) -- (0.33,0.66);

   \draw [green] (1.5,0.5) -- (1.5,1);
   \draw [green] (1.5,0.5) -- (1,0.5);
   \draw [green] (1.5,1) -- (1.66,1.33);
   \draw [green] (1,0.5) -- (0.66,0.33);

   \draw [green] (0.5,3) -- (0.5,3.5);
   \draw [green] (0.5,3.5) -- (1.5,3.5);
   \draw [green] (1.5,3) -- (1.5,3.5);

   \draw [green] (3,0.5) -- (3.5,0.5);
   \draw [green] (3.5,0.5) -- (3.5,1.5);
   \draw [green] (3,1.5) -- (3.5,1.5);

   \draw [green] (3.33,3.66) -- (3,3.5);
   \draw [green] (3.66,3.33) -- (3.5,3);
  \end{tikzpicture}
 \end{center}

  By gluing these green parts, $C_2^{\comp}(Y^{\circ})$ is a
 $1$-dimensional simplicial complex depicted below.

 \begin{center}
  \newdimen\D
  \D=1.5cm
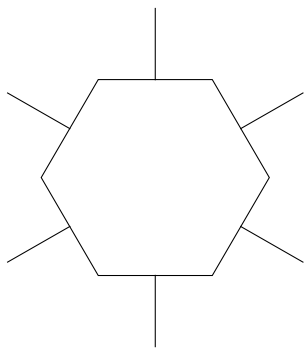
  \begin{tikzpicture}
   \draw (0:\D) \foreach \x in {60,120,...,359} {
                -- (\x:\D)
            }-- cycle (90:\D);
   \draw \foreach \x in {30,90,...,330} {
   (\x:\D*0.86) -- (\x:\D*1.5)
   };
  \end{tikzpicture}
 \end{center}  

 Thus $C_2^{\comp}(Y^{\circ})$ is a model which realizes the homotopy
 dimension of $\Conf_2(Y)$.

 As we have seen from this example, $C_k^{\comp}(X^{\circ})$ is much
 smaller than $C_k^{\comp}(X)$ in general. It often realizes the
 homotopy dimension of $\Conf_k(X)$.

 The above example also says that $C_k^{\comp}(X^{\circ})$ still has
 some room to be simplified. By removing six spines from
 $C_2^{\comp}(Y^{\circ})$, it can be collapsed to a dodecagon.
\end{example}

\subsection{A Simplified Model for Graphs}
\label{model_for_two_points}

Although our model $C_k^{\comp}(X^{\circ})$ often realizes the homotopy
dimension of $\Conf_k(X)$, it still contains collapsible parts, as we
have seen in Example \ref{two_points_in_Y_without_leaves}.
In this section, we concentrate on the case $k=2$ and construct a
minimal model by collapsing inessential parts in $C_2^{\comp}(X^{\circ})$. 

Recall from Proposition \ref{cellular_stratified_space_as_colim} and the
proof of Theorem
\ref{barycentric_subdivision_of_cellular_stratification} that both
$\Conf_k(X^{\circ})$ and $C^{\comp}_k(X^{\circ})$ have colimit
decompositions
\begin{eqnarray*}
 \Conf_k(X^{\circ}) & = & \colim_{C(\pi_{k,X^{\circ}}^{\comp})}
  D^{\Conf_k(X^{\circ})} \\ 
 C^{\comp}_k(X^{\circ}) & = & \colim_{C(\pi_{k,X^{\circ}}^{\comp})}
  \Sd\left(D^{\Conf_k(X^{\circ})}\right), 
\end{eqnarray*}
under which the embedding of $C^{\comp}_k(X^{\circ})$ into
$\Conf_k(X^{\circ})$ decomposes
into a colimit
\[
\begin{diagram}
 \node{C^{\comp}_k(X^{\circ})} \arrow{s,=} 
 \arrow[2]{e,t,J}{i_{\Conf_k(X^{\circ})}} \node{}
 \node{\Conf_k(X^{\circ})} \arrow{s,=} \\ 
 \node{\hspace*{10pt}\colim_{C(\pi_{k,X^{\circ}}^{\comp})}
  \Sd\left(D^{\Conf_k(X^{\circ})}\right)\hspace*{10pt}}
 \arrow[2]{e,t,J}{\displaystyle \colim_{C(\pi^{\comp}_{k,X^{\circ}})}
 i_{\lambda}} 
 \node{} 
 \node{\colim_{C(\pi_{k,X^{\circ}}^{\comp})} D^{\Conf_k(X^{\circ})}.} 
\end{diagram}
\]
In other words, $C_k^{\comp}(X^{\circ})$ is obtained by gluing
$\Sd(D_{\lambda})$'s for all cells $e_{\lambda}$ in the braid
stratification of $\Conf_k(X^{\circ})$. 
By simplifying each $\Sd(D_{\lambda})$, we may collapse
$C_k^{\comp}(X^{\circ})$ further.

When $k=2$, we use the following notation for $D_{\lambda}$ and
$\Sd(D_{\lambda})$. 

\begin{definition}
 Let $X$ be a connected finite graph. The sets of vertices and edges are
 denoted by 
 \begin{eqnarray*}
  V(X) & = & \set{e_{\lambda}^0}{\lambda\in\Lambda_0} \\
  E(X) & = & \set{e_{\lambda}^1}{\lambda\in\Lambda_1},
 \end{eqnarray*}
 respectively. The sets of loops, branches, and
 connections (Definition \ref{graph_terminologies}) are denoted by
 \begin{eqnarray*}
  L(X) & = & \set{e^1_{\lambda}}{\lambda\in \Lambda_{L}} \\
  B(X) & = & \set{e^1_{\lambda}}{\lambda\in \Lambda_{B}} \\
  C(X) & = & \set{e^1_{\lambda}}{\lambda\in \Lambda_{C}},
 \end{eqnarray*}
 respectively.
 
 For $1$-cells
 $\varphi_{\lambda} : D_{\lambda} \to \overline{e^1_{\lambda}}$ and
 $\varphi_{\mu} : D_{\mu} \rarrow{} \overline{e^1_{\mu}}$ in $X$, we
 use  
 \[
 \varphi_{\lambda}\times\varphi_{\mu} : D_{\lambda,\mu} =
 D_{\lambda}\times D_{\mu} \rarrow{}
 \overline{e^1_{\lambda}\times e^1_{\mu}}
 \]
 as the characteristic map for $e^1_{\lambda}\times e^1_{\mu}$.

 Up to an action of $\Sigma_2$, we classify $2$-cells in
 $X\times X$ into the following nine types:
 \begin{enumerate}
  \item $e^1_{\lambda}\times e^1_{\lambda}$ for $\lambda\in\Lambda_{L}$
  \item $e^1_{\lambda}\times e^1_{\lambda'}$ for
	 $\lambda,\lambda'\in\Lambda_{L}$ ($\lambda\neq\lambda'$)
  \item $e^1_{\lambda}\times e^1_{\mu}$ for
	 $\lambda\in \Lambda_{L}$, $\mu\in\Lambda_{B}$
  \item $e^1_{\lambda}\times e^1_{\nu}$ for $\lambda\in\Lambda_{L}$,
	$\nu\in\Lambda_{C}$ 
  \item $e^1_{\mu}\times e^1_{\mu}$ for $\mu\in\Lambda_{B}$
  \item $e^1_{\mu}\times e^1_{\mu'}$ for $\mu,\mu'\in\Lambda_{B}$
	($\mu\neq\mu'$) 
  \item $e^1_{\mu}\times e^1_{\nu}$ for $\mu\in\Lambda_{B}$,
	$\nu\in\Lambda_{C}$ 
  \item $e^1_{\nu}\times e^1_{\nu}$ for $\nu\in\Lambda_{C}$
  \item $e^1_{\nu}\times e^1_{\nu}$ for $\nu,\nu'\in\Lambda_{C}$
	($\nu\neq\nu'$) 
 \end{enumerate}
\end{definition}

Domains for $2$-cells of type 1, 2, 4, 8, and 9 are $[-1,1]^2$. In
$(X^{\circ})^2$, domains for $2$-cells of type 3, 5, 6, and 7 are
$[-1,1]\times[-1,1)$, $[-1,1)^2$, $[-1,1)^2$, and $[-1,1)\times[-1,1]$,
respectively. 

Under the subdivision via the braid arrangement $\mathcal{A}_1$,
$2$-cells of type 1, 5, and 8 are subdivided and so are their
domains. 
For a $2$-cell $e^1_{\lambda}\times e^{1}_{\lambda}$ of type 1, 5, or 8,
we denote the subdivision by
\[
 e^1_{\lambda}\times e^1_{\lambda} = e^2_{\lambda,+}\cup
  e^1_{\lambda,\Delta} \cup e^2_{\lambda,-}.
\]
Then we have the following stratification of $\Conf_2(X)$
\begin{eqnarray*}
 \Conf_2(X) & = & \bigcup_{\alpha,\beta\in\Lambda_0,\alpha\neq\beta}
  e^0_{\alpha}\times e^0_{\beta} \\
 & & \cup \bigcup_{(\alpha,\lambda)\in\Lambda_0\times\Lambda_1}
  e^0_{\alpha}\times e^1_{\lambda} 
  \cup \bigcup_{(\lambda,\alpha)\in\Lambda_1\times\Lambda_0}
  e^1_{\lambda}\times e^0_{\alpha} \\
 & & \cup \bigcup_{\lambda\in\Lambda_{L}} e^2_{\lambda,+} \cup
  e^2_{\lambda,-} 
 \cup \bigcup_{\lambda,\lambda'\in\Lambda_{L},\lambda\neq\lambda'}
  e^1_{\lambda}\times e^1_{\lambda'} 
 \cup \bigcup_{\lambda\in\Lambda_{L},\mu\in\Lambda_{B}}
  e^1_{\lambda}\times e^1_{\mu} 
 \cup \bigcup_{\lambda\in\Lambda_{L},\nu\in\Lambda_{C}}
  e^1_{\lambda}\times e^1_{\nu} \\
 & & \cup \bigcup_{\mu\in\Lambda_{B}} e^2_{\mu,+} \cup e^2_{\mu,-} 
 \cup \bigcup_{\mu,\mu'\in\Lambda_{B},\mu\neq\mu'} e^1_{\mu}\times
  e^1_{\mu'} 
 \cup \bigcup_{\mu\in\Lambda_{B},\nu\in\Lambda_{C}} e^1_{\mu}\times
  e^1_{\nu} \\ 
 & & \cup \bigcup_{\nu\in\Lambda_{C}} e^2_{\nu,+}\cup e^2_{\nu,-} 
 \cup \bigcup_{\nu,\nu'\in\Lambda_{C}} e^1_{\nu}\times e^1_{\nu'}.
\end{eqnarray*}

For $\lambda\in\Lambda_{L}$, $\mu\in\Lambda_{B}$, and
$\nu\in\Lambda_{C}$, domains 
$D_{\lambda,\pm}$, $D_{\mu,\pm}$, and $D_{\nu,\pm}$ of the
characteristic map of the cells  $e^2_{\lambda,\pm}$, $e^2_{\mu,\pm}$,
and $e^2_{\nu,\pm}$ are defined as follows:
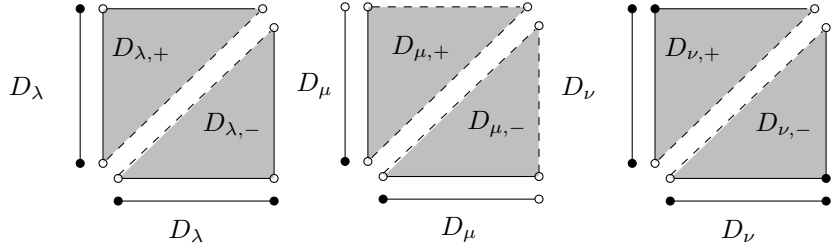
\begin{center}
 \begin{tikzpicture}
  \draw [fill=gray, opacity=.5, draw=white] (5, 0.3) --
  (7.05, 0.3) -- (7.05, 2.3) -- (5, 0.3) ; 
  \draw [fill=gray, opacity=.5, draw=white] (4.8, 0.5) --
  (4.8, 2.55) -- (6.9, 2.55) -- (4.8, 0.5) ;
  \draw (5.0, 0) -- (7.0, 0) ;
  \draw [fill] (5, 0) circle (1.5pt) ;
  \draw [fill] (7.05, 0) circle (1.5pt);
  \draw (4.5, 0.5) -- (4.5, 2.5) ;
  \draw [fill] (4.5, 0.5) circle (1.5pt) ;
  \draw [fill] (4.5, 2.55) circle (1.5pt) ;
  \draw (5.05, 0.3) -- (7, 0.3);
  \draw (5, 0.3) circle (1.5pt) ;
  \draw (7.05, 0.3) circle (1.5pt) ;
  \draw (4.8, 0.5) circle (1.5pt) ;
  \draw (4.8, 0.55) -- (4.8, 2.5) ;
  \draw (4.8, 2.55) circle (1.5pt) ;
  \draw (7.05, 2.3) circle (1.5pt) ;
  \draw (6.9, 2.55) circle (1.5pt) ;
  \draw [dashed] (5, 0.34) -- (7, 2.25) ;
  \draw [dashed] (4.85, 0.55) -- (6.85, 2.5) ;
  \draw  (7.05, 0.35) -- (7.05, 2.25) ;
  \draw  (4.85, 2.55) -- (6.8, 2.55) ;
  
  \coordinate (0) at (6.5, 1) node at (0) {$D_{\lambda,-}$} ;
  \coordinate (1) at (5.3, 2) node at (1) {$D_{\lambda,+}$} ;
  
  \coordinate (0) at (5.9, -0.4) node at (0) {$D_{\lambda}$} ;
  \coordinate (1) at (3.8, 1.5) node at (1) {$D_{\lambda}$} ;
 \end{tikzpicture}
 \begin{tikzpicture}
  \draw [fill=gray, opacity=.5, draw=white, ] (0.5, 0.3) --
  (2.55, 0.3) -- (2.55, 2.3) -- (0.5, 0.3) ; 
  \draw [fill=gray, opacity=.5, draw=white] (0.3, 0.5) --
  (0.3, 2.55) -- (2.4, 2.55) -- (0.3, 0.5) ; 
  \draw (0.5, 0) -- (2.5, 0) ;
  \draw [fill] (0.5, 0) circle (1.5pt) ;
  \draw (2.55, 0) circle (1.5pt);
  \draw (0, 0.5) -- (0, 2.5) ;
  \draw [fill] (0, 0.5) circle (1.5pt) ;
  \draw (0, 2.55) circle (1.5pt) ;
  \draw (0.55, 0.3) -- (2.5, 0.3);
  \draw (0.5, 0.3) circle (1.5pt) ;
  \draw (2.55, 0.3) circle (1.5pt) ;
  \draw (0.3, 0.5) circle (1.5pt) ;
  \draw (0.3, 0.55) -- (0.3, 2.5) ;
  \draw (0.3, 2.55) circle (1.5pt) ;
  \draw (2.55, 2.3) circle (1.5pt) ;
  \draw (2.4, 2.55) circle (1.5pt) ;
  \draw [dashed] (0.5, 0.34) -- (2.5, 2.25) ;
  \draw [dashed] (0.35, 0.55) -- (2.35, 2.5) ;
  \draw [dashed] (2.55, 0.35) -- (2.55, 2.25) ;
  \draw [dashed] (0.35, 2.55) -- (2.3, 2.55) ;
  
  \coordinate (0) at (2, 1) node at (0) {$D_{\mu,-}$} ;
  \coordinate (1) at (1, 2) node at (1) {$D_{\mu,+}$} ;
  
  \coordinate (0) at (1.5, -0.4) node at (0) {$D_{\mu}$} ;
  \coordinate (1) at (-0.4, 1.5) node at (1) {$D_{\mu}$} ;
 \end{tikzpicture}
 \begin{tikzpicture}
  \draw [fill=gray, opacity=.5, draw=white] (5, 0.3) --
  (7.05, 0.3) -- (7.05, 2.3) -- (5, 0.3) ; 
  \draw [fill=gray, opacity=.5, draw=white] (4.8, 0.5) --
  (4.8, 2.55) -- (6.9, 2.55) -- (4.8, 0.5) ;
  \draw (5.0, 0) -- (7.0, 0) ;
  \draw [fill] (5, 0) circle (1.5pt) ;
  \draw [fill] (7.05, 0) circle (1.5pt);
  \draw (4.5, 0.5) -- (4.5, 2.5) ;
  \draw [fill] (4.5, 0.5) circle (1.5pt) ;
  \draw [fill] (4.5, 2.55) circle (1.5pt) ;
  \draw (5.05, 0.3) -- (7, 0.3);
  \draw (5, 0.3) circle (1.5pt) ;
  \draw [fill] (7.05, 0.3) circle (1.5pt) ;
  \draw (4.8, 0.5) circle (1.5pt) ;
  \draw (4.8, 0.55) -- (4.8, 2.5) ;
  \draw [fill] (4.8, 2.55) circle (1.5pt) ;
  \draw (7.05, 2.3) circle (1.5pt) ;
  \draw (6.9, 2.55) circle (1.5pt) ;
  \draw [dashed] (5, 0.34) -- (7, 2.25) ;
  \draw [dashed] (4.85, 0.55) -- (6.85, 2.5) ;
  \draw  (7.05, 0.35) -- (7.05, 2.25) ;
  \draw  (4.85, 2.55) -- (6.8, 2.55) ;
  
  \coordinate (0) at (6.5, 1) node at (0) {$D_{\nu,-}$} ;
  \coordinate (1) at (5.3, 2) node at (1) {$D_{\nu,+}$} ;
  
  \coordinate (0) at (5.9, -0.4) node at (0) {$D_{\nu}$} ;
  \coordinate (1) at (3.8, 1.5) node at (1) {$D_{\nu}$} ;
 \end{tikzpicture}
\end{center}

Let us consider $\Sd$ of these domains.

\begin{proposition}
 The barycentric subdivisions of domains of $2$-cells in the braid 
 stratification on $\Conf_2(X)$ are given by the red regions in the
 following figures:
 \begin{enumerate}
  \item For $\lambda\in\Lambda_{L}$, $\mu\in\Lambda_{B}$, and
	$\nu\in\Lambda_{C}$, $\Sd(D_{\lambda,\pm})$,
	$\Sd(D_{\mu,\pm})$, and $\Sd(D_{\nu,\pm})$ are given by
	\begin{center}
	 \begin{tikzpicture}
	  \draw (5.05, 0.3) -- (7, 0.3);
	  \draw (5, 0.3) circle (1.5pt) ;
	  \draw (7.05, 0.3) circle (1.5pt) ;
	  \draw (4.8, 0.5) circle (1.5pt) ;
	  \draw (4.8, 0.55) -- (4.8, 2.5) ;
	  \draw (4.8, 2.55) circle (1.5pt) ;
	  \draw (7.05, 2.3) circle (1.5pt) ;
	  \draw (6.9, 2.55) circle (1.5pt) ;
	  \draw [dashed] (5, 0.34) -- (7, 2.25) ;
	  \draw [dashed] (4.85, 0.55) -- (6.85, 2.5) ;
	  \draw (7.05, 0.35) -- (7.05, 2.25) ;
	  \draw (4.85, 2.55) -- (6.8, 2.55) ;
  
	  \draw [red, fill] (5.9, 0.3) circle (1.5pt) ;
	  \draw [red, fill] (6.383, 0.96) circle (1.5pt) ;
	  \draw [red] (5.9, 0.3) -- (6.383, 0.96) ;
	  \draw [red] (7.05, 1.3) -- (6.383, 0.96) ;
	  \draw [red, fill] (7.05, 1.3) circle (1.5pt) ;
  
	  \draw [red, fill] (4.8, 1.55) circle (1.5pt) ;
	  \draw [red, fill] (5.5, 1.883) circle (1.5pt) ;
	  \draw [red] (4.8, 1.55) -- (5.5, 1.883) ;
	  \draw [red] (5.85, 2.55) -- (5.5, 1.883) ;
	  \draw [red, fill] (5.85, 2.55) circle (1.5pt) ;
  
  
  
	  \coordinate (2) at (4.5, 2.2) node at (2) {$\Sd(D_{\lambda,+})$} ;
	  \coordinate (2) at (6, 1.4) node at (2) {$\Sd(D_{\lambda,-})$} ;
	  \end{tikzpicture}
	 \hspace*{30pt}
	 \begin{tikzpicture}
	  \draw (0.55, 0.3) -- (2.5, 0.3);
	  \draw (0.5, 0.3) circle (1.5pt) ;
	  \draw (2.55, 0.3) circle (1.5pt) ;
	  \draw (0.3, 0.5) circle (1.5pt) ;
	  \draw (0.3, 0.55) -- (0.3, 2.5) ;
	  \draw (0.3, 2.55) circle (1.5pt) ;
	  \draw (2.55, 2.3) circle (1.5pt) ;
	  \draw (2.4, 2.55) circle (1.5pt) ;
	  \draw [dashed] (0.5, 0.34) -- (2.5, 2.25) ;
	  \draw [dashed] (0.35, 0.55) -- (2.35, 2.5) ;
	  \draw [dashed] (2.55, 0.35) -- (2.55, 2.25) ;
	  \draw [dashed] (0.35, 2.55) -- (2.3, 2.55) ;
  
	  \draw [red, fill] (1.55, 0.3) circle (1.5pt) ;
	  \draw [red, fill] (1.883, 0.96) circle (1.5pt) ;
	  \draw [red] (1.55, 0.3) -- (1.883, 0.96) ;
  
	  \draw [red, fill] (0.3, 1.55) circle (1.5pt) ;
	  \draw [red, fill] (1, 1.883) circle (1.5pt) ;
	  \draw [red] (0.3, 1.55) -- (1, 1.883) ;
  
	  \coordinate (1) at (1.2, 2.2) node at (1) {$\Sd(D_{\mu,+})$} ;
	  \coordinate (1) at (1.5, 1.3) node at (1) {$\Sd(D_{\mu,-})$} ;
	 \end{tikzpicture}  
	 \hspace*{30pt}
	 \begin{tikzpicture}
	  \draw (5.05, 0.3) -- (7, 0.3);
	  \draw (5, 0.3) circle (1.5pt) ;
	  \draw [red, fill] (7.05, 0.3) circle (1.5pt) ;
	  \draw (4.8, 0.5) circle (1.5pt) ;
	  \draw (4.8, 0.55) -- (4.8, 2.5) ;
	  \draw [red, fill] (4.8, 2.55) circle (1.5pt) ;
	  \draw (7.05, 2.3) circle (1.5pt) ;
	  \draw (6.9, 2.55) circle (1.5pt) ;
	  \draw [dashed] (5, 0.34) -- (7, 2.25) ;
	  \draw [dashed] (4.85, 0.55) -- (6.85, 2.5) ;
	  \draw (7.05, 0.35) -- (7.05, 2.25) ;
	  \draw (4.85, 2.55) -- (6.8, 2.55) ;
  
	  \draw [red, fill] (5.9, 0.3) circle (1.5pt) ;
	  \draw [red, fill] (6.383, 0.96) circle (1.5pt) ;
	  \draw [red, fill] (5.9, 0.3) -- (6.383, 0.96) -- (7.05, 1.3)
	  -- (7.05,0.3) -- (5.9,0.3) ;
	  \draw [red, fill] (7.05, 1.3) circle (1.5pt) ;

	  \draw [->] (6.543,0.8) -- (6.95,0.4);
	  \draw [->] (6.2,0.5) -- (6.8,0.4);
	  \draw [->] (6.85,1.1) -- (6.95,0.55);

	  \draw [red, fill] (4.8, 1.55) circle (1.5pt) ;
	  \draw [red, fill] (5.5, 1.883) circle (1.5pt) ;
	  \draw [red, fill] (4.8, 1.55) -- (5.5, 1.883) -- (5.85, 2.55)
	  -- (4.8, 2.55) -- (4.8,1.55);
	  \draw [red, fill] (5.85, 2.55) circle (1.5pt) ;
  
	  \draw [->] (5.34,2.043) -- (4.9,2.45);
	  \draw [->] (5,1.85) -- (4.95,2.3);
	  \draw [->] (5.55,2.35) -- (5.05,2.45);
  
	  \coordinate (2) at (6.5, 2.2) node at (2) {$\Sd(D_{\nu,+})$} ;
	  \coordinate (2) at (5.7, 1.2) node at (2) {$\Sd(D_{\nu,-})$} ;
	  \end{tikzpicture}
	\end{center}

  \item For $\lambda\in\Lambda_{L}$, $\mu\in\Lambda_{B}$, and
	$\nu\in\Lambda_{C}$, if
	$\overline{e_{\lambda}}\cap\overline{e_{\mu}}=\{v\}$, 
	$\overline{e_{\lambda}}\cap\overline{e_{\nu}}=\{v\}$, and
	$\overline{e_{\mu}}\cap\overline{e_{\nu}}=\{v\}$ for a vertex
	$v$, then 
	$\Sd(D_{\lambda,\mu})$, 
	$\Sd(D_{\lambda,\nu})$, and $\Sd(D_{\mu,\nu})$ are given by
	\begin{center}
	  \begin{tikzpicture}
	   \draw (9.5, 0.5) circle (1.5pt) ;
	   \draw (9.55, 0.5) -- (11.5, 0.5);
	   \draw (11.55, 0.5) circle (1.5pt) ;
	   \draw (11.55, 0.55) -- (11.55, 2.5) ;
	   \draw (11.55, 2.55) circle (1.5pt) ;
	   \draw [dashed] (11.5, 2.55) -- (9.55, 2.55) ;
	   \draw (9.5, 2.55) circle (1.5pt) ;
	   \draw (9.5, 2.5) -- (9.5, 0.55) ;
 
	   \draw [red, fill] (10.525, 0.5) circle (1.5pt) ;
	   \draw [red, fill] (10.525, 1.525) circle (1.5pt) ;
	   \draw [red, fill] (11.55, 1.525) circle (1.5pt) ;
	   \draw [red, fill] (9.5, 1.525) circle (1.5pt) ;
 
	   \draw [red] (10.525, 0.5) -- (10.525, 1.525) ;
	   \draw [red] (10.525, 1.525) -- (11.55, 1.525) ;
	   \draw [red] (10.525, 1.525) -- (9.5, 1.525) ;
 
	   \coordinate (1) at (10.525, 1.8) node at (1)
	   {$\Sd(D_{\lambda,\mu})$} ;  
	 \end{tikzpicture} 
	 \hspace*{30pt}
	 \begin{tikzpicture}
	  \draw [red,fill] (0.5, 0.5) circle (1.5pt) ;
	  \draw (0.55, 0.5) -- (2.5, 0.5);
	  \draw [red,fill] (2.55, 0.5) circle (1.5pt) ;
	  \draw (2.55, 0.55) -- (2.55, 2.5) ;
	  \draw (2.55, 2.55) circle (1.5pt) ;
	  \draw (2.5, 2.55) -- (0.55, 2.55) ;
	  \draw (0.5, 2.55) circle (1.5pt) ;
	  \draw (0.5, 2.5) -- (0.5, 0.55) ;
 
	  \draw [red, fill] (1.525, 0.5) circle (1.5pt) ;
	  \draw [red, fill] (1.525, 1.525) circle (1.5pt) ;
	  \draw [red, fill] (2.55, 1.525) circle (1.5pt) ;
	  \draw [red, fill] (1.525, 2.55) circle (1.5pt) ;
	  \draw [red, fill] (0.5, 1.525) circle (1.5pt) ;
 
	  \draw [red] (1.525, 1.525) -- (1.525, 2.55) ;
	  \draw [red,fill] (0.5,0.5) -- (0.5,1.525) -- (2.55,1.525)
	  -- (2.55,0.5) -- (0.5,0.5);

	  \draw [->] (1,1.3) -- (1,0.55);
	  \draw [->] (2,1.3) -- (2,0.55);
	  
	  \coordinate (1) at (1.5, 1.8) node at (1)
	  {$\Sd(D_{\lambda,\nu})$} ; 
	 \end{tikzpicture}
	 \hspace*{30pt}
	  \begin{tikzpicture}
	   \draw (9.5, 0.5) circle (1.5pt) ;
	   \draw (9.55, 0.5) -- (11.5, 0.5);
	   \draw (11.55, 0.5) circle (1.5pt) ;
	   \draw [dashed] (11.55, 0.55) -- (11.55, 2.5) ;
	   \draw (11.55, 2.55) circle (1.5pt) ;
	   \draw (11.5, 2.55) -- (9.55, 2.55) ;
	   \draw [red, fill] (9.5, 2.55) circle (1.5pt) ;
	   \draw (9.5, 2.5) -- (9.5, 0.55) ;
 
	   \draw [red, fill] (10.525, 0.5) circle (1.5pt) ;
	   \draw [red, fill] (10.525, 1.525) circle (1.5pt) ;
	   \draw [red, fill] (10.525, 2.55) circle (1.5pt) ;
	   \draw [red, fill] (9.5, 1.525) circle (1.5pt) ;
 
	   \draw [red] (10.525,1.525) -- (10.525,0.5);
	   \draw [red, fill] (9.5, 1.525) -- (10.525, 1.525)
	   -- (10.525, 2.55) -- (9.5,2.55) -- (9.5, 1.525);

	   \draw [->] (9.95,1.575) -- (9.95,1.975);
	   \draw [->] (10.475,2.1) -- (10.075,2.1);

	   \coordinate (1) at (11.55, 1.8) node at (1)
	   {$\Sd(D_{\mu,\nu})$} ;  
	 \end{tikzpicture} 
	\end{center}

  \item For $\lambda\in\Lambda_{L}$, $\mu\in\Lambda_{B}$, and
	$\nu\in\Lambda_{C}$, if the pairs $(e_{\lambda},e_{\mu})$,
	$(e_{\lambda},e_{\nu})$, and $(e_{\mu},e_{\nu})$ do not share
	common vertices, then
	$\Sd(D_{\lambda,\mu})$, 
	$\Sd(D_{\lambda,\nu})$, and $\Sd(D_{\mu,\nu})$ are given by
	\begin{center}
	  \begin{tikzpicture}
	   \draw [red,fill] (9.5, 0.5) circle (1.5pt) ;
	   \draw (9.55, 0.5) -- (11.5, 0.5);
	   \draw [red,fill] (11.55, 0.5) circle (1.5pt) ;
	   \draw (11.55, 0.55) -- (11.55, 2.5) ;
	   \draw (11.55, 2.55) circle (1.5pt) ;
	   \draw [dashed] (11.5, 2.55) -- (9.55, 2.55) ;
	   \draw (9.5, 2.55) circle (1.5pt) ;
	   \draw (9.5, 2.5) -- (9.5, 0.55) ;
 
	   \draw [red, fill] (10.525, 0.5) circle (1.5pt) ;
	   \draw [red, fill] (10.525, 1.525) circle (1.5pt) ;
	   \draw [red, fill] (11.55, 1.525) circle (1.5pt) ;
	   \draw [red, fill] (9.5, 1.525) circle (1.5pt) ;
 
	   \draw [red, fill] (9.5, 0.5) -- (9.5,1.525) -- (11.55,1.525)
	   -- (11.55,0.5) -- (9.5,0.5);

	  \draw [->] (9.75,1.3) -- (9.75,0.55);
	  \draw [->] (10.525,1.3) -- (10.525,0.55);
	  \draw [->] (11.3,1.3) -- (11.3,0.55);
 
	   \coordinate (1) at (10.525, 1.8) node at (1)
	   {$\Sd(D_{\lambda,\mu})$} ;  
	 \end{tikzpicture} 
	 \hspace*{30pt}
	 \begin{tikzpicture}
	  \draw [red, fill] (0.5, 0.5) -- (2.5, 0.5) -- (2.5,2.5) --
	  (0.5,2.5) -- (0.5,0.5) ;
	  \draw [red, fill] (0.5, 0.5) circle (1.5pt) ;
	  \draw [red, fill] (2.55, 0.5) circle (1.5pt) ;
	  \draw [red, fill] (2.55, 2.55) circle (1.5pt) ;
	  \draw [red, fill] (0.5, 2.55) circle (1.5pt) ;
 
	  \draw [red, fill] (1.525, 0.5) circle (1.5pt) ;
	  \draw [red, fill] (2.55, 1.525) circle (1.5pt) ;
	  \draw [red, fill] (1.525, 2.55) circle (1.5pt) ;
	  \draw [red, fill] (0.5, 1.525) circle (1.5pt) ;
 
	  \coordinate (1) at (1.525, 1.8) node at (1)
	  {$\Sd(D_{\lambda,\nu})$} ; 
	 \end{tikzpicture}
	 \hspace*{30pt}
	  \begin{tikzpicture}
	   \draw [red, fill] (9.5, 0.5) circle (1.5pt) ;
	   \draw (9.55, 0.5) -- (11.5, 0.5);
	   \draw (11.55, 0.5) circle (1.5pt) ;
	   \draw [dashed] (11.55, 0.55) -- (11.55, 2.5) ;
	   \draw (11.55, 2.55) circle (1.5pt) ;
	   \draw (11.5, 2.55) -- (9.55, 2.55) ;
	   \draw [red, fill] (9.5, 2.55) circle (1.5pt) ;
	   \draw (9.5, 2.5) -- (9.5, 0.55) ;
 
	   \draw [red, fill] (10.525, 0.5) circle (1.5pt) ;
	   \draw [red, fill] (10.525, 1.525) circle (1.5pt) ;
	   \draw [red, fill] (10.525, 2.55) circle (1.5pt) ;
	   \draw [red, fill] (9.5, 1.525) circle (1.5pt) ;
 
	   \draw [red, fill] (9.5, 0.5) -- (10.525, 0.5)
	   -- (10.525, 2.55) -- (9.5,2.55) -- (9.5, 0.5);

	   \draw [->] (10.3,2.3) -- (9.6,2.3);
	   \draw [->] (10.3,1.525) -- (9.6,1.525);
	   \draw [->] (10.3,0.75) -- (9.6,0.75);

	   \coordinate (1) at (11.55, 1.8) node at (1)
	   {$\Sd(D_{\mu,\nu})$} ;  
	 \end{tikzpicture} 
	\end{center}

  \item For $\lambda,\lambda'\in\Lambda_{L}$ ($\lambda\neq\lambda'$),
	$\mu,\mu'\in\Lambda_{B}$ ($\mu\neq\mu'$), and
	$\nu,\nu'\in\Lambda_{C}$ ($\nu\neq\nu'$), if 
	$\overline{e_{\lambda}}\cap\overline{e_{\lambda'}}=\{v\}$,
	$\overline{e_{\mu}}\cap\overline{e_{\mu'}}=\{v\}$, and 
	$\overline{e_{\nu}}\cap\overline{e_{\nu'}}=\{v\}$ for a vertex
	$v$, then 
	$\Sd(D_{\lambda,\lambda'})$, $\Sd(D_{\mu,\mu'})$, and
	$\Sd(D_{\nu,\nu'})$ are given by
	\begin{center}
	 \begin{tikzpicture}
	  \draw (0.5, 0.5) circle (1.5pt) ;
	  \draw (0.55, 0.5) -- (2.5, 0.5);
	  \draw (2.55, 0.5) circle (1.5pt) ;
	  \draw (2.55, 0.55) -- (2.55, 2.5) ;
	  \draw (2.55, 2.55) circle (1.5pt) ;
	  \draw (2.5, 2.55) -- (0.55, 2.55) ;
	  \draw (0.5, 2.55) circle (1.5pt) ;
	  \draw (0.5, 2.5) -- (0.5, 0.55) ;
 
	  \draw [red, fill] (1.525, 0.5) circle (1.5pt) ;
	  \draw [red, fill] (1.525, 1.525) circle (1.5pt) ;
	  \draw [red, fill] (2.55, 1.525) circle (1.5pt) ;
	  \draw [red, fill] (1.525, 2.55) circle (1.5pt) ;
	  \draw [red, fill] (0.5, 1.525) circle (1.5pt) ;
 
	  \draw [red] (1.525, 0.5) -- (1.525, 1.525) ;
	  \draw [red] (1.525, 1.525) -- (2.55, 1.525) ;
	  \draw [red] (1.525, 1.525) -- (1.525, 2.55) ;
	  \draw [red] (1.525, 1.525) -- (0.5, 1.525) ;
 
	  \coordinate (1) at (1.4, 2) node at (1)
	  {$\Sd(D_{\lambda,\lambda'})$} ; 
	 \end{tikzpicture}
	 \hspace*{30pt}
	  \begin{tikzpicture}
	  \draw (9.5, 0.5) circle (1.5pt) ;
	  \draw (9.55, 0.5) -- (11.5, 0.5);
	  \draw (11.55, 0.5) circle (1.5pt) ;
	  \draw [dashed] (11.55, 0.55) -- (11.55, 2.5) ;
	  \draw (11.55, 2.55) circle (1.5pt) ;
	  \draw [dashed] (11.5, 2.55) -- (9.55, 2.55) ;
	  \draw (9.5, 2.55) circle (1.5pt) ;
	  \draw (9.5, 2.5) -- (9.5, 0.55) ;
 
	  \draw [red, fill] (10.525, 0.5) circle (1.5pt) ;
	  \draw [red, fill] (10.525, 1.525) circle (1.5pt) ;
	  \draw [red, fill] (9.5, 1.525) circle (1.5pt) ;
 
	  \draw [red] (10.525, 0.5) -- (10.525, 1.525) ;
	  \draw [red] (10.525, 1.525) -- (9.5, 1.525) ;
 
	  \coordinate (1) at (10.525, 1.8) node at (1)
	   {$\Sd(D_{\mu,\mu'})$} ;  
	 \end{tikzpicture}
	 \hspace*{30pt}
	 \begin{tikzpicture}
	  \draw (0.5, 0.5) circle (1.5pt) ;
	  \draw (0.55, 0.5) -- (2.5, 0.5);
	  \draw [red, fill] (2.55, 0.5) circle (1.5pt) ;
	  \draw (2.55, 0.55) -- (2.55, 2.5) ;
	  \draw [red, fill] (2.55, 2.55) circle (1.5pt) ;
	  \draw (2.5, 2.55) -- (0.55, 2.55) ;
	  \draw [red, fill] (0.5, 2.55) circle (1.5pt) ;
	  \draw (0.5, 2.5) -- (0.5, 0.55) ;
 
	  \draw [red, fill] (1.525, 0.5) circle (1.5pt) ;
	  \draw [red, fill] (1.525, 1.525) circle (1.5pt) ;
	  \draw [red, fill] (2.55, 1.525) circle (1.5pt) ;
	  \draw [red, fill] (1.525, 2.55) circle (1.5pt) ;
	  \draw [red, fill] (0.5, 1.525) circle (1.5pt) ;
 
	  \draw [red, fill] (1.525, 0.5) -- (2.55, 0.5) --
	  (2.55, 2.55) -- (0.5, 2.55) -- (0.5, 1.525) -- (1.525,1.525)
	  -- (1.525, 0.5);

	   \draw [->] (0.95,1.575) -- (0.95,2.3);
	   \draw [->] (1.575,0.95) -- (2.3,0.95);
 
	  \coordinate (1) at (2.5, 2) node at (1)
	  {$\Sd(D_{\nu,\nu'})$} ; 
	 \end{tikzpicture}
	\end{center}

  \item For $\nu,\nu'\in\Lambda_{C}$ ($\nu\neq\nu'$), if
	$\overline{e_{\nu}}\cap\overline{e_{\nu'}}=\{v,w\}$ for vertices 
	$v$ and $w$, then $\Sd(D_{\nu,\nu'})$ is given by
	
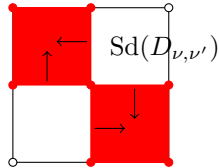
\begin{center}
	 \begin{tikzpicture}
	  \draw (0.5, 0.5) circle (1.5pt) ;
	  \draw (0.55, 0.5) -- (2.5, 0.5);
	  \draw [red, fill] (2.55, 0.5) circle (1.5pt) ;
	  \draw (2.55, 0.55) -- (2.55, 2.5) ;
	  \draw (2.55, 2.55) circle (1.5pt) ;
	  \draw (2.5, 2.55) -- (0.55, 2.55) ;
	  \draw [red, fill] (0.5, 2.55) circle (1.5pt) ;
	  \draw (0.5, 2.5) -- (0.5, 0.55) ;
 
	  \draw [red, fill] (1.525, 0.5) circle (1.5pt) ;
	  \draw [red, fill] (1.525, 1.525) circle (1.5pt) ;
	  \draw [red, fill] (2.55, 1.525) circle (1.5pt) ;
	  \draw [red, fill] (1.525, 2.55) circle (1.5pt) ;
	  \draw [red, fill] (0.5, 1.525) circle (1.5pt) ;
 
	  \draw [red, fill] (1.525, 0.5) -- (1.525, 1.525) -- (2.55, 1.525) --
	  (2.55, 0.5) -- (1.525, 0.5);

	  \draw [red, fill] (1.525, 1.525) -- (1.525, 2.55) --
	  (0.5, 2.55) -- (0.5, 1.525) -- (1.525,1.525);
 
	   \draw [->] (0.95,1.575) -- (0.95,1.975);
	   \draw [->] (1.475,2.1) -- (1.075,2.1);

	   \draw [->] (1.575,0.95) -- (1.975,0.95);
	   \draw [->] (2.1,1.475) -- (2.1,1.075);

	  \coordinate (1) at (2.5, 2) node at (1)
	  {$\Sd(D_{\nu,\nu'})$} ; 
	 \end{tikzpicture}
	\end{center}

  \item For $\lambda,\lambda'\in\Lambda_{L}$ ($\lambda\neq\lambda'$),
	$\mu,\mu'\in\Lambda_{B}$ ($\mu\neq\mu'$), and
	$\nu,\nu'\in\Lambda_{C}$ ($\nu\neq\nu'$), if the pairs
	$(e_{\lambda},e_{\lambda'})$, $(e_{\mu},e_{\mu'})$, and
	$(e_{\nu},e_{\nu'})$ do not share common vertices, then
	$\Sd(D_{\lambda,\lambda'})$, $\Sd(D_{\mu,\mu'})$, and
	$\Sd(D_{\nu,\nu'})$ are given by
	\begin{center}
	 \begin{tikzpicture}
	  \draw [red, fill] (0.5, 0.5) -- (2.5, 0.5) -- (2.5,2.5) --
	  (0.5,2.5) -- (0.5,0.5) ;
	  \draw [red, fill] (0.5, 0.5) circle (1.5pt) ;
	  \draw [red, fill] (2.55, 0.5) circle (1.5pt) ;
	  \draw [red, fill] (2.55, 2.55) circle (1.5pt) ;
	  \draw [red, fill] (0.5, 2.55) circle (1.5pt) ;
 
	  \draw [red, fill] (1.525, 0.5) circle (1.5pt) ;
	  \draw [red, fill] (2.55, 1.525) circle (1.5pt) ;
	  \draw [red, fill] (1.525, 2.55) circle (1.5pt) ;
	  \draw [red, fill] (0.5, 1.525) circle (1.5pt) ;
 
	  \coordinate (1) at (1.525, 1.8) node at (1)
	  {$\Sd(D_{\lambda,\lambda'})$} ; 
	 \end{tikzpicture}
	 \hspace*{30pt}
	  \begin{tikzpicture}
	  \draw [red, fill] (9.5, 0.5) circle (1.5pt) ;
	  \draw (9.55, 0.5) -- (11.5, 0.5);
	  \draw (11.55, 0.5) circle (1.5pt) ;
	  \draw [dashed] (11.55, 0.55) -- (11.55, 2.5) ;
	  \draw (11.55, 2.55) circle (1.5pt) ;
	  \draw [dashed] (11.5, 2.55) -- (9.55, 2.55) ;
	  \draw (9.5, 2.55) circle (1.5pt) ;
	  \draw (9.5, 2.5) -- (9.5, 0.55) ;
 
	  \draw [red, fill] (10.525, 0.5) circle (1.5pt) ;
	  \draw [red, fill] (10.525, 1.525) circle (1.5pt) ;
	  \draw [red, fill] (9.5, 1.525) circle (1.5pt) ;
 
	  \draw [red, fill] (9.5,0.5) -- (10.525, 0.5) --
	   (10.525, 1.525) -- (9.5, 1.525) -- (9.5, 0.5);

	   \draw [->] (9.7,1.5) -- (9.7,0.9);
	   \draw [->] (10.5,1.5) -- (9.8,0.8);
	   \draw [->] (10.5,0.7) -- (9.9,0.7);
 
	  \coordinate (1) at (10.525, 1.8) node at (1)
	   {$\Sd(D_{\mu,\mu'})$} ;  
	 \end{tikzpicture}
	 \hspace*{30pt}
	 \begin{tikzpicture}
	  \draw [red, fill] (0.5, 0.5) -- (2.5, 0.5) -- (2.5,2.5) --
	  (0.5,2.5) -- (0.5,0.5) ;
	  \draw [red, fill] (0.5, 0.5) circle (1.5pt) ;
	  \draw [red, fill] (2.55, 0.5) circle (1.5pt) ;
	  \draw [red, fill] (2.55, 2.55) circle (1.5pt) ;
	  \draw [red, fill] (0.5, 2.55) circle (1.5pt) ;
 
	  \draw [red, fill] (1.525, 0.5) circle (1.5pt) ;
	  \draw [red, fill] (2.55, 1.525) circle (1.5pt) ;
	  \draw [red, fill] (1.525, 2.55) circle (1.5pt) ;
	  \draw [red, fill] (0.5, 1.525) circle (1.5pt) ;
 
	  \coordinate (1) at (1.525, 1.8) node at (1)
	  {$\Sd(D_{\nu,\nu'})$} ; 
	 \end{tikzpicture}
	\end{center}
 \end{enumerate}
\end{proposition}

Obviously $\Sd(D_{\mu,+})$ and $\Sd(D_{\mu,-})$ for $\mu\in\Lambda_{B}$
can be collapsed to $(-1,0)$ and $(0,-1)$, respectively.
There are two dimensional cells whose $\Sd$ can be collapsed to
boundaries.  
They are indicated by black arrows in the above figure. 
This is because the boundaries of these $2$-dimensional parts are glued
to each other and we can define deformation retractions which moves the
boundaries at the same speed.

\begin{corollary}
 \label{simplified_model_for_two_points}
 Define a functor
 \[
  \Sd^{\mathrm{red}}\left(D^{\Conf_2(X^{\circ})}\right) \longrightarrow
 \category{Top} 
 \]
 by modifying the $\Sd\left(D^{\Conf_2(X^{\circ})}\right)$ by the
 following replacement of its values:
 \begin{itemize}
  \item for $\mu\in\Lambda_{B}$, replace 
	$\Sd(D_{\mu,+})$ by $\{(-1,0)\}$,
	and
	$\Sd(D_{\mu,-})$ by $\{(0,-1)\}$,

  \item for $\nu\in\Lambda_{C}$, replace
	$\Sd(D_{\nu,+})$ by $\{(-1,1)\}$,
	and
	$\Sd(D_{\nu,-})$ by $\{(1,-1)\}$,

  \item for $\lambda\in\Lambda_{L}$ and $\nu\in\Lambda_{C}$ with
	$\overline{e_{\lambda}}\cap\overline{e_{\nu}}=\{v\}$ for a
	vertex $v$, replace
	$\Sd(D_{\lambda,\nu})$ by
	$L_{\lambda,\nu}=\{0\}\times D_{\nu}\cup D_{\lambda}\times\{-1\}$ 
	and $\Sd(D_{\nu,\lambda})$ by
	$L_{\nu,\lambda}=D_{\nu}\times\{0\}\cup \{-1\}\times D_{\lambda}$,

  \item for $\mu\in\Lambda_{B}$ and $\nu\in\Lambda_{C}$ with
	$\overline{e_{\mu}}\cap\overline{e_{\nu}}=\{v\}$ for a vertex
	$v$, replace 
	$\Sd(D_{\mu,\nu})$ by
	$L_{\mu,\nu}=\set{(s,-s)}{-1\le s\le 0}\cup \set{(0,t)}{-1\le t\le 0}$
	and
	$\Sd(D_{\nu,\mu})$ by
	$L_{\nu,\mu}=\set{(s,-s)}{0\le s\le 1}\cup \set{(t,0)}{-1\le t\le 0}$,

  \item for $\lambda\in\Lambda_{L}$ and $\mu\in\Lambda_{B}$ with
	$\overline{e_{\lambda}}\cap\overline{e_{\mu}}=\emptyset$, replace
	$\Sd(D_{\lambda,\mu})$ by $D_{\lambda}\times\{-1\}$ and
	$\Sd(D_{\mu,\lambda})$ by $\{-1\}\times D_{\lambda}$,

  \item for $\mu\in\Lambda_{B}$ and $\nu\in\Lambda_{C}$ with
	$\overline{e_{\mu}}\cap\overline{e_{\nu}}=\emptyset$, replace
	$\Sd(D_{\mu,\nu})$ by $\{-1\}\times D_{\nu}$, 
	and
	$\Sd(D_{\nu,\mu})$ by $D_{\nu}\times\{-1\}$,

  \item for $\nu,\nu'\in\Lambda_{C}$ with $\nu\neq \nu'$ and
	$\overline{e_{\nu}}\cap\overline{e_{\nu'}}=\{v\}$ for a vertex
	$v$, replace
	$\Sd(D_{\nu,\nu'})$ by $D_{\nu}\times\{1\}\cup \{1\}\times D_{\nu'}$, 

  \item for $\nu,\nu'\in\Lambda_{C}$ with $\nu\neq \nu'$ and
	$\overline{e_{\nu}}\cap\overline{e_{\nu'}}=\{v,w\}$ for vertices
	$v$ and $w$, replace
	$\Sd(D_{\nu,\nu'})$ by
	$L_{\nu,\nu'}=\set{(s,-s)}{-1\le s\le 1}$, and

  \item for $\mu,\mu'\in\Lambda_{B}$ with $\mu\neq\mu'$ and
	$\overline{e_{\mu}}\cap\overline{e_{\mu'}}=\emptyset$, replace
	$\Sd(D_{\mu,\mu'})$ by $\{(-1,-1)\}$.
 \end{itemize}
 And define
 \[
 C_2^{\comp,r}(X^{\circ}) = 
 \colim_{C(\pi_{2,X^{\circ}}^{\comp})}
 \Sd^{\mathrm{red}}\left(D^{\Conf_2(X^{\circ})}\right) 
 \]
 Then $C_2^{\comp,r}(X^{\circ})$ is a strong deformation retract of
 $C_2^{\comp}(X^{\circ})$. 
\end{corollary}

This space $C_2^{\comp,r}(X^{\circ})$ is our combinatorial
(cell-complex) model for $\Conf_2(X)$.


\section{Sample Applications}
\label{application_of_model_for_graph}

In this final section, we present a couple of applications of our
acyclic category model for the configuration space of $1$-dimensional
cellular stratified spaces.

\subsection{The Homotopy Dimension}
\label{homotopy_dimension}

Given a finite graph $X$, we have
a $\Sigma_k$-equivariant homotopy equivalence
$\Conf_k(X)\simeq_{\Sigma_k}C^{\comp}_k(X^{\circ})$ by Corollary
\ref{2nd_model_for_graph}, and hence a homotopy equivalence 
$\Conf_k(X)/\Sigma_k \simeq C^{\comp}_k(X^{\circ})/\Sigma_k$.
Let us consider the dimension of $C^{\comp}_k(X^{\circ})$, which is the
classifying space of the finite acyclic category
$C(\pi_{k,X^{\circ}}^{\comp})$. In general, it is easy to count the
dimension of the classifying space of a finite acyclic category.

\begin{lemma}
 \label{dimension_of_acyclic_category}
 Let $C$ be a finite acyclic category. Then
 \[
  \dim BC = \max\set{k}{\overline{N}_k(C)\neq\emptyset} = \dim BP(C).
 \]
\end{lemma}

\begin{proof}
 Let $d=\max\set{k}{\overline{N}_k(C)\neq\emptyset}$. Then
 the cell decomposition of $BC$ is given by
 \[
 BC =
 \quotient{\coprod_{k=0}^{d}\overline{N}_k(C)\times\Delta^k}{\sim}
 = \coprod_{k=0}^{d}\overline{N}_k(C)\times\Int(\Delta^k). 
 \]
 Interior points in cells of dimension $d$ can not be equivalent 
 to points in lower dimensional cells. Thus $\dim BC=d$. By the
 definition of the associated poset $P(C)$, $d$ is also the rank of this
 poset. 
\end{proof}

 Lemma \ref{dimension_of_acyclic_category} says that, for a finite
 totally normal cellular stratified space $X$, $\dim BC(X)$ is the
 length of maximal chains in the face poset $P(X)$.

\begin{theorem}
 \label{dimension_of_model}
 For a finite connected graph $X$, we have
 \[
  \dim C_k^{\comp}(X) \le \min\{k,v(X)\},
 \]
 where $v(X)$ is the number of $0$-cells in $X$.
\end{theorem}

\begin{proof}
 Since $\Conf_k(X^{\circ})$ is a $k$-dimensional cellular stratified
 space, $\dim C_{k}^{\comp}(X)\le k$.
 It remains to prove that the length of maximal chains is at most
 $v(X)$. By the symmetry under the action of $\Sigma_k$, it suffices to
 consider subdivisions of cells in $X^k$ of the form
 (\ref{standard_form}).  

 First of all, by the definition of the subdivision, any cell of
 dimension less than $k$ in the stratification $\pi_{k,X}^{\comp}$ is a
 face of a $k$-dimensional face. 
 Thus it is enough to count how many times we can
 take a boundary face of a $k$-dimensional face in $\Conf_k(X)$ under the
 stratification $\pi_{k,X}^{\comp}$.

 Any $k$-dimensional cell in the stratification $\pi_{k,X}^{\comp}$ is
 of the form 
 \begin{equation}
 (e^1_{\lambda_1}\times\cdots\times e^1_{\lambda_{s}}\times
 e^{m_1}_{\mu_1,\tau_1}\times \cdots \times
 e^{m_t}_{\mu_t,\tau_t})\sigma    
 \label{general_top_cell}
 \end{equation}
 with $s+m_1+\cdots+m_t=k$ for some $\sigma\in\Sigma_k$, where
 $e^1_{\lambda_i}$ is a $1$-cell in $X$ and $e^{m_j}_{\mu_j,\tau_j}$ is 
 an $m_j$-dimensional cell in the braid stratification of
 $(e^1_{\mu_j})^{m_j}$ corresponding to a permutation
 $\tau_j\in\Pi_{m_j}$ under the correspondence in Lemma
 \ref{braid_arrangement_and_partitions}.  

 An $(m_{i}-1)$-dimensional face 
 of $e^{m_i}_{\mu_i,\tau_i}$ is of the form
 $(e^0_{\lambda}\times e^{m_{i}-1}_{\mu_i,\tau_i'})\sigma'$ for a vertex
 $e^0_{\lambda}$ in $e_{\mu_i}^1$. By taking the boundary of
 $e^1_{\lambda_i}$, it is replaced by one of its vertices. Thus we
 increase the number of $0$-cells by one if we take a boundary face of
 codimension $1$.

 Iterate the process of taking codimension $1$ faces starting from one
 of the highest dimensional cells (\ref{general_top_cell}) in
 $\Conf_k(X)$. When the same vertices appear twice as product
 factors, the boundary face cannot belong to $\Conf_k(X)$ and the game
 is over. Thus the inequality is proved.
\end{proof}

\begin{corollary}
 Let $X$ be a finite connected graph and $n$ be the number of essential
 vertices. Then we have 
 \[
  \hodim \Conf_k(X)/\Sigma_k \le \min\{n,k\}.
 \]
\end{corollary}

\begin{remark}
 The above Corollary is not new. It has been already proved by Ghrist in
 \cite{math.GT/9905023}. We included a proof of this fact in order to 
 show an optimality of our model.
\end{remark}
\subsection{Graph Braid Groups}
\label{graph_braid_group}

Farley and Sabalka \cite{math.GR/0410539,0907.2730} used Abrams' model
and discrete Morse theory to find presentations of graph braid
groups.

\begin{definition}
 For a topological space $X$, the fundamental groupoids
 $\pi_1(\Conf_n(X)/\Sigma_n)$ and $\pi_1(\Conf_n(X))$ are called the
 \emph{braid groupoid} and the \emph{pure braid groupoid} of $n$ strands
 in $X$.

 Fix a base point $\bm{x}\in \Conf_n(X)$. Then the fundamental groups
 based on $\bm{x}$ are denoted by
 \begin{eqnarray*}
  \Br_n(X) & = & \pi_1(\Conf_n(X)/\Sigma_n)([\bm{x}], [\bm{x}]) =
  \pi_1(\Conf_n(X)/\Sigma_n,[\bm{x}]) \\
  \PBr_n(X) & = & \pi_1(\Conf_n(X))(\bm{x},\bm{x}) =
  \pi_1(\Conf_n(X),\bm{x}),
 \end{eqnarray*}
 and called the \emph{braid group} and the \emph{pure braid group} of
 $n$ strands in $X$, respectively.    
 When $X$ is a $1$-dimensional CW complex, they are called the $n$-th
 \emph{graph braid group} and the \emph{pure graph braid group} of $X$,
 respectively. 
\end{definition}

As a sample application of our model introduced in
\S\ref{model_for_graph}, we give presentations of these 
groups for graphs with at most two essential vertices in this section.

Let us begin with the case of a graph with a single vertex.

\begin{theorem}[Theorem \ref{Mukouyama1}]
 \label{graph_braid_group_1_vertex}
 Let $W_{k,\ell}$ be the finite graph with a single vertex $v$, $k$
 branches, and $\ell$ loops with leaves removed. (See Definition
 \ref{graph_terminologies} for our terminology and Theorem
 \ref{Mukouyama1} for a figure of this graph.) 

 Then the fundamental groups of ordered and unordered configuration
 spaces of two points in $W_{k,\ell}$ are given by
 \begin{eqnarray*}
  \pi_1\left(\Conf_2\left(W_{k,\ell}\right)\right) & \cong &
   F_{2n_{k,\ell}+1} \\ 
  \pi_1\left(\Conf_2\left(W_{k,\ell}\right)/\Sigma_2\right) & \cong &
   F_{n_{k,\ell}+1},   
 \end{eqnarray*}
 where $n_{k,\ell}=\frac{1}{2}(k+\ell)(k+3\ell-3)$ and $F_n$ denotes the
 free group of rank $n$. 
\end{theorem}

\begin{proof}
 Since there is only one $0$-cell in $X$, both $C^{\comp}_2(W_{k,\ell})$
 and $C^{\comp}_{2}(W_{k,\ell})/\Sigma_2$ are
 $1$-dimensional cell complexes by Theorem \ref{dimension_of_model}. Thus
 the fundamental groups $\PBr_{2}(W_{k,\ell})$ and $\Br_2(W_{k,\ell})$
 are free groups. Their ranks as free groups coincide with the ranks of
 $H_1(C^{\comp}_2(W_{k,\ell}))$ and
 $H_1(C^{\comp}_2(W_{k,\ell})/\Sigma_2)$ as free Abelian groups.

 Let us consider $H_1(C^{\comp}_2(W_{k,\ell}))$. Since $W_{k,\ell}$ is
 connected, it suffices to compute the Euler characteristic
 \[
  \rank H_1(C^{\comp}_2(W_{k,\ell})) = 1-\chi(C^{\comp}_2(W_{k,\ell})). 
 \]
 Let $v$ be the vertex of $W_{k,\ell}$ and write
 \[
  W_{k,\ell} = \{v\}\cup\left(\bigcup_{\mu\in\Lambda_{B}}e_{\mu}\right)
 \cup \left(\bigcup_{\lambda\in\Lambda_L} e_{\lambda}\right).
 \]
 By Theorem \ref{simplified_model_for_two_points}, 
 $C^{\comp,r}_2(W_{k,\ell})$ is obtained as a quotient space of
 \begin{eqnarray*}
  T^{\comp,r}_2(W_{k,\ell}) & = &
   \left(\coprod_{\lambda\in\Lambda_{L}} \Sd(D_{\lambda,+})\right)
   \amalg \left(\coprod_{\lambda\in\Lambda_{L}} \Sd(D_{\lambda,-})\right) \\
  & & \amalg \left(\coprod_{\mu\in\Lambda_{B},\lambda\in\Lambda_{L}}
	      \Sd(D_{\mu,\lambda})\right) 
   \amalg \left(\coprod_{\mu\in\Lambda_{B},\lambda\in\Lambda_{L}}
	      \Sd(D_{\lambda,\mu})\right) \\
  & & \amalg \left(\coprod_{\mu,\mu'\in\Lambda_{B},\mu\neq\mu'}
	      \Sd(D_{\mu,\mu'})\right) \\
  & & \amalg \left(\coprod_{\lambda,\lambda'\in\Lambda_{L},\lambda\neq\lambda'}
	      \Sd(D_{\lambda,\lambda'})\right). 
 \end{eqnarray*}
 The numbers of subcomplexes of type $\Sd(D_{\lambda,+})$,
 $\Sd(D_{\lambda},-)$, $\Sd(D_{\mu,\lambda})$,
 $\Sd(D_{\lambda,\mu})$, $\Sd(D_{\mu,\mu'})$, 
 and $\Sd(D_{\lambda,\lambda'})$ are $\ell$, $\ell$, $k\ell$, $k\ell$,
 $k(k-1)$, and $\ell(\ell-1)$, respectively.  

 These subcomplexes are glued together along their boundary vertices.
 Thus the number of edges in $C^{\comp,r}(W_{k,\ell})$ is the same as
 that of $T^{\comp,r}(W_{k,\ell})$, which is give by
 \[
 2\ell + 2\ell +  3k\ell + 3k\ell + 2k(k-1) + 4\ell(\ell-1) = 6k\ell +
 2k^2-2k + 4\ell^2.
 \]
 The vertices in $C^{\comp,r}(W_{k,\ell})$ are in one-to-one
 correspondence with cells in the braid stratification on
 $\Conf_2(W_{k,\ell})$ with cells of the form $e_{\mu,+}$ and
 $e_{\mu,-}$ removed. 
 Besides the $2$-cells described above, $1$-cells are
 $\{v\}\times e_{\mu}$, $\{v\}\times e_{\lambda}$, $e_{\mu}\times\{v\}$,
 and $e_{\lambda}\times\{v\}$ for $\mu\in\Lambda_{B}$ and
 $\lambda\in\Lambda_{L}$. 
 Thus the number of vertices in
 $C^{\comp,r}(W_{k,\ell})$ is given by
 \[
  2k+2\ell+ 2\ell+2k\ell+k(k-1)+\ell(\ell-1) = 2k\ell + k^2+k +
 \ell^2+3\ell. 
 \]
 Thus the Euler characteristic is
 \begin{eqnarray*}
  \chi(C^{\comp,r}_2(W_{k,\ell})) & = & 2k\ell + k^2+k + \ell^2+3\ell
   - (6k\ell + 2k^2-2k + 4\ell^2) \\
  & = & -4k\ell -k^2-3\ell^2 +3k+3\ell \\
  & = & -(k+\ell)(k+3\ell-3).
 \end{eqnarray*}
 And the rank of $\PBr(W_{k,\ell})$ is given by
 \[
  1-\chi(C^{\comp,r}_2(W_{k,\ell})) = 1+(k+\ell)(k+3\ell-3).
 \]

 By identifying cells under the action of $\Sigma_2$, we see that the
 rank of $\Br(W_{k,\ell})$ is given by $1+\frac{1}{2}(k+\ell)(k+3\ell-3)$.
\end{proof}

We need to deal with $2$-cells to determine the graph braid groups of
graphs with two vertices. We use the following elementary fact.

\begin{proposition}
 \label{parallel_connection}
 Let $X$ be a CW complex. Suppose that it contains connected
 subcomplexes $A$ and $B$ and that there are regular $1$-cells
 $\varphi_{i}: [-1,1]\to \overline{e^1_{i}}$ ($i=1,\ldots,n$) with
 $\varphi_{i}(-1)\in A$, $\varphi_{i}(1)\in B$, and
 $X = A\cup B \cup \bigcup_{i=1}^n e^1_{i}$.
 Then we have a homotopy equivalence
 \[
 X \simeq A\vee B\vee \bigvee_{i=1}^{n-1}S^1. 
 \] 
\end{proposition}

\begin{proof}
 Since $A$ and $B$ are path-connected, we may move the end points of 
 $\overline{e^1_{i}}$ freely in $A$ and $B$. Move $\varphi_{i}(-1)$ for
 $1\le i\le n-1$ to $\varphi_n(-1)$ and then move those $n-1$ points to
 $\varphi_n(1)$ along $\overline{e^1_n}$. Then we obtain a CW complex
 consisting of $A$ and $\displaystyle B\vee \bigvee^{n-1}_{i=1}S^1$
 connected by an edge $\overline{e^1_{n}}$. By shrinking
 $\overline{e^1_n}$, we obtain 
 $\displaystyle A\vee B\vee \bigvee_{i=1}^{n-1}S^1$.
\end{proof}

\begin{theorem}[Theorem \ref{Mukouyama2}]
 \label{graph_braid_group_2_vertex}
 Let ${}_{x}B^{k,\ell}_{p,q}$ be the finite graph
 obtained by gluing the essential vertices of $W_{k,\ell}$ and $W_{p,q}$ 
 by $x$ parallel edges. (See the figure in Theorem \ref{Mukouyama2}.) 

 Then the fundamental groups of ordered and unordered configuration
 spaces of two points in ${}_xB^{k,\ell}_{p,q}$ are given by
 \begin{eqnarray*}
  \pi_1\left(\Conf_2\left({}_xB^{k,\ell}_{p,q}\right)\right) & \cong &
   A_{\ell,q}\ast 
   A_{q,\ell}\ast F_{2{}_{x}m^{k,\ell}_{p,q}-1} \\
  \pi_1\left(\Conf_2\left({}_xB^{k,\ell}_{p,q}\right)/\Sigma_2\right)
   & \cong & A_{\ell,q}\ast F_{{}_{x}m^{k,\ell}_{p,q}}, 
 \end{eqnarray*}
 where
 \[
  A_{\ell,q}=\langle a_1,\ldots,a_{\ell},b_1,\ldots,b_{q} \mid
 [a_j,b_t]\ (1\le j\le \ell,1\le t\le q) \rangle 
 \]
 and
 \[
  {}_{x}m^{k,\ell}_{p,q} = n_{k,\ell}+n_{p,q} + x(k+\ell+p+q) +
 \frac{x(x-1)}{2}. 
 \] 
\end{theorem}

\begin{proof}
 Let $v_1,v_2$ be vertices of ${}_xB^{k,\ell}_{p,q}$ and write
 \[
  {}_xB^{k,\ell}_{p,q} = \{v_1,v_2\} \cup
 \left(\bigcup_{\lambda\in\Lambda_{L_1}}e_{\lambda}\right)\cup
 \left(\bigcup_{\lambda\in\Lambda_{L_2}}e_{\lambda}\right)
 \left(\bigcup_{\mu\in\Lambda_{B_1}}e_{\mu}\right) \cup
 \left(\bigcup_{\mu\in\Lambda_{B_2}}e_{\mu}\right) \cup
 \left(\bigcup_{\nu\in\Lambda_{C}}e_{\nu}\right),
 \]
 where $\set{e_{\lambda}}{\lambda\in \Lambda_{L_i}}$ is the set of
 loops adjacent to the vertex $v_i$ for $i=1,2$ and
 $\set{e_{\mu}}{\mu\in\Lambda_{B_i}}$ is the set of branches adjacent to
 $v_i$ for $i=1,2$.

 There are obvious embeddings
 \begin{eqnarray*}
  W_{k,\ell} & \hookrightarrow {}_{x}B^{k,\ell}_{p,q} \\
  W_{p,q} & \hookrightarrow {}_{x}B^{k,\ell}_{p,q},
 \end{eqnarray*}
 which induce embeddings of cell complexes
 \begin{eqnarray*}
  C^{\comp,r}(W_{k,\ell}) & \hookrightarrow
   C^{\comp,r}({}_{x}B^{k,\ell}_{p,q}) \\ 
  C^{\comp,r}(W_{p,q}) & \hookrightarrow
   C^{\comp,r}({}_{x}B^{k,\ell}_{p,q}). 
 \end{eqnarray*}

 Define 
 \[
  T_{q,\ell} = [0,q]\times[0,\ell]/_{\sim_{q,\ell}}
 \]
 where the relation $\sim_{q,\ell}$ is defined by
 $(s,t)\simeq_{q,\ell} (s',t')$ if and only if $s,s'\in\{0,1,\ldots,q\}$
 and $t=t'$ or $s=s'$ and $t,t'\in\{0,1,\ldots,\ell\}$. The $T_{q,\ell}$
 and $T_{\ell,q}$ can be embedded in $C^{\comp,r}({}_xB^{k,\ell}_{p,q})$
 as subcomplexes consisting of $\Sd(D_{\lambda_1,\lambda_2})$ for
 $\lambda_i\in\Lambda_{L_i}$ ($i=1,2$).

 By Corollary \ref{simplified_model_for_two_points},
 $C^{\comp,r}_2({}_xB^{k,\ell}_{p,q})$ is obtained by sewing
 $T_{\ell,q}$, $T_{q,\ell}$, $C^{\comp,r}_2(W_{k,\ell})$, and
 $C^{\comp,r}_2(W_{p,q})$ by using $L_{\lambda,\nu}$, $L_{nu,\lambda}$,
 $L_{\mu,\nu}$, $L_{\nu,\mu}$, and $L_{\nu,\nu'}$ as strings, where
 $\lambda\in\Lambda_{L}$, $\mu\in\Lambda_{B}$, $\nu,\nu'\in\Lambda_{C}$
 ($\nu\neq\nu'$). Note that all these strings are homeomorphic to a
 closed interval $[-1,1]$. 
 The following figure shows a rough idea of connections, in which
 $C^{\comp,r}_2(W_{k,\ell})$ is abbreviated by $C_{k,\ell}$ and indices
 runs over all $\lambda_1\in\Lambda_{L_1}$, $\lambda_2\in\Lambda_{L_2}$,
 $\mu_1\in\Lambda_{B_1}$, $\mu_2\in\Lambda_{B_2}$, and
 $\nu,\nu'\in\Lambda_{C}$ ($\nu\neq\nu'$).
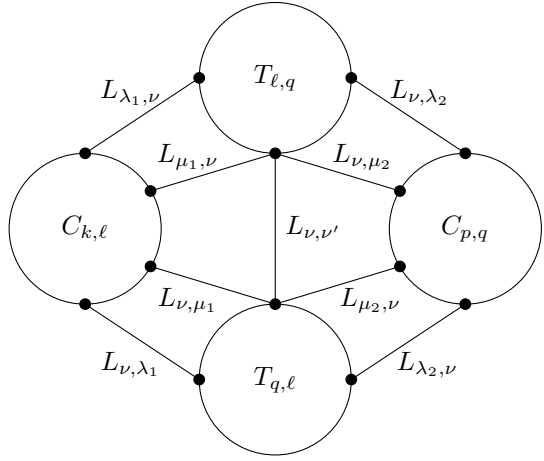
 \begin{center}
  \begin{tikzpicture}
   \draw (-2.5,0) circle (1cm);
   \draw (-2.5,0) node {$C_{k,\ell}$};
   \draw [fill] (-2.5,1) circle (2pt);
   \draw [fill] (-2.5,-1) circle (2pt);
   \draw [fill] (-1.64,0.5) circle (2pt);
   \draw [fill] (-1.64,-0.5) circle (2pt);

   \draw (2.5,0) circle (1cm);
   \draw (2.5,0) node {$C_{p,q}$};
   \draw [fill] (2.5,1) circle (2pt);
   \draw [fill] (2.5,-1) circle (2pt);
   \draw [fill] (1.64,0.5) circle (2pt);
   \draw [fill] (1.64,-0.5) circle (2pt);

   \draw (0,-2) circle (1cm);
   \draw (0,-2) node {$T_{q,\ell}$};
   \draw [fill] (0,-1) circle (2pt);
   \draw [fill] (-1,-2) circle (2pt);
   \draw [fill] (1,-2) circle (2pt);

   \draw (0,2) circle (1cm);
   \draw (0,2) node {$T_{\ell,q}$};
   \draw [fill] (0,1) circle (2pt);
   \draw [fill] (-1,2) circle (2pt);
   \draw [fill] (1,2) circle (2pt);

   \draw (-2.5,1) -- (-1,2);
   \draw (-1.9,1.8) node {$L_{\lambda_1,\nu}$};
   \draw (-1.64,0.5) -- (0,1);
   \draw (-1.15,0.95) node {$L_{\mu_1,\nu}$};

   \draw (-2.5,-1) -- (-1,-2);
   \draw (-1.9,-1.8) node {$L_{\nu,\lambda_1}$};
   \draw (-1.64,-0.5) -- (0,-1);
   \draw (-1.15,-0.95) node {$L_{\nu,\mu_1}$};

   \draw (2.5,1) -- (1,2);
   \draw (1.9,1.8) node {$L_{\nu,\lambda_2}$};
   \draw (1.64,0.5) -- (0,1);
   \draw (1.15,0.95) node {$L_{\nu,\mu_2}$};

   \draw (2.5,-1) -- (1,-2);
   \draw (2,-1.8) node {$L_{\lambda_2,\nu}$};
   \draw (1.64,-0.5) -- (0,-1);
   \draw (1.25,-0.95) node {$L_{\mu_2,\nu}$};

   \draw (0,1) -- (0,-1);
   \draw (0.5,0) node {$L_{\nu,\nu'}$};   
  \end{tikzpicture}
 \end{center}

 We first apply Proposition \ref{parallel_connection} to the union of
 $C_{k,\ell}$, $T_{\ell,k}$, $L_{\lambda_1,\nu}$'s, and
 $L_{\mu_1,\nu}$'s. There are  
 $x\ell$ $L_{\lambda_1,\nu}$'s and $xk$ $L_{\mu_1,\nu}$'s. Thus we
 obtain $T_{\ell,k}\vee C_{k,\ell}\vee \bigvee^{x(k+\ell)-1} S^1$. Apply
 Proposition \ref{parallel_connection} to the union of 
 $T_{\ell,k}\vee C_{k,\ell}\vee \bigvee^{x(k+\ell)-1} S^1$, $C_{p,q}$, 
 $L_{\nu,\lambda_2}$'s, and $L_{\nu,\mu_2}$'s and obtain
 $T_{\ell,k}\vee C_{k,\ell}\vee C_{p,q}\vee \bigvee^{x(k+\ell+p+q)-2} S^1$.
 Edges connecting this complex to $T_{q,\ell}$ are
 $L_{\nu,\lambda_1}$'s, $L_{\nu,\mu_1}$'s, $L_{\lambda_2,\nu}$'s,
 $L_{\mu_2,\nu}$'s, and $L_{\nu,\nu'}$'s. Thus there are
 $x(k+\ell+p+q)+x(x-1)$ edges. By Proposition \ref{parallel_connection}
 again, we have a homotopy equivalence
 \begin{eqnarray*}
  C^{\comp,r}_2({}_xB^{k,\ell}_{p,q}) & \simeq & T_{\ell,q} \vee
   T_{q,\ell} \vee C^{\comp,r}_2(W_{k,\ell}) \vee
   C^{\comp,r}_2(W_{p,q}) \\
 & & \vee \left(\bigvee^{(k+p+\ell+q)x-2} S^1\right)
  \vee \left(\bigvee^{(k+p+\ell+q)x + x(x-1)-1} S^1\right) \\
  & \simeq & T_{\ell,q}\vee T_{q,\ell} \vee 
   \left(\bigvee^{2n_{k,\ell}+1} S^1\right) \vee
   \left(\bigvee^{2n_{p,q}+1} S^1\right) \\
   & & \vee \left(\bigvee^{2(k+p+\ell+q)x+x(x-1)-3} S^1\right) \\
  & \simeq & T_{\ell,q}\vee T_{q,\ell}
   \vee \left(\bigvee^{2n_{k,\ell}+2n_{p,q}+2(k+\ell+p+q)x+x(x-1)-1} S^1\right)
 \end{eqnarray*}
 Since $\pi_1(T_{\ell,q}) \cong A_{\ell,q}$, van Kampen Theorem tells us
 that
 \[
  \PBr({}_xB^{k,\ell}_{p,q}) \cong A_{\ell,q}\ast A_{q,\ell} \ast
 F_{2{}_xm^{k,\ell}_{p,q}-1}, 
 \]
 where 
 \[
 {}_{x}m^{k,\ell}_{p,q} = n_{k,\ell} + n_{p,q} +
 x(k+\ell+p+q)+\frac{x(x-1)}{2}. 
 \]
 
 The action of the generator of $\Sigma_2$ on
 $C^{\comp,r}_2({}_xB^{k,\ell}_{p,q})$  
 is given as follows:
 \begin{itemize}
  \item $T_{\ell,q}$ and $T_{q,\ell}$ are identified,
  \item $C_{k,\ell}$  is mapped to itself,
  \item $C_{p,q}$ is mapped to itself,
  \item $L_{\lambda_i,\nu}$ is identified with $L_{\nu,\lambda_i}$ for
	$i=1,2$, 
  \item $L_{\mu_i,\nu}$ is identified with $L_{\nu,\mu_i}$ for $i=1,2$,
	and 
  \item $L_{\nu,\nu'}$ and $L_{\nu',\nu}$ are identified in such a way
	that the end points of $L_{\nu,\nu'}$ are identified.
 \end{itemize}
 Thus we obtain
 \begin{eqnarray*}
  C^{\comp,r}_2({}_{x}B^{k,\ell}_{p,q})/\Sigma_2 & \simeq & T_{\ell,q}
   \vee C_{k,\ell}/\Sigma_2 \vee C_{p,q}/\Sigma_2 \\
  & & \vee \left(\bigvee^{(k+\ell)x-1} S^1\right) 
   \vee \left(\bigvee^{(p+q)x-1} S^1\right) 
   \vee \left(\bigvee^{\frac{x(x-1)}{2}} S^1\right) \\
  & \simeq &  T_{\ell,q} \vee \left(\bigvee^{n_{k,\ell}+1} S^1\right)
   \vee \left(\bigvee^{n_{p,q}+1}S^1\right)
   \vee \left(\bigvee^{(k+\ell+p+q)x-2+\frac{x(x-1)}{2}} S^1\right) \\
  & = & T_{\ell,q}\vee \left(\bigvee^{{}_xm^{k,\ell}_{p,q}} S^1\right),
 \end{eqnarray*}
 where 
 \[
  {}_xm^{k,\ell}_{p,q} = n_{k,\ell} + n_{p,q} + (k+\ell+p+q)x +
 \frac{x(x-1)}{2}.
 \]
 This completes the proof.
\end{proof}

\appendix

\section{Extending a Deformation Retraction on the Boundary}
\label{deformation_retraction}

The aim of this appendix is to prove Theorem \ref{spherical_case}. A
statement and a proof of this fact first appeared in \cite{1009.1851v5},
but the paper was split into two parts and now they are not included in
the paper.

Let us begin by recalling the definition of regular neighborhoods.

\begin{definition}
 Let $K$ be a cell complex. For $x\in K$, define
 \[
  \St(x;K) = \bigcup_{x\in\overline{e}} e.
 \]
 This is called the \emph{open star} around $x$ in $K$. For a subset
 $A \subset K$, define 
 \[
  \St(A;K) = \bigcup_{x\in A} \St(x;K).
 \]
 When $K$ is a simplicial complex and $A$ is a subcomplex,
 $\St(A;K)$ is called the \emph{regular neighborhood} of $A$ in $K$. 
\end{definition}

The regular neighborhood of a subcomplex is often defined in terms of
vertices.

\begin{lemma}
 Let $A$ be a subcomplex of a simplicial complex $K$. Then
 \[
  \St(A;K) = \bigcup_{v\in \sk_0(A)} \St(v;K).
 \]
\end{lemma}

\begin{definition}
 Let $K$ be a simplicial complex. We say a subcomplex $L$ is a
 \emph{full} subcomplex if, for any collection of vertices
 $v_0,\ldots, v_k$ in $L$  which form a simplex $\sigma$ in $K$, the
 simplex $\sigma$ belongs to $L$.  
\end{definition}

The following fact is fundamental.

\begin{lemma}
 \label{DR_of_regular_neighborhood}
 If $K$ is a simplicial complex and $A$ is a full subcomplex, then $A$
 is a strong deformation retract of the regular neighborhood
 $\St(A;K)$.  
\end{lemma}

\begin{proof}
 The retraction
 \[
  r_{A} : \St(A;K) \longrightarrow A
 \]
 is given by
 \[
 r_{A}(x) = \frac{1}{\sum_{v\in A\cap \sigma} t(v)} \sum_{v\in A\cap\sigma}
 t(v)v, 
 \]
 if $x=\sum_{v\in\sigma} t(v)v$ belongs to a simplex $\sigma$. A
 homotopy between $i\circ r_{A}$ and the identity map is given by a
 ``linear homotopy''. See Lemma 9.3 in \cite{Eilenberg-Steenrod}, for
 more details.  
\end{proof}

The following modification of this fact was first proved in
\cite{1009.1851v5}. 

\begin{lemma}
 \label{regular_neighborhood_of_full_subcomplex}
 Let $K$ be a finite simplicial complex and $K'$ a subcomplex. Given a
 full subcomplex $A$ of $K$, let $A'=A\cap K'$. Suppose we are given a
 strong deformation retraction $H$ of $\St(A';K')$ onto $A'$. Then there
 exists a deformation retraction $\widetilde{H}$ of $\St(A;K)$ onto $A$ 
 extending $H$.
\end{lemma}

\begin{proof}
 We regard $K$ as a subcomplex of a large simplex $S$. Then every point
 $x\in\|K\|$ can be expressed as a formal convex combination
 \[
  x = \sum_{v\in V(K)} a_v v
 \]
 with $\sum_{v\in V(K)}a_{v}=1$ and $a_v\ge 0$, where $V(K)$ is the
 vertex set of $K$.

 Let $H$ be a strong deformation retraction of $\St(A';K')$ onto
 $A'$. Define
 \[
  \widetilde{H}:\St(A;K) \times[0,1] \rarrow{} \St(A;K)
 \]
 by
\begin{eqnarray*}
 \widetilde{H}(x,s) & = & \frac{\alpha+(1-s)\beta}{(1-s)+s(\alpha+\gamma)}
 H'\left(\sum_{i} \frac{a_i}{\alpha+\beta} u_i' + \sum_{j}
 \frac{b_j}{\alpha+\beta}v_j', s\right) \\
 & & + \sum_{k} \frac{c_k}{(1-s)+s(\alpha+\gamma)}u_k + \sum_{\ell}
  \frac{(1-s)d_{\ell}}{(1-s)+s(\alpha+\gamma)}v_{\ell}, 
\end{eqnarray*}
 where $s\in [0,1]$, $\alpha=\sum_i a_i$, $\beta=\sum_j b_j$,
 $\gamma=\sum_k c_k$, and $x\in \St(A;K)$ has the form
 \[
  x= \sum_i a_iu_i' + \sum_{j} b_jv_j' + \sum_k c_ku_k + \sum_{\ell}
 d_{\ell}v_{\ell} 
 \]
 with $u_i'\in V(A')$, $v_j'\in V(K')\setminus V(A')$, 
 $u_k\in V(A)\setminus V(A')$, and
 $v_{\ell}\in V(K)\setminus(V(K')\cup V(A))$. 
 Let us verify that this homotopy $\widetilde{H}$ satisfies our
 requirements. 

 \begin{eqnarray*}
  \widetilde{H}(x,0) & = &
   (\alpha+\beta)\left(\sum_i\frac{a_i}{\alpha+\beta}u_i' + 
			      \sum_j \frac{b_j}{\alpha+\beta}v_j'\right)
  + \sum_k c_ku_k + \sum_{\ell} d_{\ell} v_{\ell} \\
  & = & \sum_i a_iu_i' + \sum_{j} b_jv_j' + \sum_k c_ku_k + \sum_{\ell}
 d_{\ell}v_{\ell} \\
  & = & x \\
  \widetilde{H}(x,1) & = &
   \frac{\alpha}{\alpha+\gamma}r_{A'}\left(\sum_i\frac{a_i}{\alpha+\beta}u_i'
	+ \sum_j \frac{b_j}{\alpha+\beta}v_j'\right)
  + \sum_k \frac{c_k}{\alpha+\gamma}u_k \\
  & = & \frac{\alpha}{\alpha+\gamma} \sum_{i}
   \frac{\frac{a_i}{\alpha+\beta}}{\frac{\alpha}{\alpha+\beta}}u_i' +
   \sum_{k}\frac{c_k}{\alpha+\gamma} u_k \\
  & = & \sum_i\frac{a_i}{\alpha+\gamma}u_i' +
   \sum_k\frac{c_k}{\alpha+\gamma} u_k \\
  & = & r_{A}(x).
 \end{eqnarray*}

 Furthermore, when $x\in K'$, we have $c_k=d_{\ell}=0$ and
 $x=\sum_{i}a_iu_i'+\sum_j b_jv_j'$. Since $\alpha+\beta=1$, we have
 \begin{eqnarray*}
  \widetilde{H}(x,s) & = & \frac{\alpha+(1-s)\beta}{(1-s)+s\alpha}
   H\left(\sum_i a_iu_i' + \sum_j b_jv_j', s\right) \\
  & = & \frac{1-s\beta}{1-s(1-\alpha)} H(x,s) \\
  & = & H(x,s).
 \end{eqnarray*}
 Finally when $x\in A$, we have $b_j=d_{\ell}=0$ and
 $x=\sum_{i}a_iu_i'+\sum_{k}c_ku_k$. Since $\alpha+\gamma=1$, we have
 \begin{eqnarray*}
  \widetilde{H}(x,s) & = & \alpha H\left(\sum_{i}\frac{a_i}{\alpha}
				    u_i', s\right) + \sum_{k} c_k u_k \\
  & = & \sum_{i} a_iu_i' + \sum_{k} c_k u_k \\
  & = & x.
 \end{eqnarray*}
\end{proof}

In order to apply Lemma \ref{regular_neighborhood_of_full_subcomplex} to
prove Theorem \ref{spherical_case}, the following observation is crucial.   

\begin{lemma}
 \label{regular_neighborhood_of_Sd}
 Let $K$ be a regular cell complex. For any stratified subspace $L$ of
 $K$, the image of the regular neighborhood
 $\St(\Sd(L);\Sd(\overline{L}))$ of $\Sd(L)$ in $\Sd(\overline{L})$
 under the embedding
 \[
  i_{K} : \Sd(K) \hookrightarrow K
 \]
 contains $L$.
\end{lemma}

\begin{proof}
 For a point $x \in L$, there exists a cell $e$ in $L$ with
 $x\in e$. Under the barycentric subdivision of $\overline{L}$, $e$ is
 triangulated, namely there exists a sequence 
 \[
  \bm{e} : e_0< e_1<\cdots<e_n=e
 \]
 of cells in $\overline{L}$ such that
 \[
  x \in i_{\bm{e}}(\Int\Delta^n)
 \]
 and
 \[
  v(e) \in \overline{i_{\bm{e}}(\Int\Delta^n)},
 \]
 where $v(e)$ is the vertex in $\Sd(\overline{L})$ corresponding to
 $e$. By definition of $\St$, we have 
 \[
 i_{\bm{e}}(\Int\Delta^n) \subset \St(v(e);\Sd(\overline{L})) =
 \St(i_{L}(\Sd(L))\Sd(\overline{L})) 
 \]
 and we have
 \[
 L \subset \St(i_{L}(\Sd(L));\Sd(\overline{L})).
 \]

Conversely take an element $y \in \St(i_{L}(\Sd(L));\Sd(\overline{L}))$. 
There exists a simplex $\sigma$ in $\Sd(\overline{L})$ and a point
 $a \in \Sd(L)$ with $a \in \sigma$ and $y \in \Int(\sigma)$. $a$ can
 be take to be a vertex. Thus there is a cell $e$ in $L$ with
 $a=v(e)$. By the definition of $\Sd(L)$, there exists a chain 
 \[
  \bm{e} : e_0 < \cdots < e_n
 \]
 in $L$ containing $e$ with $\sigma=i_{\bm{e}}(\Delta^n)$.

 Since
 \[
 \Int(\sigma) \subset e_n \subset L,
 \]
 we have $y\in L$. Thus we have proved
 \[
   L = \St(i_{L}(\Sd(L));\Sd(\overline{L})).
 \]

 It follows from the construction of the barycentric subdivision that
 $i_{L}(\Sd(L))$ is a full subcomplex of $\overline{L}$. 
\end{proof}

\begin{proof}[Proof of Theorem \ref{spherical_case}]
 Let $L = \Int(D^n)\cup K$. This is a stratified subspace of the regular
 cell decomposition on $D^n$. By Lemma \ref{regular_neighborhood_of_Sd},
 $L$ is a regular 
 neighborhood of $i_{L}(\Sd(L))$ in $\Sd(\overline{L})$. By
 Lemma \ref{DR_of_regular_neighborhood}, there is a standard ``linear''
 homotopy which contracts $L$ on to $i_{L}(\Sd(L))$. 

 By the construction of the homotopy, it can be taken to be an extension
 of a given homotopy on $K$, under the identification
 $i_{L}(\Sd(L)) = 0\ast i_{L}(\Sd(K))$. 
\end{proof}

\addcontentsline{toc}{section}{References}
\bibliographystyle{halpha}
\bibliography{%
bib/mathAb,%
bib/mathAl,%
bib/mathAn,%
bib/mathAr,%
bib/mathA,%
bib/mathB,%
bib/mathBa,%
bib/mathBe,%
bib/mathBo,%
bib/mathBr,%
bib/mathBu,%
bib/mathCa,%
bib/mathCh,%
bib/mathCo,%
bib/mathC,%
bib/mathDa,%
bib/mathDe,%
bib/mathDu,%
bib/mathD,%
bib/mathE,%
bib/mathFa,%
bib/mathFo,%
bib/mathFr,%
bib/mathF,%
bib/mathGa,%
bib/mathGo,%
bib/mathGr,%
bib/mathG,%
bib/mathHa,%
bib/mathHe,%
bib/mathH,%
bib/mathI,%
bib/mathJ,%
bib/mathKa,%
bib/mathKo,%
bib/mathK,%
bib/mathLa,%
bib/mathLe,%
bib/mathLu,%
bib/mathL,%
bib/mathMa,%
bib/mathMc,%
bib/mathMi,%
bib/mathM,%
bib/mathN,%
bib/mathO,%
bib/mathP,%
bib/mathQ,%
bib/mathR,%
bib/mathSa,%
bib/mathSc,%
bib/mathSe,%
bib/mathSt,%
bib/mathS,%
bib/mathT,%
bib/mathU,%
bib/mathV,%
bib/mathW,%
bib/mathX,%
bib/mathY,%
bib/mathZ,%
bib/personal}

\end{document}